\documentclass[11pt]{article}
\usepackage{amsfonts,amsmath, amssymb,latexsym}
\usepackage{multirow}
\usepackage{verbatim}
\usepackage{longtable}
\usepackage{hyperref} 
\usepackage{tocloft}
\usepackage{cancel}

\usepackage{tikz}
\usetikzlibrary{matrix,arrows,decorations.pathreplacing}

\usepackage{todonotes}

\def\al{\alpha}

\newcommand{\Gal}{\mathrm{Gal}}

\newcommand{\Avg}{\mathrm{Avg}}

\newcommand{\Mass}{\mathrm{Mass}}

\def\Z{{\mathbb Z}}

\def\PGL{{\rm PGL}}
\def\SL{{\rm SL}}
\def\SO{{\rm SO}}
\def\GL{{\rm GL}}
\def\PGL{{\rm PGL}}

\def\Stab{{\rm Stab}}

\def\Cl{{\rm Cl}}
\def\O{{\mathcal O}}
\def\P{{\mathbb P}}
\def\Disc{{\rm Disc}}

\def\ind{{\rm ind}}
\def\Aut{{\rm Aut}}

\def\gen{{\rm gen}}
\def\even{{\rm even}}
\def\odd{{\rm odd}}

\def\Vol{{\rm Vol}}

\def\R{{\mathbb R}}
\def\F{{\mathbb F}}
\def\FF{{\mathcal F}}
\def\GG{{\mathcal G}}
\def\SS{{\mathcal S}}
\def\RR{{\mathcal R}}
\def\Q{{\mathbb Q}}

\def\W{{\mathcal W}}

\def\Z{{\mathbb Z}}
\def\P{{\mathbb P}}
\def\PP{{\mathcal P}}
\def\F{{\mathbb F}}
\def\Q{{\mathbb Q}}
\def\C{{\mathbb C}}

\def\J{{\mathcal J}}
\def\LL{{\mathcal L}}
\def\cQ{{\mathcal Q}}
\def\L{{L}}
\def\Res{{\rm{Res}}}
\def\max{{\rm max}}

\def\monKal{{\rm mon}(K,\alpha)}
\def\monKal{{\mathcal A_\infty(\alpha)}}
\def\maxx{{}}

\def\Tr{{\rm{Tr}}}
\def\Y{{T}}
\def\MB{{M}}

\def\ycox{{\Y;X}}
\newcommand{\RF}[1]{\FF^{(#1)}_{V_{A,b}}}
\def\cdeltaint{{\leq\! cH^\delta}}

\newtheorem{theorem}{Theorem}[section]
\newtheorem{corollary}[theorem]{Corollary}
\newtheorem{lemma}[theorem]{Lemma}
\newtheorem{proposition}[theorem]{Proposition}
\newenvironment{proof}{\noindent {\bf Proof:}}{$\Box$ \vspace{2 ex}}

\newtheorem{thm}{Theorem}

\newtheorem{defn}[theorem]{Definition}
\newtheorem{rem}[theorem]{Remark}
\newtheorem{nota}[theorem]{Notation}
\newtheorem{cons}[theorem]{Construction}

\hypersetup{colorlinks=true,linkcolor=blue, linktocpage}

\title{The mean number of 2-torsion elements in the class groups of $n$-monogenized cubic fields}
\author{Manjul Bhargava, Jonathan Hanke, and Arul Shankar}

\begin{document}
\maketitle

\begin{abstract}
We prove that, on average, the monogenicity or $n$-monogenicity of
a cubic field has an altering effect on the behavior of the 2-torsion in its class
group.
\end{abstract}

\setcounter{tocdepth}{2}
\tableofcontents

\section{Introduction}\label{secintro}

The seminal works of Cohen--Lenstra \cite{CL} and Cohen--Martinet
\cite{CM1} provide heuristics that predict, for suitable
``good primes'' $p$, the distribution of the $p$-torsion subgroups
$\Cl_p(K)$ of the class groups of number fields $K$ of degree $d$.
To date, only two cases of these conjectures have been proven; they
concern the mean sizes of the 3-torsion subgroups of the class groups
of quadratic fields, and the 2-torsion subgroups of the class groups
of cubic fields:

\begin{thm}[Davenport--Heilbronn {\cite[Theorem 3]{DH}}] \label{thquadfields}
Let $K$ run through all isomorphism classes of quadratic fields
ordered by discriminant. Then:
  \begin{itemize}
  \item[{\rm (a)}] The average size of $\Cl_3(K)$ over real
    quadratic fields $K$ is $4/3;$
  \item[{\rm (b)}] The average size of $\Cl_3(K)$ over complex quadratic
    fields $K$ is $2$.
  \end{itemize}
\end{thm}

\begin{thm}[{\cite[Theorem 5]{dodqf}}] \label{thallcubicfields}
Let $K$ run through all isomorphism classes of cubic fields ordered by
discriminant. Then:
  \begin{itemize}
  \item[{\rm (a)}] The average size of $\Cl_2(K)$ over totally real
    cubic fields $K$ is $5/4;$
  \item[{\rm (b)}] The average size of $\Cl_2(K)$ over complex cubic
    fields $K$ is $3/2$.
  \end{itemize}
\end{thm}

These two theorems have had a number of important applications.  For
example, they imply that, when ordered by discriminant, a positive
proportion of quadratic fields have class number indivisible by 3, and
a positive proportion of cubic fields have class number indivisible by
2 -- among numerous other consequences relating to the asymptotics and
ranks of elliptic curves, the density of discriminants of number
fields of given degree, the existence of units in number fields with
given signatures, rational points on surfaces, the statistics of
Galois representations and modular forms of weight one, and more (see,
e.g., \cite{FNT}, \cite{Vatsal}, \cite{DK}, \cite{BhSh}, \cite{dodqf},
\cite{BV}, \cite{Browning}, \cite{Ellenberg}, \cite{BG}, \cite{AK}).

The averages occurring in Theorems~\ref{thquadfields} and
\ref{thallcubicfields} are remarkably robust.  In \cite[Corollary
  4]{BV1} and \cite[Theorem 1]{BV}, it was shown that the averages in
these two theorems remain unchanged even when one ranges over
quadratic and cubic fields, respectively, that satisfy any specified
set of local splitting conditions at finitely many primes, or even
suitable sets of local conditions at {\it infinitely} many primes.

A natural next question that arises is: how stable are the averages in
Theorems~\ref{thquadfields} and \ref{thallcubicfields} if more global
conditions are imposed on the fields being considered? One natural
such global condition on a field is {\it monogenicity}, i.e., that the ring of integers 
in the field is generated by one element.  While monogenicity is a
global condition that automatically holds for all quadratic fields, it
is a nontrivial condition for cubic fields.  
The purpose of this paper
is to show, surprisingly (at least to the authors!), that monogenicity
does have a nontrivial effect on the behavior of class groups of cubic
fields and, in particular, the condition of monogenicity does change
the averages occurring in Theorem~\ref{thallcubicfields}.

\subsection
{The stability of averages in Theorem~\ref{thallcubicfields} when ordering cubic fields by height}

Recall that a number field $K$ is called {\bf monogenic} if its ring
$\O_K$ of integers is generated by one element as a $\Z$-algebra, i.e.,
$\O_K=\Z[\alpha]$ for some element $\alpha\in\O_K$; such an $\alpha$
is then called a {\bf monogenizer} of $K$ or of $\O_K$.
The asymptotic number of monogenic cubic fields of bounded
discriminant is not known, and it is therefore more convenient to
order these fields by the heights of their monic defining polynomials.

A pair $(K,\al)$ is called a {\bf monogenized cubic field} if $K$ is a
cubic field and $\O_K=\Z[\al]$ for some $\al\in\O_K$. More generally,
we define an {\bf $n$-monogenized cubic field} to be a pair $(K, \al)$
where $K$ is a cubic field, $\al\in\O_K$ is primitive in $\O_K/\Z$,
and $[\O_K:\Z[\al]]=n$; such an $\alpha$ is then called an {\bf $n$-monogenizer} of $K$ or
of $\O_K$. Two $n$-monogenized cubic fields
$(K,\al)$ and $(K',\al')$ are {\bf isomorphic} if $K$ and
$K'$ are isomorphic as cubic fields and, under such an isomorphism, $\al$ is mapped to
$\al'+m$ for some $m\in\Z$.

We next define an isomorphism-invariant height on the set of
$n$-monogenized cubic fields.  Let $(K,\alpha)$ be an $n$-monogenized cubic
field, and suppose that $f(x)$ is the characteristic polynomial of~$\alpha$. 
Since $(K,\alpha)$ and $(K,\alpha+m)$ are isomorphic for
every $m\in\Z$, we may define two invariants $I(f)$ and $J(f)$ on
the space of monic cubic polynomials $f(x)=x^3+ax^2+bx+c$, by setting 
\begin{equation*}
\begin{array}{rcl}
I(f)&:=&a^2-3b,\\[.05in]
J(f)&:=&-2a^3+9ab-27c.
\end{array}
\end{equation*}
These invariants satisfy $I(f(x))=I(f(x+m))$ and $J(f(x))=J(f(x+m))$
for all $m\in\Z$. 
We then define the {\bf height} $H$ of an $n$-monogenized cubic field $(K,\alpha)$ by 
\begin{equation}\label{eqheightn}
H(K,\alpha):=n^{-2}H(f):=n^{-2}\max\bigl\{|I(f)|^3,J(f)^2/4\bigr\}.
\end{equation}
Since the discriminant $\Delta(K)$ of $K$ is described in terms of the invariants of
$f$ as
\begin{equation}
27\Delta(K)=n^{-2}\bigl(4I(f)^3-J(f)^2\bigr),
\end{equation}
we see that the height of an $n$-monogenized cubic field is
comparable with its discriminant. 

We note that all cubic fields $K$ are $n$-monogenized for some
$n\ll|\Delta(K)|^{1/4}$; see Remark~\ref{remdelta14}.
Thus we may expect that the average size of the 2-torsion subgroup of the class group 
over all cubic fields $K$ ordered by
absolute discriminant is the same as the average over all
$n$-monogenized cubic fields $(K,\alpha)$ ordered by height
$H(K,\alpha)$ with $n\ll H(K,\alpha)^{1/4}$.  
This is indeed the case, as we will prove in Theorem \ref{thsubmonloc}.
We will also prove
in \S\ref{secsubmono}, that the number of $n$-monogenized cubic fields
$(K,\alpha)$ with height $H(K,\alpha)<X$ and $n\ll H(K,\alpha)^{1/4}$
grows as $\asymp X$, which are the same asymptotics as for the number of cubic
fields having absolute discriminant bounded by $X$.

We may also consider, for every $\delta\in(0,1/4]$, the average over
the much thinner family of $n$-monogenized cubic fields $(K,\alpha)$
satisfying $H(K,\alpha)< X$ and $n\leq H(K,\alpha)^\delta$; the number
of such $n$-monogenized fields grows as $\asymp
X^{5/6+2\delta/3}=o(X)$ (see Theorem~\ref{thsubcount}).  In
\S\ref{secsubmono}, we will prove the following theorem.

\begin{thm}\label{thsubmon}
Let $0<\delta\leq 1/4$ and $c>0$ be real numbers, and let
$F(\cdeltaint,X)$ denote the set of isomorphism classes of $n$-monogenized cubic fields
$(K,\al)$ with height $H(K,\al)<X$ and $n\leq \nolinebreak c H(K,\alpha)^\delta$. Then, as
$X\to\infty$:
  \begin{itemize}
  \item[{\rm (a)}] The average size of $\Cl_2(K)$
    over totally real cubic fields $K$ in $F(\cdeltaint,X)$ approaches
    $5/4;$
  \item[{\rm (b)}] The average size of $\Cl_2(K)$
    over complex cubic fields $K$ in $F(\cdeltaint,X)$ approaches
    $3/2$.
  \end{itemize}
Furthermore, these values remain the same if we instead average over
those $n$-monogenized cubic fields satisfying any given set of local
splitting conditions at finitely many primes.
\end{thm}

Theorem~\ref{thsubmon} shows that the averages in
Theorem~\ref{thallcubicfields} remain very stable, both under the
imposition of local conditions and under changing the ordering of the
fields from discriminant to~height.

\subsection
{The effect of $n$-monogenicity on the averages in Theorem~\ref{thallcubicfields} for fixed $n$}

We consider next the limiting behavior of Theorem~\ref{thsubmon} as
$\delta$ approaches $0$.  First, we consider the family of monogenized
cubic fields, ordered by height. In this case, we find that the
average values appearing in Theorems \ref{thallcubicfields} and
\ref{thsubmon} {\it do} change. We have the following theorem:

\begin{thm}\label{thmoncubicfields}
Let $K$ run through all isomorphism classes of monogenized cubic fields
  ordered by height. Then:
  \begin{itemize}
  \item[{\rm (a)}] The average size of $\Cl_2(K)$ over totally real
    cubic fields $K$ is $3/2$;
  \item[{\rm (b)}] The average size of $\Cl_2(K)$ over complex cubic
    fields $K$ is $2$. 
  \end{itemize}
Furthermore, these values remain the same if we instead average over those 
monogenized cubic fields satisfying any given set of local splitting
conditions at finitely many primes.
\end{thm}

Surprisingly, the values in Theorem \ref{thmoncubicfields} are
different from those in Theorems \ref{thallcubicfields} and \ref{thsubmon}. In particular,
these class groups do not appear to be random groups in the same sense
as Cohen--Lenstra--Martinet. It therefore appears that monogenicity
has an altering affect on the class groups of cubic fields!  To be
more precise, it appears that on average, monogenicity has a doubling
effect on the nontrivial part of the $2$-torsion subgroups of the
class groups of cubic fields.


More generally, we may ask what happens when we restrict Theorem
\ref{thsubmon} to $n$-monogenized cubic fields for any fixed positive integer
$n$. Once again, the averages occurring in
Theorems~\ref{thallcubicfields} and \ref{thsubmon} change as a rather
interesting function of $n$:

\begin{thm}\label{thmain2}
Fix a positive integer $n$ and write $n=m^2k$ where $k$ is
squarefree. Let $\sigma(k)$ denote the sum of the divisors of $k$. Let
$(K,\alpha)$ run through all isomorphism classes of $n$-monogenized
cubic fields ordered by height. Then:
\begin{itemize}
\item[{\rm (a)}] The average size of $\Cl_2(K)$ over totally real
  cubic fields $K$ is $\frac{5}{4}+\frac{1}{4\sigma(k)};$
\item[{\rm (b)}] The average size of $\Cl_2(K)$ over complex cubic
  fields $K$ is $\frac{3}{2}+\frac{1}{2\sigma(k)}$.
\end{itemize}
Furthermore, these values remain the same if we instead average over those 
$n$-monogenized cubic fields satisfying any given set of local splitting
conditions at finitely many primes $p\nmid k$.
\end{thm}

We also prove that the average sizes of $\Cl_2(K)$ in
Theorem~\ref{thmain2} surprisingly {\it do} change further if the
$n$-monogenized cubic fields $(K,\alpha)$ being averaged over satisfy
specified local conditions at primes dividing $k$, where $n=m^2k$ and
$k$ is squarefree. In Theorem \ref{main_theorem}, we compute the
average size of $\Cl_2(K)$ as $K$ varies over a family of
$n$-monogenized cubic fields satisfying any finite set (and certain
infinite sets) of local conditions. One simple consequence of
Theorem~\ref{main_theorem} is:

\begin{thm}\label{thmain1}
  Let $n$ be a fixed positive integer, and write $n=m^2k$ where $k$ is
squarefree. 
Let $(K,\alpha)$
  run through all isomorphism classes of $n$-monogenized cubic fields
  having discriminant prime to $k$, ordered by height. Then:
\begin{itemize}
\item[{\rm (a)}] The average size of $\Cl_2(K)$ over totally real
  cubic fields $K$ is $3/2$ if $k=1$ and $5/4$ otherwise;
\item[{\rm (b)}] The average size of $\Cl_2(K)$ over complex cubic
  fields $K$ is $2$ if $k=1$ and $3/2$ otherwise. 
\end{itemize}
Furthermore, these values remain the same if we instead average over
such monogenized cubic fields satisfying any given set of local
splitting conditions at finitely many primes.
\end{thm}

Note that the averages occurring in Theorem \ref{thmain1} agree with
the larger values in Theorem \ref{thmoncubicfields} when~$n$ is a
square, but with the smaller values in 
\pagebreak
Theorems~\ref{thallcubicfields} and~\ref{thsubmon} otherwise. Furthermore,
Theorems~\ref{thmain2} and \ref{thmain1} imply that for a positive
integer $n$, and a prime $p$ dividing $n$ to odd order, ramification
at $p$ causes, on average, an increasing effect on the $2$-torsion in
the class groups of $n$-monogenic cubic fields.

\vspace{.05in}

Theorems \ref{thmoncubicfields}, \ref{thmain2}, and \ref{thmain1} all
follow from our main result, Theorem~\ref{main_theorem} below, which
determines the average sizes of $\Cl_2$ and $\Cl_2^+$, the 2-torsion
subgroups of class groups and narrow class groups, respectively, in
any ``large'' family of $n$-monogenized cubic fields defined by local
conditions at~primes.  

For a prime $p$, let $T_p$ denote the set of all isomorphism classes
of {\bf $n$-monogenized \'etale cubic extensions of $\Q_p$}, i.e., the
set of all isomorphism classes of pairs $(\mathcal K_p,\alpha_p)$,
where $\mathcal K_p$ is an \'etal{e} cubic extension of $\Q_p$ with
ring of integers $\O_p$, and $\alpha_p\in\O_p$ is primitive in
$\O_p/\Z_p$, such that the $p$-part of the index of $\Z_p[\alpha_p]$
in $\O_p$ is equal to the $p$-part of $n$; here, two pairs $(\mathcal
K_p,\alpha_p)$ and $(\mathcal K'_p,\alpha'_p)$ are {\bf isomorphic} if
$\mathcal K_p$ and $\mathcal K'_p$ are isomorphic as $\Q_p$-algebras
and, under such an isomorphism, $\alpha_p$ is mapped to $\al'_p+m$ for
some $m\in\Z_p$. In Remark \ref{remTpmeasure}, we will see that the
set $T_p$ naturally injects into a closed subset of the set of cubic
polynomials over $\Z_p$ with leading coefficient~$n$, which equips
$T_p$ with a topology and a measure.

\hypertarget{fnx}{For} each prime $p$, let $\Sigma_p\subset T_p$ be an
open and closed set whose boundary has measure $0$. We say that
$\Sigma=(\Sigma_p)_p$ is a {\bf large collection of local
  specifications for $n$} if, for all but finitely many primes~$p$,
the set $\Sigma_p$ contains all pairs $(\mathcal K_p,\alpha_p)$ such
that $\mathcal K_p$ is {\it not} a totally ramified cubic extension
of~$\Q_p$. Let \hypertarget{fnx1}{$F(n,X)$} denote the set of
isomorphism classes of all $n$-monogenized cubic fields $(K,\alpha)$
such that $H(K,\alpha)<X$. Given a large collection
$\Sigma=(\Sigma_p)_p$ of local specifications, let
\hypertarget{fnxSig}{$F_\Sigma(n,X)$} denote the subset of $F(n,X)$
consisting of pairs $(K,\alpha)$ such that for all primes $p$, we have
$(K\otimes\Q_p,\alpha)\in\Sigma_p$.

For a prime $p$ dividing $n$, we say that an element
$(\mathcal K_p,\alpha_p)\in T_p$ is {\bf sufficiently ramified} if one of the
following two conditions is satisfied:
\begin{itemize}
\item [{\rm (a)}] $\mathcal K_p$ is a totally ramified cubic
  extension of $\Q_p;$
\item [{\rm (b)}] 
$\mathcal K_p=\Q_p\times F$, where $F$ is a ramified
  quadratic extension of $\Q_p$, 
and $\Z_p[\alpha_p]=\Z_p\times \O$, where~$\O$ is an order in $F$. 
\end{itemize}
For a prime $p$ dividing $n$, we define the {\bf local
  sufficiently-ramified density} $\rho_p(\Sigma_p)$ of a large
collection $\Sigma=(\Sigma_p)_p$ to be density of the set of
sufficiently-ramified elements within $\Sigma_p$. The {\bf global
  sufficiently-ramified density $\rho(\Sigma)$ with respect to $n$} is
then the product of $\rho_p(\Sigma_p)$ over all primes $p$ that divide
$n$ to an {\it odd} power. Our main theorem is as follows:

\begin{thm}[Main $n$-monogenic theorem] \label{main_theorem}
Let $n$ be a positive integer, and
let $\Sigma$ be a large collection of local specifications for
$n$. Then, as $X\to\infty$:
\begin{itemize}
\item[{\rm (a)}] The average size of $\Cl_2(K)$ over totally real
  cubic fields $K$ in $F_\Sigma(n,X)$ approaches $\frac{5}{4}+\frac{1}{4}\rho(\Sigma);$
\item[{\rm (b)}] The average size of $\Cl_2(K)$ over complex cubic
  fields $K$ in $F_\Sigma(n,X)$ approaches $\frac{3}{2}+\frac{1}{2}\rho(\Sigma);$
\item[{\rm (c)}] The average size of $\Cl^+_2(K)$ over totally real
  cubic fields $K$ in $F_\Sigma(n,X)$ approaches $2+\frac{1}{2}\rho(\Sigma)$.
\end{itemize}
\end{thm}
Theorem~\ref{thmoncubicfields} follows from Theorem~\ref{main_theorem}
by noting that there are no primes dividing $n=1$ to odd order, and so
$\rho(\Sigma)=1$ for every large collection of local specifications
$\Sigma$.  Theorem~\ref{thmain2} follows from
Theorem~\ref{main_theorem} by Corollary \ref{propdensuf}, which states
that the local sufficiently-ramified density $\rho_p(T_p)$ is equal to
$1/(p+1)$ for each $p\mid k$. Finally, Theorem~\ref{thmain1} follows
from Theorem~\ref{main_theorem} by noting that when $\Sigma_p\subset
T_p$ contains no extensions ramified at primes dividing $k$, then
$\rho_p(\Sigma_p)=0$.

Theorem~\ref{main_theorem} also implies an analogous increasing effect
of $n$-monogenicity on the 2-torsion subgroup of the {\it narrow}
class group, and this increase on average is from 2
(cf.\ \cite[Theorem 1(c)]{BV}) to~2.5, in the case when $n$ is a
square. 
\pagebreak

Taken together, Theorems \ref{thallcubicfields}, \ref{thmain2},
\ref{thmain1}, and \ref{main_theorem} thus paint the following
picture. For integers~$n$ that are perfect squares, $n$-monogenicity
has a doubling effect on the nontrivial part of the $2$-torsion in the
class groups of cubic fields. For nonsquare integers $n=m^2k$ with
$k>1$ squarefree, $n$-monogenicity still has an increasing effect on
the nontrivial part of the $2$-torsion in the class groups of cubic
fields. This increase goes to $0$ as~$k$ tends to infinity; moreover,
the increase is concentrated on those $n$-monogenized cubic fields
that
are sufficiently ramified at every prime dividing $k$. The analogous
phenomena also occur for the 2-torsion in narrow class groups of these
fields.

\subsection{Method of proof}

In \S\ref{secparam}, we prove
parametrizations of $n$-monogenized cubic fields and index 2 subgroups
of class groups of $n$-monogenic cubic fields, respectively, in terms of suitable
spaces of forms. Namely, in~\S\ref{subsecparamrings}, we begin by
proving that there is a natural bijection between
\begin{itemize}
\item[(a)]
$n$-monogenized cubic fields $(K,\alpha)$, and 
\item[(b)]
certain cubic polynomials $f$ with integer coefficients whose leading coefficient
is~$n$.  
\end{itemize}
In \S\S\ref{subsecparam2}--\ref{subsecparam3}, we use the parametrization of
quartic rings in \cite{quarparam} to give a natural bijection between:
\begin{itemize}
\item[(a)] index 2 subgroups of narrow class groups of $n$-monogenized cubic
fields $(K,\alpha)$ whose corresponding cubic polynomial is $f(x)$, and
\item[(b)] certain $\SL_3(\Z)$-orbits on pairs $(A,B)$ of integer-coefficient
ternary quadratic forms such that $4\,\det(Ax-B)=f(x)$.
\end{itemize}
Our main results are then proved by counting the relevant elements in
(b) having bounded height 
in both of the above parametrizations, and
then computing the limiting ratios.

More precisely, Theorem~\ref{thsubmon} is proven in \S\ref{secsubmono}
by counting $\SL_3(\Z)$-orbits of pairs $(A,B)$ of integer-coefficient
ternary quadratic forms such that $f(x):=4\,\det(Ax-B)$ satisfies $H(f)<X$
and $n=4\,\det(A)<c\smash{H(f)^\delta}$. The technique to obtain this count
is a suitable adaptation of the averaging and sieving methods
developed in \cite{dodqf,dodpf} for counting by discriminant, but
modified to allow imposing constraints on
two different parameters (the height $H$ and index $n$).

Next, Theorem~\ref{main_theorem} is proved in \S\ref{secn} by counting
pairs $(A,B)$ such that $f(x):=4\,\det(Ax-B)$ satisfies $H(f)<X$, but where
$4\det(A)=n$ is {\it fixed}.  This is much more subtle than the count~in Theorem~\ref{thsubmon} where $n$ varies.  It requires counting
pairs $(A,B)$ in a fundamental domain for the action of the group
$\SL_3(\Z)$ on the hypersurface defined by $\det(A)=n/4$.  To carry
out this count, we sum
the total number of $B$ in a
fundamental domain for the action of $\SO_{A}(\Z)$ on the space of real ternary
quadratic forms, where $A$ runs over a set of representatives 
for the distinct $\SL_3(\Z)$-classes of integer-coefficient ternary quadratic forms having determinant $n/4$.  
This calculation is intimately related to the
mass calculations for quadratic forms of a given determinant in 
\cite{Hanke_structure_of_massses,Hanke_ternary_quadratic_massses}; see also the work of Ibukiyama and Saito~\cite{IS} on 
Shintani zeta functions associated to spaces of quadratic forms.
This count, and the analogous counts with suitable congruence conditions, enable the completion of the proof of the general result,
Theorem~\ref{main_theorem}, implying in particular Theorems \ref{thmoncubicfields}, 
\ref{thmain2}, and \ref{thmain1}.

\vspace{.15in}
\noindent {\it Acknowledgments.} It is a pleasure to thank Hendrik Lenstra, Artane Siad,
Ashvin Swaminathan, Ling-Sang Tse, Ila Varma, and Mikaeel Yunus for
helpful conversations and many useful comments.
M.B.\ was supported by a Simons Investigator Grant and NSF
grant~DMS-1001828, and thanks the Flatiron Institute for its kind hospitality during the academic year 2019--2020. 
A.S.\ was supported by an NSERC discovery grant and a Sloan
fellowship.

\section{Parametrizations involving quartic rings and $n$-monogenized cubic rings}\label{secparam}

The purpose of this section is to describe the connection between
2-torsion subgroups in the class groups of $n$-monogenized cubic
fields and pairs $(A,B)$ of integer-coefficient ternary quadratic
forms with $4\,\det(A)=n$.

In \S\ref{subsecparamrings}, we adapt the Delone--Faddeev
parametrization \cite{DF} of cubic rings to obtain a parametrization
of isomorphism classes of $n$-monogenized cubic rings by integer-coefficient
binary cubic forms having leading coefficient $n$.  In
\S\ref{subsecparam2}, we recall the parametrization of quartic rings
by pairs of integer-coefficient ternary quadratic forms, as developed
in~\cite{quarparam}, and use these two parametrizations (in
conjunction with class field theory) in \S\ref{subsecparam3} to parametrize index $2$
subgroups of class groups of $n$-monogenized cubic fields.  Finally,
in \S\ref{subsecparampid}, we then discuss versions of these
parametrization results where $\Z$ is replaced by a principal ideal domain
$R$.

\subsection{Parametrization of $n$-monogenized cubic rings}\label{subsecparamrings}

\begin{defn}
\label{sec:2.1}
{\em 
A {\bf cubic ring} (resp.\ {\bf quartic ring}) is a ring that is
free of rank~$3$ (resp.\ rank~$4$) as a $\Z$-module.}
\end{defn}
The works of Levi \cite{Levi}, Delone--Faddeev \cite{DF}, and
Gan--Gross--Savin~\cite{GGS} give a parametrization of cubic rings by
$\GL_2(\Z)$-orbits of integer-coefficient binary cubic forms.

\begin{theorem}[\cite{Levi},\cite{DF},\cite{GGS}]\label{df}
There is a bijection between isomorphism classes of cubic rings $C$ with a chosen basis $\langle\bar\omega,\bar\theta\rangle$ of $C/\Z$, and integer-coefficient binary cubic forms $f(x,y)=ax^3+bx^2y+cxy^2+dy^3$. The bijection is given by 
\begin{equation}\label{dfbijection}
(C,\langle\bar\omega,\bar\theta\rangle)\mapsto 1\wedge(x\omega+y\theta)\wedge(x\omega+y\theta)^2= (ax^3+bx^2y+cxy^2+dy^3)\, (1\wedge\omega\wedge\theta)\in\wedge^3C
\end{equation} 
where $\omega,\theta\in C$ denote any lifts of $\bar\omega,\bar\theta\in C/\Z$.
\end{theorem}
See also \cite[\S 2]{MAJ} for a concise proof of Theorem~\ref{df}. 

\begin{nota}
\label{sec:2.3}
{\em
For any ring $R$, an element $\gamma\in \GL_2(R)$ acts on the space $U(R)$ of binary cubic forms $f$ with coefficients in $R$ via the twisted action \begin{equation}\label{gl2action}\gamma\cdot
  f(x,y):=\det(\gamma)^{-1}f((x,y)\gamma)\end{equation}
where we view $(x,y)$ as a row vector. 
}\end{nota}
We then have the following immediate corollary of Theorem~\ref{df}. 

\begin{corollary}\label{dfcor}
There is a bijection between isomorphism classes of cubic rings and $\GL_2(\Z)$-orbits on the space $U(\Z)$ of integer-coefficient binary cubic forms.
\end{corollary}
Corollary~\ref{dfcor} follows from Theorem~\ref{df} by noting that the action of $\GL_2(\Z)$ on the basis $\langle\omega,\theta\rangle$ of $C/\Z$ leads to the action (\ref{gl2action}) on the corresponding binary cubic form $f$.

We observe that the leading coefficient $a$ of the binary cubic form in the bijection of Theorem~\ref{df} is equal to the (signed) index of $\Z[\omega]$ in $C$, because setting $x=1$ and $y=0$ in \eqref{dfbijection} yields $1\wedge\omega\wedge\omega^2=a\,(1\wedge\omega\wedge\theta)$. Hence 
cubic rings having a monogenic subring of index $n$ may be classified in terms of binary cubic forms having (leading) $x^3$-coefficient $n$.

\begin{defn}
\label{sec:2.5}
{\em 
A pair $(C,\alpha)$ is an {\bf $n$-monogenized cubic ring}
if $C$ is a cubic ring, $\alpha$ is a primitive element of $C/\Z$, 
and $[C:\Z[\alpha]]=n$. Two $n$-monogenized cubic rings $(C,\al)$ and
$(C',\al')$ are {\bf isomorphic} if there is a ring isomorphism $C\to C'$ that maps $\al$ to $\al'+m$ for some $m\in\Z$.}
\end{defn}

\noindent
Thus an $n$-monogenized cubic ring is a cubic ring $C$ equipped with a basis $\langle\bar\omega,\bar\theta\rangle$ of $C/\Z$, but where $(C,\langle\bar\omega,\bar\theta\rangle)$ and $(C,\langle\bar\omega,k\bar\omega+\bar\theta\rangle)$ are considered isomorphic, as only the basis element $\bar\omega$ is relevant in defining the monogenic subring $\Z[\omega]$ of index $n$ in $C$. The change-of-basis $\langle\bar\omega,\bar\theta\rangle\mapsto\langle\bar\omega,k\bar\omega+\bar\theta\rangle$
corresponds to the transformation $f(x,y)\mapsto f(x+ky,y)$ on the associated binary cubic form. 

\begin{nota}
\label{sec:2.6}
{\em
For a ring $R$, we let $\MB(R)\subset \GL_2(R)$ denote the subgroup of lower triangular unipotent matrices
\begin{equation}\label{fz}
\MB(R):=\left\{\left(\begin{array}{rc} 1&{}\\k&1\end{array}\right):k\in R\right\}.
\end{equation}
For $n\in R$, let $U_n(R)\subset U(R)$ denote the subset of binary cubic forms having (leading) $x^3$-coefficient~$n$. 
}\end{nota}
Then we have proven the following theorem.

\begin{theorem}\label{thcubringpar}\label{df2}
  There is a bijection between isomorphism classes of
  $n$-monogenized cubic rings and $\MB(\Z)$-orbits on the set $U_n(\Z)$ of integer-coefficient binary cubic forms  having leading coefficient $n$.
\end{theorem}

\subsection{Parametrization of quartic rings having $n$-monogenic cubic resolvent rings}\label{subsecparam2} 

In this section, we recall the parametrization in \cite{quarparam} of
pairs $(Q, C)$, where $Q$ is a quartic ring and~$C$ is a cubic
resolvent ring of $Q$.  

\begin{nota}
\label{sec:2.8}
{\em
For any ring $R$, let $V(R)$ denote the space of pairs of ternary
quadratic forms with coefficients in $R$.  We represent an element in
$V(R)$ as $(A,B)$, where $A(x_1,x_2,x_3)=\sum_{1\leq i\leq j\leq 3}
a_{ij} x_i x_j$ and $B(x_1,x_2,x_3)=\sum_{1\leq i\leq j\leq 3} b_{ij}
x_i x_j$ with $a_{ij},b_{ij}\in R$.  When 2 is not a zero divisor in $R$, we may also represent
an element $(A,B)\in V(R)$ as a pair of $3\times 3$ symmetric matrices
with entries in $R[1/2]$ via the Gram identification
\begin{equation*}
(A,B)=\left( \frac12\left[ \begin{array}{ccc} 2a_{11} & a_{12} &
      a_{13} \\ a_{12} & 2a_{22} & a_{23} \\ a_{13} & a_{23} &
      2a_{33} \end{array} \right], \frac12\left[ \begin{array}{ccc}
      2b_{11} & b_{12} & b_{13} \\ b_{12} & 2b_{22} & b_{23} \\ b_{13}
      & b_{23} & 2b_{33} \end{array} \right] \right).
\end{equation*}
The group $G(R)\subset \GL_2(R)\times \GL_3(R)$ 
defined by 
\begin{equation*} G(R):=\{(g_2,g_3)\in\GL_2(R)\times\GL_3(R):\det(g_2)\det(g_3)=1\}
\end{equation*}
acts naturally on $V(R)$ as follows:
\begin{equation}\label{Gaction}
(g_2,g_3)\cdot(A,B)=(g_3Ag_3^t,g_3Bg_3^t)\cdot g_2^t.
\end{equation}
Given an element $(A,B)\in V(R)$, we define its {\bf cubic resolvent form} $\Res(A,B)\in U(R)$ by 
\begin{equation}\label{eqresolventmap}
\Res(A,B):=4\det(Ax-By).
\end{equation}
\pagebreak
Since 
the action of $G(R)$ on $V(R)$ in (\ref{Gaction}) 
and the resolvent map $\Res:V(R)\to U(R)$ in (\ref{eqresolventmap})
are defined by integer-coefficient polynomials in the entries of
$(g_2,g_3)\in G(R)$ and the coefficients $a_{ij}$ of $A$ and $b_{ij}$ of $B$, 
we see that 
the action of $G(R)$ on $V(R)$ and the definitions of $4\det(A)$ and $\Res(A,B)$ for $(A,B)\in V(R)$ make sense for
arbitrary rings $R$. 
}\end{nota} 
The following theorem is
proved~in~\cite{quarparam}.

\begin{theorem}[\cite{quarparam}]\label{thqrpar}
There is a bijection between pairs $(A,B)\in V(\Z)$ of integer-coefficient ternary quadratic forms and isomorphism classes of pairs $((Q,\langle\bar\alpha,\bar\beta,\bar\gamma\rangle),(C,\langle\bar\omega,\bar\theta\rangle))$,
where $Q$ is a quartic ring with a chosen basis $\langle\bar\alpha,\bar\beta,\bar\gamma\rangle$ of $Q/\Z$ and $C$ is a cubic resolvent
ring of $Q$ with a chosen basis $\langle\bar\omega,\bar\theta\rangle$ of~$C/\Z$. Furthermore, 
under this bijection, $(C,\langle\bar\omega,\bar\theta\rangle)$ is the data corresponding
to the cubic resolvent form of $(A,B)$ under Theorem~${\ref{df}}$.
\end{theorem}
A complete description of the construction of
$((Q,\langle\bar\alpha,\bar\beta,\bar\gamma\rangle),(C,\langle\bar\omega,\bar\theta\rangle))$
from $(A,B)$ can be found
in~\cite[\S3.2, \S3.3]{quarparam}. Theorem~\ref{thqrpar} has the
following immediate corollary.

\begin{corollary}\label{qparcor}
There is a bijection between $G(\Z)$-orbits on the space $V(\Z)$ of pairs of integer-coefficient ternary quadratic forms and isomorphism classes of pairs $(Q,C)$, where $Q$ is a quartic ring and $C$ is a cubic resolvent ring of $Q$.
\end{corollary}
Corollary~\ref{qparcor} follows from Theorems~\ref{df} and
\ref{thqrpar} by noting that the action of $G\subset
\GL_2(\Z)\times\GL_3(\Z)$ on the bases
$\langle\bar\alpha,\bar\beta,\bar\gamma\rangle$ of $Q/\Z$ and
$\langle\bar\omega,\bar\theta\rangle$ of $C/\Z$ leads to the action
(\ref{Gaction}) on the corresponding pair $(A,B)$ of ternary quadratic
forms.

Finally, the identical reasoning now yields the following parametrization of quartic rings having $n$-monogenized cubic rings. 
\begin{corollary}\label{qparcor2}
There is a bijection between $\MB(\Z)\times\SL_3(\Z)$-orbits on the set $V_n(\Z)$ of pairs $(A,B)$ of integer-coefficient ternary quadratic forms with $\det(A)=n$ and isomorphism classes of pairs $(Q,(C,\alpha))$, where $Q$ is a quartic ring and $(C,\alpha)$ is an $n$-monogenized cubic resolvent ring of~$Q$.
\end{corollary}

\subsection{Parametrization of index 2 subgroups of class 
groups of cubic fields}\label{subsecparam3} 

With the parametrizations of quartic rings having $n$-monogenized
cubic rings established, we are now in a position to parametrize index
2 subgroups in the class groups of $n$-monogenized cubic fields. This
parametrization will be useful to us because the number of elements of
order $2$ in a finite abelian group $A$ is equal to the number of
index 2 subgroups of $A$.

We will need the following definition.

\begin{defn}
\label{sec:2.12}
{\em 
  For a maximal quartic ring $Q$ and a prime $p$, we say that $Q$ is
  {\bf overramified at~$p$} if the ideal $p\Z$ factors in $Q$ as
  either $P^4$, $P^2$, or $P_1^2P_2^2$, where $P$, $P_1$, and $P_2$
  are prime ideals of $Q$.  A maximal quartic ring $Q$ is overramified
  at $\infty$ if $Q\otimes\R\cong \C^2$ as $\R$-algebras. A maximal
  quartic ring $Q$ is {\bf nowhere overramified} if it is not
  overramified at any (finite or infinite) place.}
\end{defn}
The significance of being nowhere overramified comes from the
following theorem of Heilbronn.

\begin{theorem}[\cite{Hcf}] \label{Theorem:Heilbronn_overramified}\label{h}
  Let $K_4$ be a totally real $S_4$-quartic field, and $K_3$ a cubic
  $($resolvent$)$ field inside $K_{24}$, the Galois closure of
  $K_{4}$. Let $K_6$ be the non-Galois sextic field in $K_{24}$
  containing $K_3$. Then the quadratic extension $K_6/K_3$ is
  unramified precisely when the quartic field $K_4$ is nowhere
  overramified. Conversely, every unramified quadratic extension
  $K_6/K_3$ of a cubic $S_3$-field $K_3$ lies in the Galois closure of
  a nowhere overramified quartic field $K_4$ which is unique up to
  conjugacy.
\end{theorem}

\begin{nota}
\label{sec:2.14}
{\em For the maximal
order $C$ in a cubic field $K_3$, let $\Cl_2(C)$ and $\Cl_2^+(C)$
denote the 2-torsion subgroups of the class group and narrow class
group of $C$, respectively.  Let $\Cl_2(C)^*$ and $\Cl_2^+(C)^*$
denote the groups dual to $\Cl_2(C)$ and $\Cl_2^+(C)$,
respectively. Then the set of nontrivial elements of $\Cl_2(C)^*$
(resp.\ $\Cl_2^+(C)^*$) are in bijection with the set of index
two subgroups of $\Cl(C)$ (resp.\ $\Cl^+(C)$) simply by mapping a
character to its kernel.
}\end{nota}

Theorems~\ref{thqrpar} and \ref{h}, together with class field theory,
now immediately yield a parametrization of index 2 subgroups of the
class groups and narrow class groups of cubic fields.  

\begin{theorem}\label{thclgp1}
Let $C$ be the maximal order in an $S_3$-cubic field $K_3$, and let
$f(x,y)$ be a binary cubic form corresponding to $C$ under Theorem~$\ref{df}$.
\begin{itemize}
\item[{\rm (a)}] If $\Delta(K_3)>0$, then there is a canonical bijection
  between elements of $\Cl_2^+(C)^*$ and 
  $\SL_3(\Z)$-orbits on $\Res^{-1}(f)\subset V(\Z)$.

  Under this bijection, elements of $\Cl_2(C)^*\subset \Cl_2^+(C)^*$
  correspond to 
   $\SL_3(\Z)$-orbits on pairs $(A,B)\in V(\Z)$
  such that $A(x,y,z)=B(x,y,z)=0$ has a nonzero solution over  $\R$.

\item[{\rm (b)}] When $\Delta(K_3)<0$, there is a canonical bijection
  between elements of $\Cl_2^+(C)^*=\Cl_2(C)^*$
  and $\SL_3(\Z)$-orbits on
  $\Res^{-1}(f)\subset V(\Z)$.
\end{itemize}
\end{theorem}

\begin{proof}
By class field theory, the nontrivial elements of $\Cl_2(C)^*$
(resp.\ $\Cl_2^+(C)^*$) correspond to quadratic extensions $K_6/K_3$
that are unramified at all places (resp.\ unramified at all finite
places). These quadratic extensions $K_6/K_3$, by Theorem~\ref{h}, in
turn correspond to quartic fields $K_4$ whose maximal orders $Q$ are
nowhere overramified (resp.\ not overramified at all finite
places). In this scenario, we have the equality of discriminants
$\Delta(Q)=\Delta(C)$ (by \cite{Hcf}), and so $C$ is~the unique cubic
resolvent ring of $Q$ (\cite[Def.~8]{quarparam}).  The bijections in
(a) and (b), for nontrivial elements of $\Cl_2^+(K_3)^*$ and
$\SL_3(\Z)$-orbits on $\Res^{-1}(f)\subset V(\Z)$ corresponding to
integral domains $Q$, now follow from Theorem~\ref{thqrpar}. The
bijections in both (a) and (b) of Theorem~\ref{thclgp1} are completed
by sending the identity element in $\Cl_2^+(K_3)^*$ to the unique
$\SL_3(\Z)$-orbit on $\Res^{-1}(f)\subset V(\Z)$ corresponding to the
quartic ring $Q=\Z\times C$, whose unique cubic resolvent ring is $C$
as well.
\end{proof}
\subsection{Parametrizations over principal ideal domains}\label{subsecparampid}

\begin{defn}{\em 
Let $R$ be a principal ideal domain. A {\bf cubic ring}
$($resp.\ {\bf quartic ring}$)$ {\bf over $R$} is an $R$-algebra that is
free of rank $3$ $($resp.\ rank $4)$ as an $R$-module.}
\end{defn}
The following theorem, proved by Gross and Lucianovic
\cite{GS} and in~\cite{BSW}, generalizes
the parametrization of cubic and quartic rings over $\Z$ in \cite{DF}
and \cite{quarparam}, respectively, to the setting of cubic and
quartic rings over a principal ideal domain. For a formulation and
proof in the vastly more general case where $\Z$ is replaced by an
arbitrary ring, or even an arbitrary base scheme, see
Wood~\cite{Wood}.

\begin{theorem}[{\cite[Proposition 2.1]{GS}, \cite[Theorem~5]{BSW}}]\label{thpid}
Let $R$ be a principal ideal domain. 

\medskip

{\rm (a)} There is a natural bijection between isomorphism classes of
cubic rings over $R$ and $\GL_2(R)$-orbits on $U(R)$. Under this
bijection, the group of automorphisms of a cubic ring over $R$ is
isomorphic to the stabilizer in $\GL_2(R)$ of the corresponding binary
cubic form in $U(R)$.

\medskip

{\rm (b)} There is a natural bijection between isomorphism classes of
pairs $(Q,C)$, where $Q$ is a quartic ring over $R$ and $C$ is a cubic
resolvent ring of $Q$, and $G(R)$-orbits on $V(R)$. Under this
bijection, the group of automorphisms of a quartic ring $Q$ over $R$
is isomorphic to the stabilizer in $G(R)$ of the corresponding pair of
ternary quadratic forms in $V(R)$.
\end{theorem}

\begin{defn}{\em 
Let $R$ be a principal ideal domain. For $n\in R$, an {\bf $n$-monogenized cubic
  ring over $R$} is a pair $(C,\alpha)$, where $C$ is a cubic ring
over $R$, the element $\alpha\in C$ is primitive in $C/R$, and~the $R$-ideal
$(C:R[\alpha]):=\{r\in R: r \wedge^3C\subset\wedge^3(R[\alpha]) \mbox{ as $R$-modules}\}$ 
is generated by $n$. Two $n$-monogenized
  cubic rings $(C,\alpha)$ and $(C',\alpha')$ over~$R$ 
  are {\bf isomorphic} if there is an $R$-algebra isomorphism
  $C\to C'$ that maps $\alpha$ to $\alpha'+m$ for some~$m\in R$.}
  \end{defn}
Note that if $R=\Z$ or $R=\Z_p$, then $(C:R[\alpha])$ is the ideal generated by
  the usual index $[C:R[\alpha]].$

\vspace{.0125in}
The proofs of Theorems \ref{thcubringpar} and \ref{thpid} have the following consequence.
\begin{corollary}\label{thZp}
Let $R$ be a principal ideal domain. There is a bijection between
isomorphism classes of $n$-monogenized cubic rings over $R$ and
$M(R)$-orbits on the set $U_n(R)$.
\end{corollary}

\begin{rem}\label{remTpmeasure}{\em 
When $R=\Z_p$, Corollary \ref{thZp} provides a bijection between isomorphism classes
of $n$-monogenized cubic rings over $\Z_p$ and the set
\begin{equation}\label{eqFDZp}
\bigl\{f(x,y)=nx^3+bx^2y+cxy^2+dy^3:b\in\Z,\;c,d\in\Z_p,\;0\leq b< \gcd(p,3n)\bigr\}
\end{equation}
which is a fundamental domain for the action of $M(\Z_p)$ on $U_n(\Z_p)$. 
The set \eqref{eqFDZp} can be identified with
$\Z_p\times\Z_p\times\{0,1,\ldots,\gcd(p,3n)-1\}$. We then equip the set of
isomorphism classes of $n$-monogenized cubic rings over $\Z_p$ with the product 
topology and measure, via the usual topology and additive Haar 
measure on $\Z_p$ and the discrete topology and counting measure on
$\{0,1,\ldots,\gcd(p,3n)-1\}$.

By Lemma~\ref{propcondmax}, the set $T_p$ of isomorphism classes of 
\'etale cubic extensions of $\Q_p$ can be
identified with an open and closed subset of
$\Z_p\times\Z_p\times\{0,1,\ldots,\gcd(p,3n)-1\}$, namely, by identifying $T_p$ with the subset of \eqref{eqFDZp} corresponding to $n$-monogenized cubic rings
$(\O_p,\alpha_p)$ over $\Z_p$, where $\O_p$ is the maximal order in an \'etale cubic extension of $\Q_p$. 
The set $T_p$ also thereby inherits a topology and a measure.}
\end{rem}

We will have occasion to use Theorem~\ref{thpid} with $R$ specialized
to $\R$, $\Z_p$, and $\F_p$. 
In these cases, it will be convenient to use the language of splitting types, which we now define.
\begin{defn}
\label{sec:2.17}
{\em 
Let $f$ be a binary cubic form in $U(\R)$, $U(\Z_p)$, or $U(\F_p)$. When $f$
belongs to $U(\R)$, assume that $f$ has three roots (counted with
multiplicity) in $\P^1(\C)$, and when $f$ belongs to $U(\Z_p)$ or
$U(\F_p)$, assume that $f$ has three roots
(counted with multiplicity) in $\P^1(\overline{\F}_p)$.  We define the
{\bf splitting type} of $f$ to be $(f_1^{e_1}f_2^{e_2}\cdots)$, where
the $f_i$ are the degrees over $\R$ or $\F_p$ of the fields of definition of these roots
and the $e_i$ are their respective multiplicities.

Similarly, let $(A,B)$ be an element in $V(\R)$, $V(\Z_p)$, or
$V(\F_p)$. When $(A,B)$ belongs to $V(\R)$, assume that the
intersection of the conics defined by $A$ and $B$ consists of four
points (counted with multiplicity) in $\P^2(\C)$. When $(A,B)$ belongs
to $V(\Z_p)$ or $V(\F_p)$, assume that the intersection of the conics
defined by $A$ and $B$ consists of four points (counted with
multiplicity) in $\P^2(\overline{\F}_p)$. We define the {\bf splitting
  type} of $(A,B)$ to be $(f_1^{e_1}f_2^{e_2}\cdots)$, where the $f_i$
are the degrees over $\R$ or $\F_p$ of the fields of definition of
these points and the $e_i$ are their respective multiplicities.}
\end{defn}

We conclude this section with a discussion of local versions of
Theorem \ref{thclgp1}.  Let $p$ be a prime and let $K_m$ be an
\'etale degree $m$ extension of $\Q_p$. We write $K_m=\prod_{i=1}^k
L_i$ as a product of fields $L_i$. An unramified degree $2$ extension
$K_{2m}$ of $K_m$ is a product \smash{$\prod_{i=1}^k L_i'$}, where~$L_i'$ is
either $L_i\times L_i$ (i.e., $L_i$ splits) or the (unique) quadratic
unramified extension of $L_i$ (i.e., $L_i$ is~inert).  Let
$\Aut_{K_m}(K_{2m})$ denote the group of automorphisms of $K_{2m}$
fixing $K_m$ pointwise. As there are~$2^k$ different unramified
extensions $K_{2m}$ of $K_m$, and each such extension has automorphism
group~$(\Z/2\Z)^k$, we have
\begin{equation}\label{eqmasstriv}
\sum_{\substack{[K_{2m}:K_m]=2\\{\rm unramified}}}\frac{1}{|\Aut_{K_m}(K_{2m})|}=1.
\end{equation}
The norm of the discriminant of each $K_{2m}/K_m$ is either a 
square or nonsquare in $\Z_p^\times$. Indeed, 
\begin{equation*}
N_{K_m/\Q_p}\Delta(K_{2m}/K_m)=\smash{\prod_{i=1}^k} N_{L_i/\Q_p}\Delta(L_i'/L_i). 
\end{equation*} 
Let $e_i$ denote the ramification degree of $L_i$; then 
$N_{K_m/\Q_p}\Delta(K_{2m}/K_m)$ is a square if and only if there are
an even number of $i$ for which both 
$L_i$ is inert and $e_i$ is odd. 

We now focus on the case $m=3$. 
The arguments in \cite[\S2]{Baily} imply the following.
\begin{theorem}[\cite{Baily}]\label{thHeilbronPID}
There is a bijection between \'etale non-overramified quartic
extensions~$K_4/\Q_p$ with cubic resolvent $K_3$ and \'etale unramified
quadratic extensions $K_6/K_3$ such that $N_{K_3/\Q_p}\Delta(K_6/K_3)$
is a square in~$\Z_p^\times$.
\end{theorem}

\begin{rem}\label{geom} 
{\em 
The bijection of Theorem \ref{thHeilbronPID} can be made explicit
using the correspondences of Theorem \ref{thpid}. Let $K_3$ be an
\'etale cubic algebra corresponding to a binary cubic form $f(x,y)$.
Let $\PP=\{P_1,P_2,P_3\}$ denote the set of roots of $f(x,y)$ in
$\P^1(\overline{\Q}_p)$. Let $K_4$ be an \'etale non-overramified
quartic extension of $\Q_p$ with resolvent $K_3$ corresponding to a
pair $(A,B)$ of ternary quadratic forms. Let $\cQ=\{Q_1,Q_2,Q_3,Q_4\}$
denote the set of common zeros of $A$ and $B$ in
$\P^2(\overline{\Q}_p)$. These two sets $\PP$ and~$\cQ$ come equipped
with an action of the absolute Galois group $G_{\Q_p}$ of
$\Q_p$. Furthermore, a set of three points (resp.\ four points)
together with an action of $G_{\Q_p}$ uniquely determines an \'etale
cubic (resp.\ quartic) extension of $\Q_p$, and in this manner $\PP$
corresponds to the \'etale cubic extension $K_3$ and $\cQ$ corresponds
to the \'etale quartic extension $K_4$ of $\Q_p$.

Let $\PP'$ denote the set of pairs of lines $(L_1,L_2)$ where $L_1$
passes through two of the points $Q_i$ and $L_2$ passes through the
other two $Q_i$. Then $\PP'$ has three elements and we have a natural
bijection  $\PP \to \PP'$
that is $G_{\Q_p}$-equivariant. Indeed, we simply associate to
$P_i=(x_i,y_i)$ the zero set of $Ax_i-By_i$, which is a pair of lines
in $\P^2(\overline{\Q}_p)$ since $4\det(Ax_i-By_i)=f(x_i,y_i)=0$. Moreover,
since the four points in $\cQ$ are in general position, both of these
lines must pass through exactly two points in~$\cQ$. We may thus
naturally identify the Galois sets $\PP$ and $\PP'$.  Let $\LL$ denote
the set of six lines passing through each choice of two points in
$\cQ$. The Galois action on $\cQ$ induces one 
on~$\LL$, which in turn yields an \'etale sextic extension $K_6$ of
$\Q_p$; this is indeed the unramified quadratic extension of $K_3$
with discriminant of square norm corresponding to $K_4$ in~Theorem
\ref{thHeilbronPID}.}
\end{rem}

\section{The mean number of $2$-torsion elements in the class
  groups of $n$-monogenized cubic fields ordered by height with varying $n$}\label{secsubmono}

The purpose of this section is to prove a more general version of
Theorem~\ref{thsubmon} where we also allow our $n$-monogenized cubic
fields to satisfy certain infinite sets of congruence conditions.  We
fix real numbers $c$ and $\delta$ satisfying $c>0$ and $0<\delta\leq
1/4$.

\begin{defn}
\label{sec:3.1}
{\em 
For each prime $p$, let $S_p\subset U(\Z_p)$ be an open and closed
nonempty subset whose boundary has measure $0$. Then $S:=(S_p)_p$ is
a {\bf collection of cubic local specifications}. The
collection $S=(S_p)_p$ is {\bf large} if, for all but finitely many primes~$p$, 
the set $S_p$ contains all elements $f\in U(\Z_p)$ with
$p^2\nmid\Delta(f)$. To each $S_p$, we associate the set $\Sigma_p$ of
pairs $(\mathcal K_p,\alpha_p)$, up to isomorphism, where $\mathcal K_p$ is an \'etale cubic extension of
$\Q_p$ with ring of integers $\O_p$, $\alpha_p$ is an element of
$\O_p$ that is primitive in $\O_p/\Z_p$, and the pair
$(\O_p,\alpha_p)$ corresponds to some $f(x,y)\in S_p$.  The collection
$\Sigma:=(\Sigma_p)_p$ is called {\bf large} if $S$ is large. For a large collection $\Sigma$, let $F_\Sigma(\cdeltaint,X)$ denote
the set of isomorphism classes of $n$-monogenized cubic fields $(K,\alpha)$ such that $n\leq
cH(K,\alpha)^\delta$, $H(K,\alpha)<X$, and
$(K\otimes\Q_p,\alpha)\in\Sigma_p$ for all primes~$p$.}
\end{defn}

The main result of this section is then the following theorem.

\pagebreak
\begin{theorem}\label{thsubmonloc}
Let notation be as above. As $X\to\infty$, we have:
  \begin{itemize}
  \item[{\rm (a)}] The average size of $\Cl_2(K)$
    over totally real cubic fields $K$ in $F_\Sigma(\cdeltaint,X)$ approaches
    $5/4$;
  \item[{\rm (b)}] The average size of $\Cl_2(K)$
    over complex cubic fields $K$ in $F_\Sigma(\cdeltaint,X)$ approaches
    $3/2$;
    \item[{\rm (c)}] The average size of $\Cl^+_2(K)$ over totally
      real cubic fields $K$ in $F_\Sigma(\cdeltaint,X)$ approaches $2$.
  \end{itemize}
\end{theorem}
That is, the averages in Theorem~\ref{thallcubicfields} and
\cite[Theorem 1]{BV} remain the same even when ordering cubic fields
by height, restricting $n$ to slowly growing ranges relative to the
height, and imposing quite general local conditions on the
cubic~fields.

\begin{rem}\label{remdelta14}
{\em We always take $\delta\leq1/4$ because
  every cubic field $K$ is $n$-monogenized for some
  $n\ll|\Delta(K)|^{1/4}$. Indeed, the covolume of $\O_K$ in
  $\O_K\otimes \R$ is $|\Delta(K)|^{1/2}$, and so the length of the second
  successive minimum $\alpha\in\O_K$ is $\ll
  |\Delta(K)|^{1/4}$ (the first successive minimum being $1\in\O_K$). Therefore, $$|\Delta(\Z[\alpha])|=|\Delta(\langle1,\alpha,\alpha^2\rangle)|\ll
  (1\: |\Delta(K)|^{1/4}  |\Delta(K)|^{1/2})^2\ll  |\Delta(K)|^{3/2}.$$ Hence $$[\O_K:\Z[\alpha]]=|\Delta(\Z[\alpha])/\Delta(\O_K)|^{1/2} \ll (|\Delta(K)|^{3/2}/|\Delta(K)|)^{1/2} =
  |\Delta(K)|^{1/4}.$$}
\end{rem}

This section is organized as follows.  In \S\ref{subsecsubcub}, we
give asymptotics for the number of $n$-monogenized cubic rings of
bounded height $H$ and $n<cH^\delta$ in terms of local volumes of
certain sets of binary cubic forms. In \S\ref{subsecsubquar}, we
determine asymptotics for the number of quartic rings with 
$n$-monogenized cubic resolvent rings, where these resolvent rings
again have bounded height and bounded $n$. In~\S\ref{subsecsubunif},
we then prove uniformity estimates that allow us to impose conditions
of maximality on these counts. 

The leading constants for these asymptotics are expressed
as a product of local volumes of sets in $U(R)$ and $V(R)$, where $R$
ranges over $\R$ and $\Z_p$ for all primes $p$. In \S\ref{seclmsub},
we prove certain mass formulas relating \'etale quartic and cubic
extensions of~$\Q_p$.  Finally, in \S\ref{subsecsubvol}, we use these
mass formulas to compute the required local volumes, concluding the
proofs of Theorem~\ref{thsubmon} and Theorem~\ref{thsubmonloc}.

\subsection{The number of $n$-monogenized cubic
  rings of bounded height $H$ and $n<cH^\delta$}\label{subsecsubcub}

In this subsection, we determine the asymptotic number of
$n$-monogenized cubic rings of bounded index and height. By Theorem
\ref{thcubringpar}, such rings are parametrized by $\MB(\Z)$-orbits on
binary cubic forms in $U(\Z)$ of bounded height $H$ whose
$x^3$-coefficient is positive and less than~$cH^\delta$.

\begin{defn}
\label{sec:3.4}
{\em
For a binary cubic form $f(x,y)=ax^3+bx^2y+cxy^3+dy^3\in U(\R)$, we define the {\bf index}
$\ind$, the {\bf $F$-invariants} $I$ and $J$, {\bf height} $H$, and {\bf discriminant} $\Delta$ of $f$ by:
\begin{equation*}
\begin{array}{rcl}
\ind(f)&:=&a;\\[.05in]    
I(f)&:=&b^2-3ac;\\[.05in]
J(f)&:=&-2b^3+9abc-27a^2d;\\[.05in]
H(f)&:=&a^{-2}\max\bigl\{|I(f)|^3,J(f)^2/4\}; \\[.05in]
\Delta(f)&:=&b^2c^2-4ac^3-4b^3d-27a^2d^2+18abcd. 
\end{array}
\end{equation*}
}\end{defn} 

If the $\MB(\Z)$-orbit of an element $f\in U(\Z)$ corresponds to an
$n$-monogenized ring $(C,\alpha)$ by the bijection of Theorem~\ref{df2}, then: 
\begin{equation*}
\ind(f)=n;\quad I(f)=I(C,\alpha);\quad J(f)=J(C,\alpha);\quad H(f)=H(C,\alpha);\quad \Delta(f)=\Delta(C).\end{equation*} 
\vspace{-.3in}\pagebreak

\begin{cons} 
\label{sec:3.5}
{\em
Define the sets \smash{$\FF_U^\pm$} and \smash{$\FF_U^\pm(\cdeltaint,X)$} as follows:
\begin{equation*}
\begin{array}{rcl}
\displaystyle\FF_U^\pm&\!\!:=\!\!&
\displaystyle\{f(x,y)=ax^3+bx^2y+cxy^2+dy^3\in U(\R):\pm\Delta(f)>0,
  \;0<a,\;0\leq b<3a\};
\\[.05in]
\displaystyle \smash{\FF_U^\pm}(\cdeltaint,X)&\!\!:=\!\!&\displaystyle\bigl\{f(x,y)\in\FF_U^\pm:\ind(f)\leq
cH(f)^\delta,\; H(f)<X\bigr\}.
\end{array}
\end{equation*}
Then $\smash{\FF_U^\pm}$ is a fundamental domain for the action of $\MB(\Z)$ on
the set of binary cubic forms $f(x,y)\in U(\R)$ such that
$\pm\Delta(f)>0$ and the $x^3$-coefficient of $f(x,y)$ is
positive. (In this paper, the symbol ``$\pm$'' always refers to two
distinct statements, one for $+$ and one for $-$.)
}\end{cons}
In this subsection, we prove the following theorem.

\begin{theorem}\label{thsubcount}
Let $\smash{N^\pm_3}(\cdeltaint,X)$ denote the number of isomorphism classes of $n$-monogenized $S_3$-orders
$(C,\alpha)$ such that $\pm\Delta(C)>0$, $n\leq \smash{cH(C,\alpha)^{\delta}}$, and $H(C,\alpha)<X$. Then
\begin{equation*}
  \smash{N^\pm_3(\cdeltaint,X)}=\Vol\bigl(\smash{\FF_U^\pm}(\cdeltaint,X)\bigr)+O(X^{5/6}) =
  k\, X^{5/6+2\delta/3}+O(X^{5/6}),
\end{equation*}
where $k$ and the implied $O$-constant depend only on $n$, $c$, and $\delta$.
\end{theorem}

We begin by counting $\MB(\Z)$-equivalence classes of integer-coefficient binary
cubic forms, with bounded index and height, that satisfy any finite set
of congruence conditions. We use the following
result on counting lattice points in regions due to Davenport.

\begin{proposition}[\cite{Davenport1}]\label{davlem}
  Let $\mathcal R$ be a bounded, semi-algebraic multiset in $\R^n$
  having maximum multiplicity $m$ that is defined by at most $k$
  polynomial inequalities each having degree at most $\ell$.  Let
  $\RR'$ denote the image of $\RR$ under any $($upper or lower$)$ 
  triangular unipotent transformation of $\R^n$. Then the number of
  integer lattice points $($counted with multiplicity$)$ contained in
  the region $\mathcal R'$ is
\[\Vol(\mathcal R)+ O(\max\{\Vol(\bar{\mathcal R}),1\}),\]
where $\Vol(\bar{\mathcal R})$ denotes the greatest $d$-dimensional
volume of any projection of $\mathcal R$ onto a coordinate subspace
obtained by equating $n-d$ coordinates to zero, where $d$ takes all
values from $1$ to $n-1$.  The implied constant in the second summand
depends only on $n$, $m$, $k$, and $\ell$.
\end{proposition}

\begin{nota}
\label{sec:3.8}
{\em Let $L\subset U(\Z)$ denote an $\MB(\Z)$-invariant set defined by
  congruence conditions modulo some positive integer, and let $\nu(L)$
  denote the volume of the closure of $L$ in $U(\widehat{\Z})$,
  where~$\widehat{\Z}:=\prod_p \Z_p$. For any subset $L\subset U(\R)$,
  let $L^\pm$ denote the set of elements $f\in L$ with
  $\pm\Delta(f)>0$.  For real numbers $\Y,X>0$ such that
  $\Y=O(X^{1/4})$, define the sets
\begin{equation}\label{definesets}
\displaystyle\FF_U^\pm(\ycox):=\displaystyle\bigl\{f(x,y)\in\FF_U^\pm:\Y\leq
\ind(f)<2\Y,\; X\leq H(f)<2X\bigr\}.
\end{equation}
}\end{nota}

\begin{proposition}\label{propcountsmf}
We have
\begin{equation*}
  \#\bigl\{f\in \MB(\Z)\backslash L^\pm:
  \Y\leq \ind(f)<2\Y,\;X\leq H(f)<2X\bigr\}
=\nu(L)\Vol\bigl(\FF_U^\pm(\ycox)\bigr)+O\bigl(X^{5/6}/\Y^{1/3}\bigr).
\end{equation*}
\end{proposition}

\begin{proof}
The height and leading-coefficient bounds on an element
$f(x,y)=ax^3+bx^2y+cxy^2+dy^3\in \FF_U^\pm(\ycox)$ imply that we have:
\begin{equation}\label{eqFUYX}
|a|\ll \Y;\quad |b|\ll \Y;\quad |c|\ll X^{1/3}/\Y^{1/3};
\quad |d|\ll X^{1/2}/\Y.  
\end{equation}
Applying Proposition \ref{davlem} yields
\begin{equation*}
  \#\bigl(\FF_U^\pm(\ycox)\cap L\bigr)=\nu(L)
  \Vol\bigl(\FF_U^\pm(\ycox)\bigr)+O\bigl(X^{5/6}/\Y^{1/3}).
\end{equation*}
Since $\FF_U^\pm$ is a fundamental domain for the action of $\MB(\Z)$ on
$U(\R)^\pm$, the proposition follows.
\end{proof}

\noindent
The main term in Proposition~\ref{propcountsmf} grows as $X^{5/6}\Y^{2/3}$. Hence
the error term is smaller than the main term whenever $\Y\gg_\epsilon
X^{\epsilon}$.\pagebreak

\begin{defn}
\label{sec:3.10}
{\em 
An element $f\in U(\Z)$ is {\bf generic} if the cubic ring
corresponding to $f$ is an order in an $S_3$-cubic field. Similarly,
an element $v\in V(\Z)$ is {\bf generic} if the quartic ring
corresponding to $v$ is an order in an $S_4$-quartic field.  For any
subset $L$ of $U(\Z)$ (resp.\ $V(\Z)$), we denote the set of generic
elements in $L$ by $L^\gen$.}
\end{defn}

\begin{lemma}\label{lemsubnongencub}
  We have
$  \#\bigl((U(\Z)\setminus U(\Z)^\gen)\cap \FF_U^\pm(\ycox)
  \bigr)=O_\epsilon(X^{1/2+\epsilon}\,\Y).$
\end{lemma}
\begin{proof}
This follows immediately from the coefficient bounds
\eqref{eqFUYX} along with the fact that when $a$ and $d$ are
nonzero, the coefficients $a$, $b$, and $d$ of an integer-coefficient binary cubic
form $ax^3+bx^2y+cxy^2+dy^3\in U(\Z)\setminus U(\Z)^\gen$ determine
$c$ up to $O_\epsilon(X^\epsilon)$ choices (as in the proofs of
\cite[Lemmas~21--22]{MAJ}).
\end{proof}

\vspace{.025in}\noindent
{\bf Proof of Theorem~\ref{thsubcount}:}
Theorem \ref{thsubcount} immediately follows by breaking the intervals
$[1,X]$ and $[1,cX^\delta]$ into dyadic ranges, and then applying
Proposition \ref{propcountsmf} and Lemma \ref{lemsubnongencub} to each
pair of dyadic ranges. $\Box$

\subsection{The number of quartic rings having $n$-monogenized cubic
  resolvent rings of bounded height~$H$ and $n<cH^\delta$}\label{subsecsubquar}

In this subsection, we determine the asymptotic number of quartic
rings having an $n$-monogenized cubic resolvent ring with bounded
height $H$ and $n<cH^\delta$.

\begin{theorem}\label{thqsubcount}
For $i\in\{0,1,2\}$, let $\smash{N_4^{(i)}}(\cdeltaint,X)$ denote the number
of isomorphism classes of pairs $(Q,(C,\alpha))$, where $Q$ is an order in an $S_4$-quartic
field with $4-2i$ real embeddings and $(C,\alpha)$ is an $n$-monogenized
cubic resolvent ring of $Q$ with $n\leq cX^{\delta}$ and $H<X$. Then
\begin{equation}\label{n4eq}
N^{(i)}_4(\cdeltaint,X)=\frac{1}{m_i}\,
\Vol\bigl(\FF_{\SL_3}\cdot\FF_V^{(i)}(\cdeltaint,X)\bigr)+O(X^{5/6+\delta/3}),
\end{equation}
where \smash{$\FF_{\SL_3}$} is a fundamental domain for the action of
$\SL_3(\Z)$ on $\SL_3(\R)$, the sets $\smash{\FF_V^{(i)}}(\cdeltaint,X)$ and the constants $m_i$ are
defined immediately after Proposition $\ref{propsmfssize}$. 
\end{theorem}

\begin{nota}
\label{sec:3.13}
{\em 
Let $V(\Z)_+:=\{(A,B)\in V(\Z) : n=\det(A)>0\}=V(\Z)\cap V(\R)_+$. 
}\end{nota}

To prove Theorem~\ref{thqsubcount}, we use the natural bijection of Corollary~\ref{qparcor2} 
between the sets
\begin{equation*}
  \bigl\{(Q,(C,\alpha))\bigr\}\leftrightarrow \bigl(\MB(\Z)\times\SL_3(\Z)\bigr)\backslash
  V(\Z)_+,
\end{equation*}
where $Q$ is a quartic ring, $C$ is a cubic resolvent ring of $Q$, and
$\alpha$ is an $n$-monogenizer of $C$ for some $n\geq 1$. Given $(A,B)\in V(\Z)_+$
corresponding to $(Q,(C,\alpha))$, the resolvent binary cubic form
$\Res(A,B)=f(x,y)=4\det(Ax-By)$ corresponds to the $n$-monogenized cubic
ring~$(C,\alpha)$ under the bijection of Theorem~\ref{thcubringpar}.

\begin{defn}
\label{sec:3.14}
{\em
We use the resolvent map $\Res$ to define the functions $\ind$, $I$, $J$,
$H$, and $\Delta$ on $V(\R)_+:=\{(A,B)\in V(\R) : \det(A)>0\}.$ For $v\in V(\R)_+$,
we set
\begin{equation*}
  \ind(v)\!:=\!\ind(\Res(v));\:
  I(v)\!:=\!I(\Res(v));\:
  J(v)\!:=\!J(\Res(v));\:
  H(v)\!:=\!H(\Res(v));\:
  \Delta(v)\!:=\!\Delta(\Res(v)).
\end{equation*}
These functions are invariants for the action of $\MB(\R)\times\SL_3(\R)$ on $V(\R)_+$. 
}\end{defn}

In the rest of this subsection, to prove Theorem~\ref{thqsubcount}, we determine the number of
$\MB(\Z)\times\SL_3(\Z)$-orbits on $V(\Z)_+$ having bounded height and
index.

\subsubsection{Reduction theory for the action of
  $\MB(\Z)\times\SL_3(\Z)$ on $V(\R)_+$}

\begin{nota}
\label{sec:3.15}
{\em
For $i\in\{0,1,2\}$, let $V(\R)^{(i)}\subset V(\R)_+$ denote the set of elements 
corresponding to the $\R$-algebra $\R^{4-2i}\times\C^i$.  
That is, $V(\R)^{(i)}$ consists of elements having splitting type $(1111)$ when
$i=0$, $(112)$ when $i=1$, and $(22)$ when $i=2$.
}\end{nota}

\begin{lemma}\label{lemsubmonss}
The size of the stabilizer in $\SL_3(\R)$ of an element $(A,B)\in
V(\R)$ is $4$ if $\Delta(A,B)>0$ and $2$ if $\Delta(A,B)<0$.
\end{lemma}
\begin{proof}
As described in Remark \ref{geom}, a nondegenerate element
$(A,B)\in V(\R)$ gives rise to a set $\cQ$ of four points in
$\P^2(\C)$, equipped with the action of $\Gal(\C/\R)=\Z/2\Z$. Let
$\PP'$ denote the set of pairs of lines $(L_1,L_2)$, where $L_1$
passes through two of the points in $\cQ$, and $L_2$ passes through
the other two. The set $\PP'$ inherits an action of $\Gal(\C/\R)$ from
$\cQ$. Theorem \ref{thpid} now implies that the stabilizer in
$\SL_3(\R)$ of $(A,B)$ is isomorphic to the set of Galois-invariant
permutations of $\cQ$ which induce the trivial permutation on $\PP'$.

If $\Delta(A,B)>0$, then $\cQ$ either consists of four points defined
over $\R$ or consists of two pairs of complex conjugate points. In
either case, the nontrivial permutations of $\cQ$ which induce the
trivial permutation on $\PP'$ are the double transpositions. If
$\Delta(A,B)<0$, then $\cQ$ consists of two points $Q_1$ and $Q_2$
defined over $\R$ and a pair $Q_3$ and $Q_4$ of complex conjugate
points. In this case, the only nontrivial permutation of $\cQ$ which
induces the trivial permutation on $\PP'$ is the one switching $Q_1$
with $Q_2$ and $Q_3$ with~$Q_4$. This concludes the proof of the
lemma.
\end{proof}

\vspace{-.085in}
\begin{lemma}\label{lemsubmonos}
Let $f(x,y)\in U(\R)$ be a binary cubic form with $\Delta(f)\neq 0$ and
positive $x^3$-coefficient. 
\begin{itemize}
\item[{\rm (a)}] If $\Delta(f)>0$, then the set of elements
  in $V(\R)_+$ with resolvent $f$ consists of one $\SL_3(\R)$-orbit
  with splitting type $(1111)$ and three $\SL_3(\R)$-orbits with
  splitting type $(22)$.
\item[{\rm (b)}] If $\Delta(f)<0$, then the set of elements
  in $V(\R)_+$ with resolvent $f$ consists of a single $\SL_3(\R)$-orbit
  with splitting type $(112)$.
\end{itemize}
\end{lemma}
\begin{proof}
In Case (a), Theorem \ref{thpid} implies that the set of
$\SL_3(\R)$-orbits on $V(\R)_+$ with resolvent $f$ are in bijection
with the set of \'etale quartic extensions of $\R$ having cubic resolvent
algebra $\R^3$ along with a choice of basis for $\R^3$. Theorem
\ref{thHeilbronPID} implies that the latter set is in bijection with
the set of \'etale quadratic extensions of $\R^3$ whose discriminants
have square norm. There are four such extensions: first, each 
$\R$-factor can split, yielding the algebra $\R^6$ corresponding to
elements in $V(\R)$ with splitting type $(1111)$. Second, exactly one
of the $\R$-factors can split, yielding three different extensions
$\R^2\times\C^2$ each corresponding to elements in $V(\R)$ with
splitting type $(22)$.

Similarly, in Case (b), the set of $\SL_3(\R)$-orbits on $V(\R)_+$ with
resolvent $f$ are in bijection with \'etale quadratic extensions of
$\R\times\C$ whose discriminants have square norm in $\R$. There is only one
such extension, namely, $\R^2\times\C^2$ corresponding to elements in
$V(\R)$ with splitting type $(112)$.
\end{proof}

\vspace{-.085in}
\begin{rem}\label{three22orbits}{\em 
When $\Delta(f)>0$, we can describe the three
$\SL_3(\R)$-orbits in $V(\R)_+$ with splitting type $(22)$ as follows. An element $(A,B)\in V(\R)_+$ has splitting type $(22)$ if and only if $A$ and $B$ have no common zeros in $\P^2(\R)$. This can occur in three ways:
\begin{itemize}
\item[(i)] $A$ has no zeros in $\P^2(\R)$, i.e., $A$ is anisotropic.
\item[(ii)] $A$ has zeros in $\P^2(\R)$, i.e., $A$ is isotropic, and $B$ takes only {\it positive} values on the zeros of~$A$ in $\P^2(\R)$.
\item[(iii)] $A$ has zeros in $\P^2(\R)$, i.e., $A$ is isotropic, and $B$ takes only {\it negative} values on the zeros of~$A$ in $\P^2(\R)$.
\end{itemize}\pagebreak
Note that the conditions (ii) and (iii) are disjoint because, if $B$
took both positive and negative values on the zeros of $A$, then by
the intermediate value theorem, $A$ and $B$ would have a common zero
in $\P^2(\R)$---a contradiction.
}\end{rem}
\begin{nota}
\label{sec:3.19}
{\em
The conditions (i), (ii), and (iii) on $(A,B)\in V(\R)_+$ in Remark~\ref{three22orbits} correspond
exactly to the three orbits in Lemma~\ref{lemsubmonos}(a)
having splitting type (22); we denote these three orbits in $V(\R)_+$
by $V(\R)^{(2\#)}$, $V(\R)^{(2+)}$, and $V(\R)^{(2-)}$,
respectively. Thus $V(\R)^{(2)}=V(\R)^{(2\#)}\cup V(\R)^{(2+)}\cup
V(\R)^{(2-)}$.
For any $i\in\{0,1,2,2\#,2+,2-\}$ and any subset $L\subset
V(\R)_+$, let $L^{(i)}$ denote the set $L\cap V(\R)^{(i)}$.
}\end{nota}
\begin{cons}
\label{sec:3.20}
{\em
Suppose $2\Y<cX^\delta$, and recall the sets $\smash{\FF_U^\pm}(\ycox)\subset U(\R)$
defined in~(\ref{definesets}).~Let
\begin{equation*}
\kappa=\kappa(\ycox):=\smash{\bigl\lfloor X^{1/6}/\Y^{2/3}\bigr\rfloor}.
\end{equation*}
Since $2\Y<cX^\delta\ll X^{1/4}$, we have $\kappa\gg 1$. Let $\smash{\widetilde{\FF}_U^\pm}(\ycox)$ denote the 
$\kappa$-fold cover 
\begin{equation*}
  \widetilde{\FF}_U^\pm(\ycox):=\bigcup_{0\leq k\leq\kappa}\left(\begin{matrix}1 & \\ k & 1\end{matrix}\right)
  \cdot
  \FF_U^\pm(\ycox)
\end{equation*}
}\end{cons}
of~$\smash{\FF_U^\pm}(\ycox)$.
The coefficients of elements $f\in \smash{\FF_U^\pm}(\ycox)$ satisfy the
bounds \eqref{eqFUYX}. It follows that the coefficients of an element
$ax^3+bx^2y+cxy^2+dy^3\in \smash{\widetilde{\FF}_U^\pm}(\ycox)$ satisfy:
\begin{equation}\label{eqFUYXcov}
|a|\ll \Y;\quad 0\leq b\ll X^{1/6}\Y^{1/3};\quad |c|\ll X^{1/3}/\Y^{1/3};
\quad |d|\ll X^{1/2}/\Y.  
\end{equation}

\vspace{.05in}
We now describe a fundamental set for the action of
$\SL_3(\R)$ on elements in $V(\R)_+$ with
resolvent in $\smash{\widetilde{\FF}_U^\pm}(\ycox)$.

\begin{proposition}\label{propsmfssize}
There exist continuous maps
\begin{equation*}
\begin{array}{rcl}
s:\smash{U(\R)^+}&\to& \smash{V(\R)^{(i)}}\quad \mathrm{ for }\quad i\in\{0,\smash{2\#},2+,2-\},\\[.035in]
s:U(\R)^-&\to& V(\R)^{(i)}\quad \mathrm{ for }\quad i=1,\\[-.05in]
\end{array}
\end{equation*}
satisfying:
\begin{itemize}
\item[{\rm (a)}] The resolvent cubic form of $s(f)$ is $f$, i.e., $s$ gives a section of the cubic resolvent map $V(\R)^{(i)}\to U(\R)^\pm$.
\item[{\rm (b)}] 
If $f\in\widetilde{\FF}_U^\pm(\ycox)$, and $s(f)=(A,B)$ with $A=(a_{ij})$ and $B=(b_{ij})$, 
then 
\begin{equation}\label{sfbounds}
|a_{ij}|\leq \Y^{1/3} {\rm{ \;\;\: and \;\;\: }} |b_{ij}|\leq X^{1/6} / \Y^{1/3}.
\end{equation}
\end{itemize}
\end{proposition}
\begin{proof}
Lemma \ref{lemsubmonos} implies that sections $s:U(\R)^\pm\to V(\R)^{(i)}$ exist, and it suffices to prove that the section $s$ can be chosen to satisfy the bounds of~(b). 

Let $g(x,y)=f(\Y^{-1/3}x,X^{-1/6}\Y^{1/3}y)$. Then by
(\ref{eqFUYXcov}), the absolute values of all coefficients of $g$ are
$\leq 1$.  A section $t\in V(\R)$ can be constructed on elements $h\in
U(\R)$ having absolutely bounded coefficients so that all the entries
of $t(h)$ have size $O(1)$; for example, when $i=0$ or $i=1$, we may
take the section:
\begin{equation}\label{tsection}
t \,:\, ax^3+bx^2y+cxy^2+dy^3\mapsto
\left(\left[
\begin{array}{ccc}
&&1/2\\&-a&\\1/2&&
\end{array}\right],
\left[
\begin{array}{ccc}
&1/2&\\1/2&-b&c/2\\&c/2&-d
\end{array}\right]
\right).
\end{equation}
Writing $t(g)=(A,B)$, we set $s(f)=(\Y^{1/3}A,X^{1/6}\Y^{-1/3}B);$ then $s(f)$ satisfies the bounds (\ref{sfbounds}).  
\end{proof}

\vspace{-.085in}
\begin{cons}
\label{sec:3.22}
{\em
  For $i\in\{0,1,2\#,2+,2-\}$, let $\FF_V^{(i)}(\ycox)$ (resp.\ $\widetilde{\FF}_V^{(i)}(\ycox)$,
  $\FF_V^{(i)}(\cdeltaint,X)$) denote the image of
$\FF_U^\pm(\ycox)$ (resp.\ $\widetilde{\FF}_U^\pm(\ycox)$, $\FF_U^\pm(\cdeltaint,X)$)
under the map $s:U(\R)^\pm\to V(\R)^{(i)}$ of
Proposition~\ref{propsmfssize}. Let
\begin{equation*}
\begin{array}{rcl}
\smash{\displaystyle\FF_V^{(2)}(\ycox)}&:=&\smash{\displaystyle\FF_V^{(2+)}(\ycox)\cup\FF_V^{(2-)}(\ycox)\cup\FF_V^{(2\#)}(\ycox);}
\\[.125in]
\smash{\displaystyle\widetilde{\FF}_V^{(2)}(\ycox)}&:=&\smash{\displaystyle\widetilde{\FF}_V^{(2+)}(\ycox)\cup\widetilde{\FF}_V^{(2-)}(\ycox)\cup\widetilde{\FF}_V^{(2\#)}(\ycox);}
\\[.125in]
\smash{\displaystyle\FF_V^{(2)}(\cdeltaint,X)}&:=&\smash{\displaystyle\FF_V^{(2+)}(\cdeltaint,X)\cup \FF_V^{(2-)}(\cdeltaint,X)\cup \FF_V^{(2\#)}(\cdeltaint,X).}
\end{array}
\end{equation*}
Let $\FF_{\SL_3}$ be a
fundamental domain for the action of $\SL_3(\Z)$ on $\SL_3(\R)$ by left multiplication, and~let
\begin{equation}\label{eqmi}
m_0=m_2=m_{2\pm}=m_{2\#}=4;\quad m_1=2.
\end{equation}
}\end{cons}

\begin{proposition}\label{propsmfd}
  For $i\in \{0,1,2,2\#,2+,2-\}$, the multiset $\FF_{\SL_3}\cdot
  \widetilde{\FF}_V^{(i)}(\ycox)$ is an $m_i\kappa$-fold cover
  of a fundamental domain for the action of $\MB(\Z)\times\SL_3(\Z)$ on
  the set of elements $v\in V(\R)^{(i)}$ with $\Y\leq\ind(v)<2\Y$ and
  $X\leq H(v)<2X$.
\end{proposition}

\begin{cons}
\label{sec:3.24}
{\em
We now fix $\FF_{\SL_3}$ to lie in a {\bf Siegel domain} as in
\cite{BH}, i.e., every element $\gamma\in \smash{\FF_{\SL_3}}$ can be
expressed in Iwasawa coordinates in the form $\gamma=n tk$, where $n$
is a lower triangular unipotent matrix with coefficients bounded by
$1$ in absolute value, $k$ belongs to the compact group $\SO_3(\R)$,
and $t=t(s_1,s_2)$ is a diagonal matrix having diagonal entries
$s_1^{-2}s_2^{-1}$, $s_1s_2^{-1}$, $s_1s_2^2$, with $s_1,s_2\gg 1$. A
Haar measure $d\gamma$ on $\SL_3(\R)$ in these coordinates is
$(s_1s_2)^{-6}d\nu\,d^\times s_1\, d^\times s_2 \,dk$ (for a proof,
see \cite[Proposition 1.5.3]{GoldBook}).
}\end{cons}

\subsubsection{Averaging and cutting off the cusp}

Let $L\subset V(\Z)$ be a finite union of translates of lattices 
that is $\MB(\Z)\times\SL_3(\Z)$-invariant. We now determine
the asymptotic number of $\MB(\Z)\times\SL_3(\Z)$-orbits on
$L^\gen$ having bounded height $H$ and index bounded by $cH^\delta$. 

\begin{nota}\label{N:3.25}{\em
Let $N^{(i)}(L;\ycox)$ denote the number of
$\MB(\Z)\times\SL_3(\Z)$-orbits on the set of elements $v\in L^\gen\cap \smash{V(\R)_+^{(i)}}$ such
that $\Y\leq\ind(v)<2\Y$ and $X\leq H(v)<2X$.
}
\end{nota}

Then we prove the following theorem.

\begin{theorem}\label{thmsmqc}
We have
\begin{equation*}
N^{(i)}(L;\ycox)=
\frac{1}{m_i\kappa}\nu(L)\,
\Vol\bigl(\FF_{\SL_3}\cdot \widetilde{\FF}_V^{(i)}(\ycox)\bigr)
+O_\epsilon(X^{5/6+\epsilon}\Y^{1/3}).
\end{equation*}
\end{theorem}
\begin{proof}
By Proposition \ref{propsmfd}, it follows that 
\begin{equation*}
N^{(i)}(L;\ycox)=\frac{1}{m_i\kappa}\#\bigl\{
\FF_{\SL_3}\cdot \widetilde{\FF}_V^{(i)}(\ycox)
\cap L^\gen\bigl\}.
\end{equation*}
Let $G_0\subset \SL_3(\R)$ be a nonempty open bounded set. Then,
using the identical averaging argument as in \cite[Theorem 2.5]{BhSh}, we have
\begin{equation}\label{eqsmmtavg}
\displaystyle \#\bigl\{
\FF_{\SL_3}\cdot \widetilde{\FF}_V^{(i)}(\ycox)
\cap L^\gen\bigl\}
=
\displaystyle
\frac{1}{\Vol(G_0)}
\int_{\gamma\in\FF_{\SL_3}}\#\bigl\{
\gamma G_0\cdot \widetilde{\FF}_V^{(i)}(\ycox)
\cap L^\gen\bigl\}d\gamma.
\end{equation}
The remainder of the proof now follows exactly the arguments of \pagebreak
\cite[\S3]{dodqf}. First, proceeding as in the proofs of
\cite[Lemma 11]{dodqf} and \cite[Lemmas~12--13]{dodqf}, we 
obtain the following two estimates, respectively:
\begin{equation*}
\begin{array}{rcl}
\displaystyle\int_{\gamma\in\FF_{\SL_3}}\#\bigl\{(A,B)\in
\gamma G_0\cdot \widetilde{\FF}_V^{(i)}(\ycox)
\cap L^\gen:a_{11}=0\bigl\}d\gamma&\ll& X/\Y^{1/3};
\\[.15in]
\displaystyle\int_{\gamma\in\FF_{\SL_3}}\#\bigl\{(A,B)\in
\gamma G_0\cdot \widetilde{\FF}_V^{(i)}(\ycox)
\cap (L\setminus L^\gen):a_{11}\neq 0\bigl\}d\gamma&\ll_\epsilon&
X^{1+\epsilon}/\Y^{1/3}. \\[-.05in]
\end{array}
\end{equation*}
Next, since Proposition \ref{propsmfssize} implies that the
coefficients $a_{ij}$ of elements in
$\widetilde{\FF}_V^{(i)}(\ycox)$ are bounded by
$O(\Y^{1/3})$, it follows that the set
\begin{equation*}
\bigl\{(A,B)\in
\gamma G_0\cdot \widetilde{\FF}_V^{(i)}(\ycox)
\cap V(\Z):a_{11}\neq 0\bigl\}
\end{equation*}
is empty unless $\smash{s_1^{-4}s_2^{-2}}\Y^{1/3}\gg 1$, or equivalently,
$\smash{s_1^4s_2^2}\ll \Y^{1/3}$. Carrying out the integral in~\eqref{eqsmmtavg}, cutting it off when $\smash{s_1^4s_2^2}\ll
\Y^{1/3}$, and applying Proposition \ref{davlem} to estimate the number
of lattice points in $\gamma G_0\cdot
\smash{\widetilde{\FF}_V^{(i)}}(\ycox)$, we obtain 
\begin{equation}\label{eqsmmtavg1}
\begin{array}{rcl}
\displaystyle \#\bigl\{
\FF_{\SL_3}\cdot \widetilde{\FF}_V^{(i)}(\ycox)
\cap L^\gen\bigl\}&=&
\displaystyle
\frac{\nu(L)}{\Vol(G_0)}
\int_{\gamma\in\FF_{\SL_3}}\Vol\bigl(
\gamma G_0\cdot \widetilde{\FF}_V^{(i)}(\ycox)\bigr)d\gamma
+O_\epsilon\bigl(X^{1+\epsilon}\Y^{-1/3}\bigr)
\\[.2in]&=&
\displaystyle\frac{\nu(L)}{\Vol(G_0)}
\int_{\gamma\in G_0}\Vol\bigl(
\FF_{\SL_3}\gamma\cdot \widetilde{\FF}_V^{(i)}(\ycox)\bigr)d\gamma
+O_\epsilon\bigl(X^{1+\epsilon}\Y^{-1/3}\bigr)
\\[.2in]&=&
\displaystyle\nu(L)\cdot\Vol\bigl(
\FF_{\SL_3}\cdot \widetilde{\FF}_V^{(i)}(\ycox)\bigr)
+O_\epsilon\bigl(X^{1+\epsilon}\Y^{-1/3}\bigr).
\end{array}
\end{equation}
Finally, we note that by Proposition \ref{propsmfd}, the multiset
$\FF_{\SL_3}\cdot \smash{\widetilde{\FF}_V^{(i)}}(\ycox)$ contains exactly
$m_i\,\kappa$ representatives of an
$\MB(\Z)\times\SL_3(\Z)$-orbit $v\in L^\gen\cap V(\R)^{(i)}$. Dividing
the first and last terms of \eqref{eqsmmtavg1} by
$m_i\,\kappa$ thus yields Theorem \ref{thmsmqc}.
\end{proof}

\vspace{-.005in}\noindent
{\bf Proof of Theorem~\ref{thqsubcount}:} Theorem
\ref{thqsubcount} now follows immediately by breaking the intervals $[1,X]$ and
$[1,cX^\delta]$ into dyadic ranges, and then applying
Theorem \ref{thmsmqc} to each pair of dyadic ranges. $\Box$

\subsection{Uniformity estimates and squarefree sieves}\label{subsecsubunif}

In this subsection, we count $\MB(\Z)$-orbits on $U(\Z)$ and
$\MB(\Z)\times\SL_3(\Z)$-orbits on $V(\Z)$ of bounded height and index
satisfying certain infinite sets of congruence conditions.

\begin{nota}
\label{sec:3.26}
{\em
For  a large collection $S=(S_p)_p$ of cubic local specifications, 
let $U(\Z)_S$ denote the set of elements $f\in U(\Z)$ such that $f$
belongs to $S_p$ for all $p$, and let $V(\Z)_S$ denote the set
of elements in $V(\Z)$ whose cubic resolvent forms are in
$U(\Z)_S$.
}\end{nota}
We prove the following theorem:

\begin{theorem}\label{thsubsieve}
For a set $L\subset U(\Z)$ $($resp.\ $L\subset V(\Z))$ defined via
congruence conditions, let $\nu(L)$ denote the volume of the closure
of $L$ in \smash{$U(\widehat{\Z})$} $($resp.\ \smash{$V(\widehat{\Z}))$}. Then 
\begin{equation*}
\begin{array}{rcl}
\displaystyle \smash{N^{\pm}(U(\Z)^\maxx_S;\ycox)}
&=&\displaystyle
\;\;\;\;\displaystyle 
\;\,\smash{\nu\bigl(U(\Z)^\maxx_S\bigr)}\,\Vol\bigl(
\FF_U^\pm(\ycox)\bigr)+o(X^{5/6}\Y^{2/3});
\\[.05in]
\displaystyle N^{(i)}(V(\Z)^\maxx_S;\ycox)&=&
\displaystyle \frac{1}{m_i
}\, \nu\bigl(V(\Z)^\maxx_S\bigr)
\,\Vol\bigl(\FF_{\SL_3}\cdot 
\,{\FF}_V^{(i)}(\ycox)\bigr)
+o(X^{5/6}\Y^{2/3}).
\end{array}
\end{equation*}
\end{theorem}

To prove the first claim of Theorem~\ref{thsubsieve}, we must impose
infinitely many congruence conditions on $U(\Z)$. Namely, we must
impose the congruence conditions of maximality at every prime as well as the congruence conditions forced by $\Sigma$.  For this, we require 
estimates on the number of $\MB(\Z)$-orbits in $U(\Z)$ having bounded height
and index with discriminant divisible by the square of \pagebreak some 
prime $p>M$.   For a prime $p$ and an element $f\in U(\Z)$, note that
$p^2\mid\Delta(f)$ precisely when the cubic ring corresponding to $f$
is {totally ramified} (i.e., has splitting
type $(1^3)$) at~$p$ or is non-maximal~at~$p$. 

Similarly, to prove the second claim of Theorem~\ref{thsubsieve}, we
require estimates on the number of $\MB(\Z)\times\SL_3(\Z)$-orbits on
$V(\Z)$ having bounded height and index with discriminant divisible by
the square of some prime $p>M$. This time, note that
$p^2\mid\Delta(v)$ for $(A,B)\in V(\Z)$ precisely when $(A,B)$
corresponds to a quartic ring that is either {\bf extra ramified}
(i.e., has splitting type $(1^31)$, $(1^21^2)$, $(2^2)$, or $(1^4)$)
at~$p$ or is nonmaximal at $p$.

\begin{nota}
\label{sec:3.28}
{\em
For a prime $p$, let
$\smash{\W_p^{(1)}}(U)$ denote the set of integer-coefficient binary cubic
forms with splitting type \smash{$(1^3)$} at $p$, and $\smash{\W_p^{(2)}}(U)$
the set of integer-coefficient binary cubic forms that are
nonmaximal at $p$.  Similarly, let \smash{$\W_p^{(1)}(V)$} denote the set of
elements in $V(\Z)$ with splitting type \smash{$(1^31)$, $(1^21^2)$,
$(2^2)$, or $(1^4)$} at $p$, and $\smash{\W_p^{(2)}}(V)$ the set of
elements in $V(\Z)$ that are nonmaximal at~$p$. Finally, set
$\W_p(U):=\smash{\W_p^{(1)}}(U)\cup \smash{\W_p^{(2)}}(U)$ and
$\W_p(V):=\smash{\W_p^{(1)}}(V)\cup \smash{\W_p^{(2)}}(V)$.

In the language of \cite[\S1.5]{geosieve}, the sets \smash{$\W_p^{(i)}(U)$}
and \smash{$\W_p^{(i)}(V)$} consist of elements in $U(\Z)$ and $V(\Z)$,
respectively, whose discriminants are divisible by $p^2$ for ``mod
$p^i$ reasons''.}
\end{nota}

We begin by bounding the number of $n$-monogenized cubic rings
(resp.\ quartic rings with $n$-mon\-ogenized cubic resolvent rings) that
are totally ramified (resp.\ extra ramified) at some prime~$p>M$. 

\begin{proposition}\label{propsmusgs}
For any positive real numbers $X$, $\Y$, and $M$ with $\Y\ll X^{1/4}$, we have
\begin{equation*}
\begin{array}{rcl}
\displaystyle
\sum_{p>M}N^{\pm}(\smash{\W_p^{(1)}}(U);\ycox)
&=&\displaystyle
O_\epsilon\bigl({X^{5/6}\Y^{2/3}}/{M^{1-\epsilon}}\bigr)
+O_\epsilon(X^{5/6+\epsilon});
\\[.215in]
\displaystyle\sum_{p>M}
N^{(i)}(\smash{\W_p^{(1)}}(V);\ycox)&=&
\displaystyle O_\epsilon\bigl({X^{5/6}\Y^{2/3}}/{M^{1-\epsilon}}\bigr)
+O(X^{5/6}\Y^{1/3}). \\[-.05in]
\end{array}
\end{equation*}
\end{proposition}

\begin{proof}
Theorem~\ref{propsmusgs} follows immediately by 
\cite[Theorem 3.3]{geosieve} and our counting results
(Theorems~\ref{thsubcount} and \ref{thqsubcount}) of the previous two subsections.
\end{proof}

Next, we obtain bounds on the number of orbits on $U(\Z)$ and $V(\Z)$
of bounded height and index that correspond to nonmaximal rings. 

\begin{nota}
\label{sec:3.30}
{\em
For a cubic number field $K$, we define the following {\bf height} function $H$ on its ring of integers $\O_K$: 
for $\beta\in\O_K$,
let
$H(\beta):=\max_v\bigl\{|\beta'|_v\bigr\},$
where $v$ ranges over the archimedean valuations of $K$ and $\beta'$
denotes the unique $\Z$-translate of $\beta$ such that the absolute trace ${\rm
  Tr}(\beta')\in\{0,1,2\}$.
}\end{nota}

Let $\alpha\in \O_K$. Since either $\alpha'$, $\alpha'-1/3$, or $\alpha'-2/3$ satisfies the equation
\begin{equation}\label{ijeq}
x^3 - (I(f)/3) x - J(f)/27  = 0, 
\end{equation}  
we have the estimate
\begin{equation}\label{htest}
H(\alpha)\ll 
\max\{|I(f)|^3,J(f)^2/4\}^{1/6} = (n^{2}H(f))^{1/6}=\ind(f)^{1/3}H(f)^{1/6}. 
\end{equation}
Indeed, if $H(\alpha)\gg \ind(f)^{1/3}H(f)^{1/6}$, then the left-hand side of (\ref{ijeq}) could not equal zero as the $x^3$-term would dominate the other two terms.

\begin{defn}
\label{sec:3.31}
{\em Let $C$ be a nondegenerate cubic ring, and consider $C$ embedded as a
lattice in $C\otimes\R\cong\R^3$ with covolume $\sqrt{\Delta(C)}$.
Let $1\leq\ell_1(C)\leq\ell_2(C)$ denote the three successive minima
of~$C$. We define the {\bf skewness} of a cubic ring $C$ to be}
\begin{equation*}
{\rm sk}(C):=\ell_2(C)/\ell_1(C).
\end{equation*}
\end{defn}

As in \cite[Lemma 4.7]{SSW}, the number of integers $\alpha\in C$ with
bounded height can be controlled by the discriminant and skewness of
$C$. \pagebreak
\begin{lemma}\label{lemcountint}
Let $C$ be a cubic ring with discriminant $D=\Delta(C)$ and skewness $Z={\rm sk}(C)$. Then
the number of elements $\alpha\in C\setminus\Z$ with ${\rm
  Tr}(\alpha)\in\{0,1,2\}$ and $H(\alpha)< H$ is 
\begin{equation*}
\ll\left\{
\begin{array}{ccl}
  0 &\mbox{if}& H< \ell_1(C)\asymp D^{1/4}/Z^{1/2};\\
  HZ^{1/2}/D^{1/4} &\mbox{if}& \ell_1(C)< H<
  \ell_2(C)\asymp D^{1/4}Z^{1/2};\\
  H^2/D^{1/2} &\mbox{if}& \ell_2(C)< H.
\end{array}\right.
\end{equation*}
\end{lemma}
\begin{proof}
The definition of ${\rm sk}(C)$ and the fact that
$\ell_1(C)\ell_2(C)\asymp\sqrt{D}$ imply that $\ell_1(C)\asymp
D^{1/4}/Z^{1/2}$ and $\ell_2(C)\asymp D^{1/4}Z^{1/2}$. When
$H<\ell_1(C)$, there are no elements $\alpha\in C\setminus\Z$ with
$H(\alpha)<H$. When $\ell_1(C)<H<\ell_2(C)$, the only elements
$\alpha\in C\setminus\Z$ with ${\rm Tr}(\alpha)\in\{0,1,2\}$ are
(appropriately translated) multiples of the element in $C$
of length $\ell_1(C)$. When $H>\ell_2(C)$, the
number of such $\alpha$ is then bounded by the volume of the region
$\{\theta\in C\otimes\R:H(\theta)<H \mbox{ and } \Tr(\theta)\in\{0,1,2\}\}$
divided by $\sqrt{D}$, the covolume of $C$.
\end{proof}

\vspace{-.085in}
\begin{proposition}\label{propskringscount}
The number of cubic orders $C$ $($resp.\ pairs $(Q,C)$, where $Q$ is a
quartic order and $C$ is a cubic resolvent order of $Q)$ satisfying
$D\leq|\Delta(C)|<2D$ and ${\rm sk}(C)\gg Z$ is $O(D/Z)$.
\end{proposition}

\begin{proof}
Proposition~\ref{propskringscount} follows from the proofs of
\cite[Theorem 4.1 and Proposition~4.4]{SSW}.
\end{proof}

\vspace{-.085in}
\begin{corollary}\label{corcountint}
The number of pairs $(C,\alpha)$, where $C$ is a cubic order
$($resp.\ triples $(Q,C,\alpha)$, where $Q$ is a quartic order and $C$
is a cubic resolvent order of $Q)$ satisfying $|\Delta(C)|<X$,
$\alpha\in C\setminus\Z$ with ${\rm Tr}(\alpha)\in\{0,1,2\}$, and
$H(\alpha)\ll H$ is $O_\epsilon(H^2X^{1/2+\epsilon})$.
\end{corollary}
\begin{proof}
We break up the discriminant range $[0,X]$ of $C$ into $O(\log X)$
dyadic ranges. For each such dyadic range $[D,2D]$, we break up the
range of possible skewness of $C$ into dyadic ranges $[Z,2Z]$, where
clearly $Z\ll D^{1/2}$. For a cubic order $C$ with $|\Delta(C)|\asymp
D$ and ${\rm sk}(C)\asymp Z$, Lemma \ref{lemcountint} implies that the
number of elements $\alpha\in C\setminus\Z$ with ${\rm
  Tr}(\alpha)\in\{0,1,2\}$ and $H(\alpha)\ll H$ is bounded by
\begin{equation*}
  \ll\left\{ \begin{array}{cl}
    \displaystyle H^2/D^{1/2}+HZ^{1/2}/D^{1/4}
    &\mbox{if } \displaystyle Z\gg D^{1/2}/H^2;\\[.02in]
    0 & \mbox{otherwise}.
  \end{array}\right.
\end{equation*}
Adding up over the $O(D/Z)$ orders with discriminant and skewness in
this range, gives us the following bound on the number of pairs $(C,\alpha)$
(resp.\ triples $(Q,C,\alpha)$):
\begin{equation*}
  \ll\left\{ \begin{array}{cl}
    \displaystyle D^{1/2}H^2/Z+HD^{3/4}/Z^{1/2}
    &\mbox{if } \displaystyle Z\gg D^{1/2}/H^2;\\[.02in]
    0 &\mbox{otherwise.}
  \end{array}\right.
\end{equation*}
In either case, the bound is $O(H^2D^{1/2})$. Adding up over
the dyadic ranges yields the result.
\end{proof}

We are now ready to prove the other required uniformity estimate.

\begin{proposition}\label{propunifsub2}
For any positive real numbers $X$, $\Y$, and $M$ with $\Y\ll X^{1/4}$, we have
\begin{equation*}
\begin{array}{rcl}
\displaystyle
\sum_{p>M}N^{(i)}(\smash{\W_p^{(2)}}(U);\ycox)&=&
\displaystyle
O_\epsilon\bigl({X^{5/6}\Y^{2/3}}/{M^{1-\epsilon}}\bigr);
\\[.215in]
\displaystyle\sum_{p>M}N^{(i)}(\smash{\W_p^{(2)}}(V);\ycox)&=&
\displaystyle O_\epsilon\bigl({X^{5/6}\Y^{2/3}}/{M^{1-\epsilon}}\bigr).\\[-.05in]
\end{array}
\end{equation*}
\end{proposition}
\begin{proof}
If a prime $p$ satisfies $p\ll X^{1/8}$ in the first 
sum, then we have the following upper bound on the corresponding summand:
\begin{equation}\label{smallp}
  N^{(i)}(\smash{\W_p^{(2)}}(U);\ycox)=\#\bigl\{\FF_U^\pm(\ycox)
  \cap\smash{\W_p^{(2)}}(U)\bigr\}\ll {X^{5/6}\Y^{2/3}}/{p^2}. \pagebreak 
\end{equation}
This can be seen by simply partitioning $\smash{\W_p^{(2)}}(U)$ into $O(p^6)$
translates of $p^2 U(\Z)$ and counting integer points of bounded
height in each translate using Proposition~\ref{davlem}, noting that
the last variable $d$ takes a range of length at least $X^{1/2}/\Y\gg
X^{1/4}\gg p^2$. The sum of the bound \eqref{smallp} over all primes
$p$ with $M < p < X^{1/8}$ is less than the bound of 
Proposition~\ref{propunifsub2}.  We may thus assume that $M\gg
X^{1/8}$ for the purpose of proving the first estimate of 
the proposition. Similarly, since the variable $b_{33}$ takes a range
of at least $X^{1/6}/\Y^{1/3}\gg X^{1/12}$ in the proof of Theorem
\ref{thmsmqc}, we may assume that $M\gg X^{1/24}$ for the purpose of
proving the second estimate.

Let $f$ be an irreducible element of
\begin{equation}\label{equniftemp}
\bigcup_{p>M}\bigl\{\FF_U^\pm(\ycox)\cap\smash{\W_p^{(2)}}(U)\bigr\}.
\end{equation}
Then $f$ corresponds to a pair $(C,\alpha)$, where $C$ is a cubic ring
and $\alpha$ is an element of $C$. Furthermore, $C$ is nonmaximal at
some prime $p>M$. Let $C_1$ be the (unique) order containing $C$ with
index $p$. The treatment of the case $n=3$ in
\cite[\S4.2]{geosieve} implies that $C_1$ comes from at most $3$ such
rings $R$. It thus follows from \eqref{htest} that the set
\eqref{equniftemp} maps into the set
\begin{equation}\label{eqKal}
  \bigl\{(C_1,\alpha):|\Delta(C_1)|\ll X/M^2,\;H(\alpha)\asymp
  X^{1/6}\Y^{1/3}\bigr\},
\end{equation}
where $K$ is a cubic field and $\alpha$ is an element in $\O_K$ with
${\rm Tr}(\alpha)\in\{0,1,2\}$. Moreover, the sizes of the fibers of
this map are bounded by $3$. Therefore, we have
\begin{equation}\label{equnifRal}
  \sum_{p>M}N^{(i)}(\smash{\W_p^{(2)}}(U);\ycox)\ll
  \#\bigl\{(C_1,\alpha):|\Delta(C_1)|
  \ll X/M^2,\;H(\alpha)\asymp X^{1/6}\Y^{1/3}\bigr\}.
\end{equation}

Similarly, let $v\in V(\Z)$ be an element contributing to the count of
the left-hand side of the second line of the proposition. Then the
cubic resolvent of $v$ belongs to \eqref{equniftemp}, and thus $v$
corresponds to a triple $(Q,C,\alpha)$, where $Q$ is a quartic ring,
$C$ is a cubic resolvent of $Q$ nonmaximal at a prime $p>M$, and
$\alpha$ is an element of $C$. Let $C_1$ be as in the previous paragraph. The treatment
of the case $n=4$ in \cite[\S4.2]{geosieve} implies that the set of
quartic rings with resolvent $C$ maps to the set of quartic rings with
resolvent $C_1$, where the sizes of the fibers of this map are bounded
by $6$. Therefore, 
\begin{equation}\label{equnifQRal}
  \sum_{p>M}N^{(i)}(\smash{\W_p^{(2)}}(V);\ycox)\ll
  \#\bigl\{(Q_1,C_1,\alpha):|\Delta(C_1)|
  \ll X/M^2,\;H(\alpha)\asymp X^{1/6}\Y^{1/3}\bigr\},
\end{equation} 
where $Q_1$ is a quartic ring with resolvent $C_1$ and $\alpha\in C_1$
has trace in $\{0,1,2\}$.

Corollary \ref{corcountint} implies that the right-hand sides of \eqref{equnifRal} and \eqref{equnifQRal} are bounded by
\begin{equation*}
O_\epsilon\bigl({X^{5/6+\epsilon}\Y^{2/3}}/{M}\bigr).
\end{equation*}
This concludes the proof of the lemma, since $M\gg X^{1/24}$.
\end{proof}

\vspace{-.005in}\noindent
{\bf Proof of Theorem~\ref{thsubsieve}:} 
Theorem \ref{thsubsieve} follows immediately from the counting results
of Proposition \ref{propcountsmf} and Theorem~\ref{thmsmqc}, in
conjunction with the tail estimates in Propositions~\ref{propsmusgs}
and \ref{propunifsub2}, by applying a standard squarefree sieve. See 
\cite[\S8.5]{MAJ} for an identical proof. $\Box$

\subsection{Local mass formulas}\label{seclmsub}

To obtain the volumes in Theorem~\ref{thsubsieve},
we prove certain
mass formulas relating \'etale quartic and cubic algebras over $\Q_p$.
Let $K_3$ be an \'etale cubic extension of $\Q_p$. Let $K_4$ be an
\'etale quartic extension of $\Q_p$ with cubic resolvent $K_3$ that 
corresponds to the unramified quadratic extension $K_6/K_3$ \pagebreak
under the bijection of Theorem \ref{thHeilbronPID}. By 
Remark~\ref{geom}, the \'etale $\Q_p$-algebras $K_3$, $K_4$, and~$K_6$ correspond to sets~$\PP$, $\cQ$,
and~$\LL$, respectively, equipped with actions of~$G_{\Q_p}$. The 
automorphism groups $\Aut(K_3)$, $\Aut(K_4)$, and~$\Aut(K_6)$ can be
identified with the groups of $G_{\Q_p}$-equivariant permutations of
$\PP$, $\cQ$, and~$\LL$, respectively.
A $G_{\Q_p}$-equivariant permutation of $\cQ$ induces $G_{\Q_p}$-equivariant
permutations of~$\PP$ and~$\LL$. 

\begin{nota}
\label{sec:3.36}
{\em 
Let $\Aut_{K_3}(K_4)$ denote the subgroup
of $\Aut(K_4)$ consisting of automorphisms of $K_4$ that induce the
trivial automorphism of $K_3$ (equivalently, the subgroup consisting of Galois
equivariant permutations of $\cQ$ that induce 
the trivial permutation of~$\PP$). 
}\end{nota}
We have the following result.

\begin{theorem}\label{thmmassmain}
Let $K_3$ be an \'etale cubic extension of $\Q_p$, and let $\RR(K_3)$
denote the set of \'etale non-overramified quartic extensions $K_4$ of
$\Q_p$, up to isomorphism, with resolvent $K_3$. Then 
\begin{equation*}
\displaystyle\sum_{K_4\in\RR(K_3)}\frac{1}{|\Aut_{K_3}(K_4)|}
=1.
\end{equation*}
\end{theorem}
\begin{proof}
Recall that $\Aut_{K_3}(K_6)$ is the subgroup of $\Aut(K_6)$
consisting of automorphisms of $K_6$ that fix every element in $K_3$.
Let $\smash{\Aut_{K_3}^+(K_6)}$ denote the index $2$ subgroup of
$\Aut_{K_3}(K_6)$ consisting of even permutations of $\LL$. We may
check that the map from Galois-equivalent permutations of $\cQ$ to the
Galois-equivalent permutations of $\LL$ induces an isomorphism
$\Aut_{K_3}(K_4)\cong\smash{\Aut^+_{K_3}(K_6)}$.

Exactly half of the unramified quadratic extensions $K_6/K_3$ have discriminant whose norm is a square
in $\Z_p^\times$. Moreover, all of these quadratic extensions $K_6/K_3$ have isomorphic
automorphism groups. Therefore, we have
\begin{equation*}
\displaystyle\sum_{K_4\in\RR(K_3)}\frac{1}{|\Aut_{K_3}(K_4)|}
\,=\,
\displaystyle\frac{1}{2}\sum_{\substack{[K_{6}:K_3]=2\\{\rm unramified}}}
\frac{1}{|\Aut_{K_3}^+(K_6)|}
\,=\,
\sum_{\substack{[K_{6}:K_3]=2\\{\rm unramified}}}
\frac{1}{|\Aut_{K_3}(K_6)|}=1,
\end{equation*}
where the last equality follows from \eqref{eqmasstriv}.
\end{proof}

Finally, we translate Theorem \ref{thmmassmain} into the language of
binary cubic forms.

\begin{defn}
\label{sec:3.38}
{\em 
For $f\in U(\Z_p)$, define the {\bf local mass}
$\Mass_p(f)$ of $f$ by 
\begin{equation}\label{eqmasssub}
\Mass_{p}(f):=
\sum_{v\in\frac{\Res^{-1}(f)}{\SL_3(\Z_p)}}\frac1{\#\Stab_{\SL_3(\Z_p)}(v)},\\[-.05in]
\end{equation}}
where \smash{$\frac{\Res^{-1}(f)}{\SL_3(\Z_p)}$} is a
set of representatives for the action of $\SL_3(\Z_p)$ on $\Res^{-1}(f)$.
\end{defn}

Since the stabilizers in $M(\Q_p)$ and $G(\Q_p)$ of maximal elements in $U(\Z_p)$ and $V(\Z_p)$ 
actually lie in $M(\Z_p)$ and $G(\Z_p)$, respectively, we obtain the following consequence of Theorems
\ref{thpid} and~\ref{thmmassmain}:\!\!\!
\begin{corollary}\label{cortotmass}
Suppose $f\in U(\Z_p)$ corresponds to a maximal
cubic ring over~$\Z_p$. Then 
\vspace{-.05in}
\begin{equation*}
\displaystyle \Mass_p(f)
=1.
\end{equation*}
\end{corollary}

\subsection{Volume computations and proof of Theorem~\ref{thsubmon}}
\label{subsecsubvol}

In this subsection, we prove Theorem \ref{thsubmonloc}, from which
Theorem \ref{thsubmon} will be shown to follow.  To compute the
volumes of sets in $V(\R)$ and $V(\Z_p)$, we use the following
Jacobian change of variables.

\pagebreak

\begin{proposition}\label{propjacsub}
Let $K$ be $\R$ or $\Z_p$, let $\smash{|\cdot|}$ denote the usual normalized absolute
value on $K$, and let $s:U(K)\to V(K)$ be a continuous map
such that $\Res(s(f))=f$, for each $f\in U(K)$. Then there
exists a constant $\J\in\Q^\times$, independent of $K$ and $s$,
such that for any measurable function $\phi$ on $V(K)$, we have:
\begin{equation}\label{eqjacsub}
  \begin{array}{rcl}
 \!\! \! \displaystyle\int_{\SL_3(K)\cdot s(U(K))}\!\!\!\phi(v)dv&\!\!\!=\!\!\!&
|\J|\!\displaystyle\int_{f\in U(K)}\displaystyle\int_{g\in \SL_3(K)}
    \!\!\phi(g\cdot s(f))\omega(g) df,\\[0.25in]
\displaystyle\int_{V(K)}\!\!\phi(v)dv&\!\!\!=\!\!\!&
|\J|\!\displaystyle\int_{\substack{f\in U(K)\\ \Disc(f)\neq 0}}\!
\displaystyle\sum_{v\in\!\!\textstyle{\frac{\Res^{-1}(f)}{\SL_3(K)}}}
\!\!\!\frac{1}{\#\Stab_{\SL_3(\Z_p)}(v)}\int_{g\in \SL_3(K)}\!\!\!\phi(g\cdot v)\omega(g)df,\\[-.05in]
  \end{array}
\end{equation} 
where \smash{$\frac{\Res^{-1}(f)}{\SL_3(K)}$} is a
set of representatives for the action of $\SL_3(K)$ on $\Res^{-1}(f)$.
\end{proposition}

\begin{proof}
Proposition~\ref{propjacsub} follows immediately from the proofs of \cite[Propositions 3.11 and
  3.12]{BhSh} (see also \cite[Remark 3.14]{BhSh}). 
\end{proof}

\vspace{-.085in}
\begin{corollary}\label{corjacsub}
Let $S_p$ be an open and closed subset of $U(\Z_p)$ such that the
boundary of $S_p$ has measure $0$ and every element in $S_p$ is
maximal. Consider the set $\smash{V(\Z)_{S,p}:=\Res^{-1}(S_p)}$.
Then 
\begin{equation}
\Vol(V(\Z)_{S,p})=|\J|_p\Vol(\SL_3(\Z_p))\Vol(S_p).
\end{equation}
\end{corollary}

\begin{proof}
We set $\phi$ to be the characteristic function of $V(\Z)_{S,p}$ in
the second line of \eqref{eqjacsub} to obtain
\begin{equation*}
\Vol(V(\Z)_{S,p})=|\J|_p\Vol(\SL_3(\Z_p))\int_{f\in S_p}\Mass_p(f)df.
\end{equation*}
The result then follows from Corollary \ref{cortotmass}.
\end{proof}

\vspace{.025in}\noindent
{\bf Proof of Theorem \ref{thsubmonloc}:} Let $\Sigma$ be
associated to a large collection $S=(S_p)_p$ of local specifications,
where we may assume that every element $f(x,y)\in S_p$ is maximal. For
any finite set $T$ of $n$-monogenized cubic fields, let \smash{$\Avg(\Cl_2,T)$}
(resp.\ \smash{$\Avg(\Cl^+_2,T)$}) denote the average size of the $2$-torsion in
the class group (resp.\ narrow class group) of the fields in $T$.
Let \smash{$F^\pm_\Sigma(\cdeltaint,X)$} denote the set of elements \smash{$(K,\alpha)\in F_\Sigma(\cdeltaint,X)$} with
$\pm\Delta(K)>0$. By Theorem
\ref{thsubsieve} and Corollary~\ref{corjacsub}, 
\begin{equation*}
\begin{array}{rcl}
\displaystyle\Avg(\Cl_2,{F_\Sigma^+(\cdeltaint,X)})
&=&\displaystyle
1+\frac{\frac1{m_0}\nu(V(\Z)_S)\,\Vol(\FF_{\SL_3}\cdot\FF_{V}^{(0)}(\cdeltaint,X))}
{\Vol(\FF_{U}^{+}(\cdeltaint,X))\,\nu(V(\Z)_S)}+o(1)
\\[.175in]&=&\displaystyle
1+\frac{\frac{1}{m_0}|\J|\Vol(\FF_{\SL_3})
\prod_p\bigl[|\J|_p\Vol(\SL_3(\Z_p))\Vol(S_p)\bigr]}{\prod_p\Vol(S_p)}
  +o(1)
\\[.175in]\displaystyle
&=&\displaystyle 1+\frac{1}{m_0}+o(1)\;=\;\frac{5}{4}+o(1),
\end{array}
\end{equation*}
since $\Vol(\FF_{\SL_3})\cdot\prod_{p}\Vol(\SL_3(\Z_p))$ is equal to
$1$, the Tamagawa number of $\SL_3$. Similarly, 
\begin{equation*}
\begin{array}{rcccl}
\displaystyle\Avg(\Cl_2,{F_\Sigma^-(\cdeltaint,X)})
&=&\displaystyle 1+\frac{1}{m_1}+o(1)&=&\displaystyle\frac{3}{2}+o(1);\\[.15in]
\displaystyle\Avg(\Cl^+_2,{F_\Sigma^+(\cdeltaint,X)})
&=&\displaystyle 1+\frac{1}{m_0}+\frac{1}{m_{2+}}+\frac{1}{m_{2-}}+\frac{1}{m_{2\#}}+o(1)
&=&2+o(1). \qquad\Box
\end{array}
\end{equation*}

\medskip\noindent
{\bf Proof of Theorem \ref{thsubmon}:} 
Theorem~\ref{thsubmon} now follows immediately from Theorem~\ref{thsubmonloc} by letting $\Sigma_p$ consist of all pairs $(\mathcal K_p,\alpha_p)$, where $\mathcal K_p$ is an \'etale cubic extension of $\Q_p$ satisfying the splitting conditions prescribed in Theorem~3. 
$\Box$

\section{The mean number of $2$-torsion elements in the class
  groups of $n$-monogenized cubic fields ordered by height and fixed $n$}\label{secn}

In this section we fix a positive integer $n$ throughout, and prove
Theorem \ref{main_theorem} by computing, for such a {\it fixed} $n$, the
average size of the $2$-torsion subgroups in the class groups of
$n$-monogenized cubic fields when these fields are ordered by height.

In \S\ref{subsecncub}, we determine asymptotics for the number of
$n$-monogenized cubic rings of bounded height. In
\S\ref{subsecnquar}, we determine asymptotics for the number of
quartic rings having an $n$-monogenic cubic resolvent of bounded
height. In \S\ref{subsecnunif}, we prove uniformity estimates that
enable us to impose conditions of maximality on these counts of
cubic and quartic rings. The leading constants for these asymptotics
are expressed as products of volumes of sets in $U_n(R)$ and $V(R)$,
where $R$ ranges over $\R$ and $\Z_p$ for all primes $p$. In
\S\ref{seclmn}, we prove certain mass formulas relating \'etale
quartic and cubic algebras over $\Q_p$. Finally, in
\S\ref{subsecnlocal}, we use these mass formulas to compute the
necessary local volumes in order to prove our main results.

\subsection{The number of $n$-monogenized cubic rings of
  bounded height}\label{subsecncub}

In this subsection, we determine the number of primitive
$n$-monogenized cubic rings having bounded height that are orders in
$S_3$-number fields. Specifically, we prove the following result.

\begin{theorem}\label{thcubringcount}
Let $\smash{N_3^{+}}(n,X)$ $($resp.\ $\smash{N_3^{-}}(n,X))$ denote the number of
isomorphism classes of $n$-monogenized cubic orders in $S_3$-cubic fields with positive
$($resp.\ negative$)$ discriminant and height bounded by $X$. Then 
\begin{itemize}
\item[{\rm (a)}]
  $N_3^{+}(n, X)=\displaystyle\frac{8}{135n^{1/3}}X^{5/6}+O_\epsilon(X^{1/2+\epsilon});$
\item[{\rm (b)}]
  $N_3^{-}(n, X)=\displaystyle\frac{32}{135n^{1/3}}X^{5/6}+O_\epsilon(X^{1/2+\epsilon}).$
\end{itemize}
\end{theorem}
Theorem \ref{thcubringcount} is proved by
using the parametrizations of \S\ref{subsecparamrings} in conjunction
with a count of elements in $U_n(\Z)$ having bounded height.

\begin{nota}
\label{sec:4.2}
{\em
For any ring $R$ and integers $n$ and $b$, let $U_{n,b}(R)\subset
U_n(R)$ denote the set of binary cubic forms $f(x,y)\in U_n(R)$ such that the
$x^2y$-coefficient of $f$ is $b$. For a subset $S\subset U_{n}(\R)$,
let $S^\pm$ denote the set of elements $f\in S$ with
$\pm\Delta(f)>0$. For a subset $L\subset U_{n}(\Z)$, let $N^\pm(L;n,X)$
denote the number of generic $\MB(\Z)$-equivalence classes $f$ in $L^\pm$
such that $H(f)<X$.
}\end{nota}

\begin{lemma}\label{lred411}
The number of non-generic $\MB(\Z)$-orbits in $U_{n}(\Z)$ of
height less than $X$ is $O_\epsilon(X^{1/2+\epsilon})$.
\end{lemma}

\begin{proof}
An $\MB(\Z)$-orbit on $U_n(\Z)$ has a unique representative with
$x^2y$-coefficient in $[0,3n)$.  Let $f(x)=nx^3+bx^2y+cxy^2+dy^3$ be
  such an element with height less than $X$. Then we have the
  bounds $c\ll X^{1/3}$ and $d\ll X^{1/2}$. Furthermore, the proofs of
  \cite[Lemmas 21, 22]{MAJ} imply that if $f(x,y)$ is either reducible
  or corresponds to an order in a $C_3$-number field, then the value
  of $d$ determines the value of $c$ up to $O_\epsilon(X^\epsilon)$
  choices, proving the lemma.
\end{proof}

\vspace{-.085in}
\begin{proposition}\label{thbccount}
  Let $b\in[0,3n)$ be an integer, and
  let $L\subset U_{n,b}(\Z)$ be defined by finitely many congruence
  conditions. Then
\begin{itemize}
\item[{\rm (a)}] $N^+(L;n,X)
  =\displaystyle\frac{8}{405n^{4/3}}\nu(L)X^{5/6}+O_\epsilon(X^{1/2+\epsilon}),$ \pagebreak
\item[{\rm (b)}] $N^-(L;n,X)
  =\displaystyle\frac{32}{405n^{4/3}}\nu(L)X^{5/6}+O_\epsilon(X^{1/2+\epsilon}),$
\end{itemize}
where $\nu(L)$ denotes the volume of the closure of $L$ in
$U_{n,b}(\widehat{\Z})$.
\end{proposition}
\begin{proof}
By Lemma \ref{lred411}, it suffices to estimate the number of elements
in $L^\pm$ that have height bounded by $X$. By Proposition
\ref{davlem}, this is equal to $\nu(L)$ times the volume of
$\smash{U_{n,b}(\R)^\pm_{H<X}}$, up to an error of $O(X^{1/2})$, since the
largest projection of $\smash{U_{n,b}(\R)^\pm_{H<X}}$ is onto the
$y^3$-coefficient of elements in $U_{n,b}(\R)$ and this projection
has size $O(X^{1/2})$.

The volume of ${U_{n,0}(\R)^+_{H<X}}$ is 
\begin{equation}
  \int_{c=0}^{X^{1/3}/(3n^{1/3})}
  \int_{\smash{d=\frac{-2c^{3/2}}{3^{3/2}n^{1/2}}}}^{\smash{\frac{2c^{3/2}}{3^{3/2}n^{1/2}}}}dc\,dd=
  \int_{c=0}^{X^{1/3}/(3n^{1/3})}\frac{4c^{3/2}}{3^{3/2}n^{1/2}}dc=
  \frac{8}{405n^{4/3}}X^{5/6},
\end{equation}
and the volume of ${U_{n,0}(\R)^-_{H<X}}$ is 
\begin{equation}
\frac{8}{81n^{4/3}}X^{5/6}-\Vol(U_{n,0}(\R)^+_{H<X})=\frac{32}{405n^{4/3}}X^{5/6}.
\end{equation}

To compute the volumes of $\smash{U_{n,b}(\R)^\pm_{H<X}}$, note that there
exists a bijective unipotent invariant-preserving map 
$\smash{U_{n,0}(\R)^\pm_{H<X}} \to \smash{U_{n,b}(\R)^\pm_{H<X}}$ defined by $g(x,y)\mapsto g(x+by/(3n),y)$, and the Jacobian determinant of this map is
equal to $1$. Thus the volume of $\smash{U_{n,b}(\R)^\pm_{H<X}}$ is equal to that of
$\smash{U_{n,0}(\R)^\pm_{H<X}}$, concluding the proof of Proposition \ref{thbccount}.
\end{proof}

\vspace{.025in}\noindent
{\bf Proof of Theorem \ref{thcubringcount}:} Theorem \ref{thcubringcount} follows from the parametrization of
$n$-monogenized cubic orders in Theorem~\ref{thcubringpar}, together
with Proposition~\ref{thbccount} applied to $U_{n,b}(\Z)$ for each $b$
with $0\leq b<3n$ and then summing over these $b$. $\Box$

\subsection{The number of quartic rings having $n$-monogenized cubic resolvent rings of
bounded height}\label{subsecnquar}

In this subsection, we determine asymptotics for the number of pairs
$(Q,(C,\alpha))$, where $Q$ is an order in an $S_4$-quartic field and
$(C,\alpha)$ is an $n$-monogenized cubic resolvent ring of $Q$ having
bounded height. This requires determining asymptotics for the number
of $\MB(\Z)\times\SL_3(\Z)$-orbits $(A,B)$ on $V(\Z)$ of bounded
height that satisfy $\det(A)=n/4$. Fixing the determinant of $A$
imposes a cubic polynomial condition on elements in $V(\R)$; however, 
counting integer points on a cubic hypersurface is in general a
difficult question. 

Instead, we proceed as follows. 
We use the action of $\SL_3(\Z)$ to
transform any element $(A,B)\in V(\Z)$, with $\det(A)=n/4$, to an
element $(A',B')$, where $A'$ belongs to a fixed finite set of
representatives for the action of $\SL_n(\Z)$ on integer-coefficient ternary quadratic forms of determinant~$n/4$.

\begin{nota}
\label{sec:4.5}
{\em For a ring $R$ and an element $n\in R$, let $\mathcal{Q}_n(R)$ denote the space of ternary quadratic forms $A$ with coefficients in~$R$ 
such that $4\det(A)=n$. 
For a fixed $A$, let $V_{A}(R)$ denote the set of pairs $(A,B)\in V(R)$
  where $B$ is arbitrary. If $A\in\mathcal{Q}_n(R)$,
then the resolvent map $\Res$ sends
$V_A(R)$ to $U_n(R)$. For a fixed $b\in R$, let
$V_{A,b}(R)$ denote the subset of $V_A(R)$ mapping to 
$U_{n,b}(R)$ under~$\Res$. The set $V_{A,b}(R)$ is defined by linear conditions on $V_A(R)$, 
since $b$ is linear in the coefficients of $B$.
}\end{nota}

Instead of counting $\MB(\Z)\times\SL_3(\Z)$-orbits $(A',B')$
with $\det(A')=n/4$, it suffices to count $\SO_A(\Z)$-orbits on
$V_A(\Z)$, where $A$ ranges over a fixed set of representatives for
$\SL_3(\Z)\backslash\mathcal{Q}_n(\Z)$. We note that these representations
are different $\Z$-forms of the same representation over $\C$. 
\pagebreak

\begin{rem}{\em 
One such $\Z$-form of the representation of $\SO_A(\Z)$ on $V_A(\Z)$ 
is the representation of $\PGL_2$ on the space of binary
quartic forms. Indeed, when $n=1$, the set $\mathcal{Q}_1(\Z)$ consists of
a single $\SL_3(\Z)$-orbit, and one of the representatives of this
orbit is the $3\times3$ symmetric matrix $A_1$ given~by
\begin{equation}\label{eqA1}
A_1:=\left( \begin{array}{ccc}  &  &
    1/2 \\  & -1 &  \\ 1/2 &  &
  \end{array} \right).
\end{equation}
The algebraic group $\PGL_2$ is isomorphic to $\SO_{A_1}$ via the map
\begin{equation}\label{eqtwistedaction1}
\begin{array}{rcl}
\tau:\PGL_2(\Z)&\rightarrow&\SL_3(\Z),\;\; \mbox{given explicitly
  by}\\[.05in] {\left(\begin{array}{cc} a & {b} \\ c &
    d \end{array}\right)}&\mapsto&
\displaystyle\frac{{1}}{ad-bc}{\left(\begin{array}{ccc} {d^2} & {cd} &
    {c^2} \\ {2bd} & {ad+bc} & {2ac} \\ {b^2} & {ab} & {a^2}
\end{array}\right)}.
\end{array}
\end{equation}
Furthermore, we have a map from the space of binary quartic forms to
the space of pairs $(A_1,B)$ given by
\begin{equation}\label{vtow}
\phi: ax^4+bx^3y+cx^2y^2+dxy^3+ey^4\mapsto \left(
\left[ \begin{array}{ccc}\phantom0 & \phantom0 & 1/2 \\ \phantom0 & -1
    & \phantom0 \\ 1/2 & \phantom0 & \phantom0 \end{array} \right],
\left[ \begin{array}{ccc} a & b/2 & 0 \\ b/2 & c & d/2 \\ 0 & d/2 &
    e \end{array} \right]\right).
\end{equation}
The above two maps (with the latter map slightly modified) are
considered in work of Wood~\cite{MWt}, who then shows that the
representation of $\PGL_2$ on the space of binary quartic forms is
isomorphic over $\Z$ to the representation of
$\MB(\Z)\times\SO_{A_1}(\Z)$ on the space $V_{A_1}$.  
}\end{rem}

Asymptotics for
the number of $\PGL_2(\Z)$-orbits on the set of integer-coefficient binary
quartic forms with bounded height were determined in~\cite{BhSh}.  In
this section, we determine analogous asymptotics for other
$\Z$-forms of this representation. As a consequence, we prove the
following~theorem.

\begin{theorem}\label{thqnmccountnr}
For $i\in\{0,1,2\}$, let $N^{(i)}_4(n,X)$ denote the number of isomorphism classes of pairs
$(Q,(C,\alpha))$, where $Q$ is an order in a quartic $S_4$-field
having $i$ complex embeddings, and $(C,\alpha)$ is an $n$-monogenized
cubic resolvent ring of $Q$ with height bounded by $X$. Then
\begin{equation*}
N_4^{(i)}(n,X)=\frac{3n}{m_i}\sum_{A\in\SL_3(\Z)\backslash\mathcal{Q}_n(\Z)}\Vol(\FF_A\cdot \RF{i}(X))+o(X^{5/6}),
\end{equation*} 
where $\FF_A$ is a fundamental domain for the
action of $\SO_{A}(\Z)$ on $\SO_A(\R)$, and the sets $\RF{i}(X)$ are
defined in 
$\S\ref{sssnqred}$ prior to Proposition~$\ref{propfundset}$. 
\end{theorem}

In \S\ref{sssnqred},
we construct fundamental domains for the action of $\SO_A(\Z)$ on
$V_{A,b}(\R)$ for any $3\times3$ symmetric matrix $A$ having
determinant $n/4$, and any $b$ with $0\leq b<3n$. In~\S\ref{countsoava}, using these
fundamental domains, we count the number of generic $\SO_A(\Z)$-orbits
on $V_{A,b}(\Z)$ having bounded height. 
Summing over $A\in
\SL_3(\Z)\backslash\mathcal{Q}_n(\Z)$ and $0\leq b<3n$ will then immediately 
yield Theorem~\ref{thqnmccountnr}.

\subsubsection{Reduction theory}\label{sssnqred}
Let $A$ be the Gram matrix of an integer-coefficient ternary quadratic form with
$\det(A)=n/4$. In this subsubsection, we construct finite covers of
fundamental domains for the action of $\SO_A(\Z)$ on $V_{A,b}(\R)$. We
start by constructing fundamental sets for the action of $\SO_A(\R)$
on $V_{A,b}(\R)$.

\begin{cons}
\label{sec:4.8}
{\em
First, suppose that $A$ is isotropic over $\R$. If $A=A_1$, where $A_1$ is defined in
\eqref{eqA1}, then by the discussion above we may construct our
fundamental sets $\smash{\RF{i}}$, for 
$i\in\{0,1,2+,2-\}$, using analogous fundamental sets constructed
for the action of $\PGL_2$ on binary quartic forms. Namely, let
\begin{equation*}
  {R}_1^{(i)}:= \phi\bigl(\R_{>0}\cdot L^{(i)}\bigr),
\end{equation*}
where $\phi$ is defined in \eqref{vtow} and the sets $L^{(i)}$ are
defined in {\rm \cite[Table 1]{BhSh}}. If $A$ is a
general integer-coefficient ternary quadratic form that is isotropic over $\R$, 
then we simply translate the sets $\smash{R_1^{(i)}}$ using the group action
of $\GL_3(\R)$ on the space of ternary quadratic forms. Namely, let
$g_3\in\GL_3(\R)$ be the element such that
$g_3 \cdot A_1:=g_3A_1g_3^t=A$, and let
\begin{equation*}
{R}^{(i)}:=g_3. \smash{\phi\bigl(\R_{>0} \cdot L^{(i)}\bigr)}.
\end{equation*}
Then, for $A$ isotropic, $i\in\{0,1,2+,2-\}$, and any integer $b\in[0,3n)$, let
\begin{equation}\label{firstRF}
\smash{\RF{i}}:=\bigl\{(A,c_BA+B): (A,B)\in {R}^{(i)}\bigr\},
\end{equation}
where $c_B$ is the unique real number such that 
$\Res(A,c_BA+B)$ has $x^2y$-coefficient $b$.

Next, suppose that $A$ is anisotropic over $\R$. We begin with the
case $b=0$. Let
\begin{equation*}
\smash{\FF^{(2\#)}_{V_{A,0}}}:=\bigl\{
g_3\cdot({\rm Id},cB_f):c\in\R_{>0},\;f\in U_{n,0}(\R)^+,\;H(f)=1
  \bigr\},
\end{equation*}
where $g_3\in\GL_3(\R)$ is the matrix taking the identity matrix ${\rm Id}$ to
$A$, and $B_f$ corresponds to the diagonal matrix whose entries are
the (real) roots of $f(x,1)$ in ascending order. For a general integer
$b\in [0,3n)$, let
\begin{equation}\label{secondRF}
\smash{\RF{2\#}}:=\bigl\{(A,B+bA/3):(A,B)\in \smash{\FF^{(2\#)}_{V_{A,0}}}\bigr\}.
\end{equation}
If {$A$ is isotropic over $\R$ and
$i=2\#$, or $A$ is anisotropic over $\R$ and $i\in\{0,1,2+,2-\}$}, then let 
\begin{equation}\label{fourthRF}
\smash{\RF{i}=\emptyset.}
\end{equation}
Let 
\begin{equation}\label{thirdRF}
\smash{\RF{2}}:=\smash{\RF{2+}}\cup \smash{\RF{2-}}\cup \smash{\RF{2\#}}.
\end{equation}
Finally,
for $X>0$, let $\smash{\RF{i}(X)}$ denote the
subset of elements in $\smash{\RF{i}}$ having height less than $X$.
}\end{cons}

\begin{proposition}\label{propfundset}
If $A$ is isotropic over $\R$, then the sets $\smash{\RF{0}}$, \smash{$\RF{1}$},
and $\smash{\RF{2}}=\smash{\RF{2+}}\cup \smash{\RF{2-}}$ are fundamental sets for the
action of ${\SO_A(\R)}$ on $\smash{V_{A,b}(\R)^{(0)}}$, $\smash{V_{A,b}(\R)^{(1)}}$, and
$\smash{V_{A,b}(\R)^{(2)}}$, respectively. 
If~$A$ is anisotropic over $\R$, then the set \smash{$\RF{2\#}$} is a
fundamental set for the action of $\SO_A(\R)$ on~\smash{$V_{A,b}(\R)$}.
Moreover, for $i\in\{0, 1, 2, \smash{2\#},2+, 2-\}$, all entries of elements in $\smash{\RF{i}(X)}$ are~$O(X^{1/6})$.
\end{proposition}
\begin{proof}
When $A$ is anisotropic over $\R$, $A$ can be translated via $\GL_3(\R)$ into the
identity matrix~${\rm Id}$.  In that case, every element in $\smash{V_{{\rm Id},0}(\R)}$ has splitting type $(22)$
over $\R$. Furthermore, by the spectral theorem, every element in
$\smash{V_{{\rm Id},0}(\R)}$ is $\SO_3(\R)$ equivalent to $({\rm Id},B)$,
where $B$ is diagonal and $b_{11}<b_{22}<b_{33}$. Then $\smash{\RF{2\#}}$ gives the desired fundamental set. 
By \cite[\S2.1]{BhSh}, the $\smash{\RF{i}}$ are fundamental sets for $i\neq(2\#)$.

Regarding the heights of elements in $\smash{\RF{i}}$,
note that the set of elements in $\smash{\RF{i}}$ having height~$1$ have
absolutely bounded coefficients. The final assertion now follows since
$H$ is a homogeneous function on $V_{{\rm Id},0}(\R)$ of degree $6$.
\end{proof}
\vspace{-.085in}
\begin{theorem}\label{thfunddom}
Let $\FF_A$ be a fundamental domain for the action of $\SO_A(\Z)$ on
$\SO_A(\R)$. Then $\FF_A\cdot \smash{\RF{i}}$ is an $m_i$-fold cover of a
fundamental domain for the action of $\SO_A(\Z)$ on
$V(\R)^{(i)}$.
\end{theorem}

\begin{proof}
If $(A,B)\in V(\R)^{(i)}$, then by \pagebreak
Lemma \ref{lemsubmonss} the stabilizer of $(A,B)$ in
$\SL_3(\R)$ has size $m_i$. Since every element of this stabilizer
must belong to $\SO_A(\R)$, it follows that the size of
the stabilizer  of $(A,B)$  in $\SO_A(\R)$ is also $m_i$. The result now
follows from Proposition \ref{propfundset}.
\end{proof}

We will describe explicit constructions of $\FF_A$ in the next subsubsection.

\subsubsection{Counting $\SO_A(\Z)$-orbits on $V_A(\Z)$}\label{countsoava}

In this subsection, we fix positive integers $n$ and $b\in[0,3n)$.

\begin{nota}\label{N:4.11}{\em
Let $A$ be an integer-coefficient ternary quadratic form of
determinant $n/4$.  For a subset $\L\subset V_{A,b}(\Z)$ defined by
congruence conditions and $i\in\{0,1,2,2\pm,2\#\}$, let
$N^{(i)}(\L;A,X)$ denote the number of generic $\SO_A(\Z)$-orbits on
$\L^{(i)}$ having height less than~$X$. Let $\nu(\L)$ denote the
volume of the closure of $\L$ in $V_{A,b}(\widehat{\Z})$.
}\end{nota}

\begin{theorem}\label{thisotcount}
Let $\L\subset V_{A,b}(\Z)$ be defined
by finitely many congruence conditions, and let $\FF_A$ denote a
fundamental domain for the action of $\SO_A(\Z)$ on $\SO_A(\R)$.
Then 
\begin{equation*}
N^{(i)}(\L;A,X)=\frac{1}{m_i}\nu(\L)\Vol(\FF_A\cdot \smash{\RF{i}(X)})+o(X^{5/6}),
\end{equation*}
where $i=2\#$ if $A$ is isotropic over $\R$ and 
$i\in\{0,1,2+,2-\}$ if $A$ is anisotropic over $\R$. 
\end{theorem}
Since the truth of Theorem~\ref{thisotcount} is independent of the
choice of $\FF_A$, we begin by constructing convenient fundamental
domains $\FF_A$ for the action of $\SO_A(\Z)$ on $\SO_A(\R)$.

\begin{cons}
\label{sec:4.12}
{\em
Suppose that $A$ is anisotropic over $\Q$. Then we may
choose $\FF_A$ to be a compact fundamental domain for the left action of
$\SO_A(\Z)$ on $\SO_A(\R)$. By Theorem~\ref{thfunddom}, the
multiset $\FF_A\cdot \smash{\RF{i}}$ is an $m_i$-fold cover of a fundamental
domain for the action of $\SO_A(\Z)$ on~$\smash{V_{A,b}(\R)^{(i)}}$. 

Next, suppose that $A$ is isotropic over $\Q$. Then there 
exists $g_A\in\SL_3(\Q)$ such that $\smash{g_AAg_A^t}$ is
\begin{equation}\label{eqan}
A_n:=\left( \begin{array}{ccc}  &  &
    1/2 \\  & -n &  \\ 1/2 &  &
  \end{array} \right).
\end{equation}
For $F=\R$ or $\Q$, consider the maps
\begin{equation}\label{eqAtoAn}
\begin{array}{rcl}
\sigma_A:V_{A,b}(F)&\to& \smash{V_{A_n,b}(F)},\\[.05in]
\sigma_G:\SO_A(F)&\to&\smash{\SO_{A_n}(F)},
\end{array}
\end{equation}
given by $\smash{\sigma_A(A,B)=(A_n,g_ABg_A^t)}$ and
$\smash{\sigma_G(g)=g_Agg_A^{-1}}$. The map $\sigma_A$ is height preserving
and $\sigma_A(g\cdot v)=\sigma_G(g)\cdot\sigma_A(v)$. Let
$\L_n\subset V_{A_n,b}(\R)$ denote the lattice $\sigma_A(\L)$
and $\Gamma\subset\SO_{A_n}(\R)$ the subgroup
$\sigma_G(\SO_A(\Z))$. Then $\Gamma$ is commensurable with
$\SO_{A_n}(\Z)$. By~\cite[Example 2.5]{Borel}, 
there exists a fundamental domain $\FF$ for the action of $\Gamma$ on 
$\SO_{A_n}(\R)$ that is contained in a finite union~$\cup_j g_j
\SS$, where $g_i\in\SO_{A_n}(\Q)$ and $\SS$ is a Siegel domain. We
choose $\SS$ to be $N'T'K$, where
\begin{equation*}\label{nak}
N':= \left\{\left(\begin{array}{ccc} 1 & {} &{}\\ {2nu} & 1 & {}\\ {nu^2} & {u} & {1}\end{array}\right):
        u\in\ [-1/2,1/2] \right\}    , \;\;
T' := \left\{\left(\begin{array}{ccc} t^{-2} & {} & {} \\ {} & 1 & {} \\{} & {} & t^2 \end{array}\right):
       t\geq 1/2 \right\}, \;\;
\end{equation*}
and $K$ is a maximal compact subgroup of $\SO_{A_n}(\R)$. We set
$\FF_A:=\smash{\sigma_G^{-1}(\FF)}$. Then $\FF_A$ is a
fundamental domain for the action of $\SO_A(\Z)$ on $\SO_A(\R)$.
}\end{cons}

\noindent
{\bf Proof of Theorem \ref{thisotcount}:}
First, we consider the case that $A$ is anisotropic.  
By Proposition~\ref{propfundset}, the entries of elements in $\FF_A\cdot \smash{\RF{i}(X)}$ are
each $O(X^{1/6})$.  An argument identical to the proof
of \cite[Lemma 14]{dodpf} implies that the number of non-generic integral points \pagebreak
in $\FF_A\cdot \smash{\RF{i}(X)}$ is 
$o(X^{5/6})$.  Thus, to prove Theorem~\ref{thisotcount}, it suffices to
estimate the number of elements in $\FF_A\cdot \smash{\RF{i}(X)}\cap \L$, and this is immediate 
from Proposition~\ref{davlem}.
 
Next, we assume that $A$ is isotropic.
Theorem \ref{thfunddom} implies that
\begin{equation*}
N^{(i)}(\L;A,X)=
\frac1{m_{i}}\#\bigl(\FF\cdot\sigma_A(\RF{i}(X))\cap\L_n^\gen\bigr),
\end{equation*}
where $\L_n^\gen$ denotes the subset $v\in\L_n$ such that
$\sigma_A^{-1}(v)$ is generic.  Let $G_0$ be an open nonempty bounded
subset of $\SO_{A_n}(\R)$.  Then, by an averaging argument as in
\cite[Theorem 2.5]{BhSh}, we obtain
\begin{equation}\label{eqavg423}
N^{(i)}(\L;A,X)=
\frac{1}{m_{i}\Vol(G_0)}\int_{h\in\FF}
\#\bigl(h G_0\cdot\sigma_A(\smash{\RF{i}}(X))\cap\L_n^{\gen}\bigr)dh,
\end{equation}
where the volume of $G_0$ is computed with respect to any fixed Haar measure
$dh$. 

\begin{lemma}\label{lemredisot1}
If $h=utk\in\SS=N'T'K$ as above, $t\gg X^{1/24}$, and $g\in\SO_{A_n}(\Q)$, then
$$hG_0\cdot\sigma_A(\smash{\RF{i}(X)})\cap g^{-1}\L_n^\gen=\emptyset.$$
\end{lemma}
\begin{proof}
The entries of elements in $G_0\cdot\sigma_A(\smash{\RF{i}(X)})$ are all
$O(X^{1/6})$. Since $h=utk\in\SS$, the entries of $u$ and $k$ are
bounded. Hence the six coefficients of elements in
$hG_0(\sigma_A(\smash{\RF{i}(X)})$, considered as a subset of $V_{A_n}(\R)$, satisfy:
\begin{equation}\label{eqcoeffbdbiAn}
b_{11}\ll t^{-4}X^{1/6};\;\;
b_{12}\ll t^{-2}X^{1/6};\;\;
b_{13},b_{22}\ll X^{1/6};\;\;
b_{23}\ll t^{2}X^{1/6};\;\;
b_{33}\ll t^{4}X^{1/6}.
\end{equation}
The subset $\smash{V_{A_n,b}(\R)}$ of $V_{A_n}(\R)$ is cut out by one linear
equation, involving only $b_{13}$ and $b_{22}$. Hence $b_{11},b_{12},b_{22},b_{23},b_{33}$ form a complete set of
variables for $\smash{V_{A_n,b}(\R)}$, and elements in
$hG_0(\sigma_A(\smash{\RF{i}}(X))$ satisfy the same
coefficient bounds as in \eqref{eqcoeffbdbiAn}.

Since $g^{-1}\L_n$ is a lattice commensurable to $\smash{V_{A_n,b}}(\Z)$, the
denominators of elements in $g^{-1}\L$ are absolutely
bounded. Therefore, if $t\gg X^{1/24}$, then elements $(A_n,B)\in
hG_0\cdot\sigma_A(\smash{\RF{i}(X)})\cap g^{-1}\L_n$ satisfy $b_{11}=0$, and
are reducible since $A_n$ and $B$ have a common zero in
$\P^2(\Q)$.
\end{proof}

\vspace{-.1in}
\begin{lemma}\label{lemredisot2}
We have $\displaystyle
\int_{\substack{h=g_iutk\in\FF\\ t\ll X^{1/24}}}
\#\bigl(h G_0\cdot\sigma_A(\RF{i}(X))\cap(\L_n\setminus\L_n^{\gen})\bigr)dh=o(X^{5/6}).$
\end{lemma}

\begin{proof} 
The proof is identical to that of \cite[Lemma 14]{dodpf}.
\end{proof}
 
By Lemma~\ref{lemredisot1}, the integral on the right-hand side of
\eqref{eqavg423} can be restricted to $h\in\FF'$, where
$\FF'=\{h=g_intk\in\FF:t\ll X^{1/24}\}$. Lemma
\ref{lemredisot2} implies that replacing $\smash{\L_n^\gen}$ by $\L_n$ on the right-hand side of 
\eqref{eqavg423} introduces an error of at most $o(X^{5/6})$.

Using Proposition \ref{davlem} to estimate
\smash{$\#\bigl(hG_0\sigma_A(\smash{\RF{i}(X)})\cap\L_n\bigr)$} for $i=0$, $1$,
$2+$, and $2-$, we~obtain
\begin{equation*}
\begin{array}{rcl}
N^{(i)}(\L;A,X)&=&\displaystyle\frac1{m_i\Vol(G_0)}
\int_{h\in\FF'}\#\bigl(hG_0\sigma_A(\RF{i}(X))\cap\L_n\bigr)dh+o(X^{5/6})\\[.15in]
&=&\displaystyle\frac1{m_i\Vol(G_0)}\int_{h\in\FF'}
\Vol'(hG_0\cdot\sigma_A(\RF{i}(X)))+O(t^4X^{4/6})dh+o(X^{5/6}),
\end{array}
\end{equation*}
where the volumes $\Vol'$ of sets in $V_{A_n,b}(\R)$ are computed with
respect to Euclidean measure normalized so that $\L_n$ has covolume
$1$.

The measure $dn\,d^\times t\,dk/t^2$ is a Haar measure on
$\SO_{A_n,b}(\R)$. \pagebreak The integral over
$\FF$ of $t^4$ is $O(X^{1/12})$ while the volume of
$\FF\setminus\FF'$ is $O(X^{-1/12})$. It follows that
\begin{equation*}
\begin{array}{rcl}
N^{(i)}(\L;A,X)&=&\displaystyle\frac{1}{m_i\Vol(G_0)}\int_{h\in\FF}
\Vol'(hG_0\cdot\sigma_A(\RF{i}(X)))dh+O(X^{3/4})+o(X^{5/6})\\[.15in]
&=&\displaystyle\frac{1}{m_i}\Vol'(\FF\cdot\sigma_A(\RF{i}(X)))+o(X^{5/6})\\[.15in]
&=&\displaystyle\frac{1}{m_i}\nu(\L)\Vol(\FF_A\cdot \RF{i}(X))+o(X^{5/6}),
\end{array}
\end{equation*}
where the second equality follows from the analogue of Proposition~\ref{propjacsub} for $V$, stated as Proposition~\ref{propjac} in
\S\ref{subsecnlocal} below, and the volume of $\FF_A\cdot \smash{\RF{i}}(X)$ is
computed with respect to Euclidean measure on $V_{A,b}(\R)$ normalized
so that $V_{A,b}(\Z)$ has covolume $1$. We have proven 
Theorem~\ref{thisotcount}. $\Box$

\medskip
\noindent
{\bf Proof of Theorem \ref{thqnmccountnr}:} By Corollary~\ref{qparcor2}, we have
\begin{equation*}
N^{(i)}_4(n,X)
\!=\!\!\!
\sum_{A\in\SL_3(\Z)\backslash\cQ_n(\Z)}\sum_{b=0}^{3n-1}N^{(i)}(V_{A,b}(\Z);A,X)
\!=\!
\frac{3n}{m_i}\sum_{A\in\SL_3(\Z)\backslash\cQ_n(\Z)}\!\!\!
\Vol(\FF_A\cdot \RF{i}(X))+o(X^{5/6}),
\end{equation*}
where the final equality follows from Theorem \ref{thisotcount}. $\Box$

\subsection{Uniformity estimates and squarefree sieves}\label{subsecnunif}

As before, $n\geq 1$ is a fixed integer. Throughout this subsection, we
also fix an integer $b\in[0,3n)$. We begin with the following definition.
\begin{defn}
\label{sec:4.15}
 {\em A collection $S=(S_p)_p$, where $S_p$ is
  an open and closed subset of $U_{n,b}(\Z_p)$ whose boundary has
  measure $0$, is called a {\bf collection of cubic local
    specifications}. Such a collection $S$ is
  {\bf large} if, for all but finitely many primes $p$, the set $S_p$
  contains all elements $f\in U_{n,b}(\Z_p)$ with
  $p^2\nmid\Delta(f)$. We associate to $S$ the set
  $U_{n,b}(\Z)_S \subset U_{n,b}(\Z)$ of integer-coefficient binary cubic forms,
  where $f\in U_{n,b}(\Z)_S $ if and only if $f\in S_p$ for all primes
  $p$. We also associate to $S$ the set $V_{A,b}(\Z)_S\subset
  V_{A,b}(\Z)$, where $v\in V_{A,b}(\Z)_S$ if $\Res(v)\in S_p$ for all
  primes $p$.}
\end{defn}

In this subsection, we deduce the following theorem.

\begin{theorem}\label{countingpairsfinal}
  Let $S$ be a large collection of local specifications on $U_{n,b}$, and 
   let $A$ be an integer-coefficient ternary cubic form with
  determinant $n/4$. For $i\in \{0,1,2\#,2+,2-\}$, we have
\begin{equation*}
\begin{array}{rcl}
\displaystyle N^\pm(U_{n,b}(\Z)_S;n,X)&=&\;\;\;\;\;
\displaystyle\nu(U_{n,b}(\Z)_S)\Vol(U_{n,b}(\R)^\pm(X))+o(X^{5/6});\\[.065in]
\displaystyle N^{(i)}(V_{A,b}(\Z)_S;A,X)&=&
\displaystyle\frac{1}{m_i}\nu(V_{A,b}(\Z)_S)\Vol(\FF_A\cdot
\RF{i}(X))+o(X^{5/6}).
\end{array}
\end{equation*}
\end{theorem}

\noindent Note that we have $\nu(U_{n,b}(\Z)_S)=\prod_p\Vol(S_p)$. The
determination of $\nu(V_{A,b}(\Z)_S)$ in terms of $S$ is more subtle,
and we devote most of the next subsection to it.

To prove Theorem \ref{countingpairsfinal}, we require the following tail estimate.

\begin{proposition}\label{propunif}
For a prime $p$, let $\W_p(U_{n,b})$ $($resp.\ $\W_p(V_{A,b}))$ denote the set of elements in
$U_{n,b}(\Z)$ $($resp.\ $V_{A,b}(\Z))$ whose discriminants are divisible by $p^2$.
Then for any $M>0$, we
have
\begin{equation}\label{eqnunif}
\begin{array}{rcl}
\displaystyle\sum_{p> M}N(\W_p(U_{n,b});n,X)&=&\displaystyle O_\epsilon(X^{5/6+\epsilon}/M)+O(X^{1/3}),\\[.2in]
\displaystyle\sum_{p> M}N(\W_p(V_{A,b});A,X)&=&\displaystyle O_\epsilon(X^{5/6+\epsilon}/M)+O(X^{19/24}),
\end{array}
\end{equation}
where the implied constants are independent of $M$ and $X$.
\end{proposition}
\begin{proof}
We write $\W_p(U_{n,b})$ (resp.\ $\W_p(V_{A,b})$) as the disjoint
union $\smash{\W_p^{(1)}}(U_{n,b})\cup \smash{\W_p^{(2)}}(U_{n,b})$ 
(resp.\ $\smash{\W_p^{(1)}}(V_{A,b})\cup \smash{\W_p^{(2)}}(V_{A,b})$), where
$\smash{\W_p^{(1)}}(U_{n,b})$ (resp.\ $\smash{\W_p^{(1)}}(V_{A,b})$) consists of
elements in $\W_p(U_{n,b})$ (resp.\ $\W_p(V_{A,b})$) whose
discriminants are divisible by $p$ for mod $p$ reasons (in the sense
of \cite[\S1.5]{geosieve}). The bounds of \eqref{eqnunif}, with
$\W_p(U_{n,b})$ and $\W_p(V_{A,b})$ replaced by $\smash{\W_p^{(1)}}(U_{n,b})$
and $\smash{\W_p^{(1)}}(V_{A,b})$, respectively, follow from \cite[Theorem
  3.5]{geosieve} in conjunction with the proofs of Theorems
\ref{thbccount} and \ref{thisotcount}.

The proof of \eqref{eqnunif}, with $\W_p(U_{n,b})$ and $\W_p(V_{A,b})$ replaced by
$\smash{\W_p^{(2)}}(U_{n,b})$ and $\smash{\W_p^{(2)}}(V_{A,b})$, respectively, follows from
Proposition \ref{propunifsub2} by taking $\Y=X^\epsilon$.
\end{proof}

\vspace{.035in}
\noindent
{\bf Proof of Theorem~\ref{countingpairsfinal}:} Theorem
\ref{countingpairsfinal} is an immediate consequence of the tail
estimate in Proposition~\ref{propunif}, together with an application
of a squarefree sieve identically as in the proof of \cite[Theorem
  2.21]{BhSh}. $\Box$

\subsection{Local mass formulas}\label{seclmn}

To compute the volumes in Theorem \ref{countingpairsfinal} in a manner analogous to those computed in \S\ref{seclmsub}, we prove certain mass formulas relating 
\'etale quartic and cubic algebras over $\Q_p$.

\begin{defn}
\label{sec:4.18}
{\em
Let $n$ be a positive integer, and let $p$ be a prime dividing
$n$. Let $(\O,\alpha)$ be an $n$-monogenized cubic ring over $\Z_p$, where $\O$
is the ring of integers of an \'etale cubic extension $K$ of~$\Q_p$.
We define the {\bf algebra at infinity} $\mathcal A_\infty(\alpha)$ of $(\O,\alpha)$
in two equivalent ways. 

Let $f(x,y)\in U_n(\Z_p)$ be a binary cubic form corresponding to
$(\O,\alpha)$, so that $K=\Q_p[x]/(f(x,1))$.  Then $K=\prod L_i$ as a
product of fields, where $L_i:=\Q_p[x]/(f_i(x,y))$ and $f=\prod f_i$
as a product of irreducible factors over~$\Z_p$. Without loss of
generality, we may assume that the leading coefficient of $f_1$ is
$n$, and that the other $f_i$'s have leading coefficient $1$. (Indeed,
if two factors $f_1$ and $f_2$ both have leading coefficients that are
multiples of $p$, then $\prod f_i$ would not correspond to a maximal
ring by Lemma~\ref{propcondmax}.) Then we define $\mathcal
A_\infty(\alpha)$ to be $L_1$, the factor of $K$ corresponding to
$f_1$.

Equivalently, we can also define $\mathcal A_\infty(\alpha)$
intrinsically in terms of the data $(\O,\alpha)$.  If
$\Z_p[\alpha]\subset \O$ factors as $\Z_p\times S$ for some quadratic
ring $S$ over $\Z_p$, then $K\cong\Q_p\times (S\otimes\Q_p)=L_1\times
L_2$, and we define $\mathcal
A_\infty(\alpha):=L_1\cong\Q_p$. Otherwise, if $\Z_p$ is not a factor
of $\Z_p[\alpha]$, then we define $\monKal$ as the (unique) factor
$L_1$ of $K$ that is a ramified field extension of $\Q_p$.
}\end{defn}

\begin{rem}{\em  
To see the equivalence of the two definitions of the algebra at
infinity $\mathcal A_\infty(\alpha)$, note that if
$f(x,y)=\prod_{i\geq1} f_i(x,y)$ as in Definition \ref{sec:4.18}, then
the characteristic polynomial of $\alpha$ is $g(x)= (f_1(x,n)/n)
\prod_{i> 1} f_i(x,n)$ when expressed as a product of monic
polynomials. If $f_1(x,y)$ has degree $1$, then $\O$ has a factor of $\Z_p$
corresponding to $f_1(x,y)$; this is because $f_1(x,n)/n$ is a linear
factor of $g(x)$ sharing no common root modulo $p$ with
$\prod_{i>1}f_i(x,n)\equiv x^2\pmod{p}$. Otherwise, $f_1(x,y)$ has
degree at least $2$ with a root at infinity modulo $p$. It follows
that $f_1(x,y)$ corresponds to a ramified factor of~$K$. Hence our two
definitions of $\mathcal A_\infty(\alpha)$ agree.}
\end{rem}

\begin{nota}
\label{sec:4.19}
{\em For an \'etale cubic extension $K_3$ of $\Q_p$, let $\RR(K_3)$
  denote again the set of \'etale non-overramified quartic extensions
  of $\Q_p$, up to isomorphism, with cubic resolvent $K_3$.  For an
  $n$-monogenized \'etale cubic extension $(K_3,\alpha)$ of $\Q_p$,
  let $\RR^+(K_3,\alpha)$ (resp.\ $\RR^-(K_3,\alpha)$) consist of
  those $K_4\in\RR(K_3)$ such that $\monKal$ splits (resp.\ stays
  inert) in the unramified quadratic extension $K_6/K_3$ corresponding
  to~$K_4$.  }\end{nota}
\begin{theorem}\label{thpartialmass}
Let $(K_3,\alpha)$ be an $n$-monogenized \'etale cubic extension of
$\Q_p$. If $(K_3,\alpha)$ is sufficiently ramified, then
\begin{equation}\label{firsteq}
  \sum_{K_4\in\RR^+(K_3,\alpha)}\smash{\frac{1}{|\Aut_{K_3}(K_4)|}}=1;\quad
  \sum_{K_4\in\RR^-(K_3,\alpha)}\smash{\frac{1}{|\Aut_{K_3}(K_4)|}}=0.
\end{equation}
Otherwise, 
\begin{equation}\label{secondeq}
  \sum_{K_4\in\RR^+(K_3,\alpha)}\smash{\frac{1}{|\Aut_{K_3}(K_4)|}}=\frac12; \quad
  \sum_{K_4\in\RR^-(K_3,\alpha)}\smash{\frac{1}{|\Aut_{K_3}(K_4)|}}=\frac12.
\end{equation}
\end{theorem}
\begin{proof}
First, assume that $(K_3,\alpha)$ is sufficiently ramified;
then, by definition, either:
\begin{itemize}
\item[{\rm (a)}] $K_3$ is a totally ramified cubic extension of $\Q_p$, in
  which case $\monKal=K_3$; or
\item[{\rm (b)}] $K_3=\Q_p\times F$, where $F$ is a ramified quadratic
  extension of $\Q_p$, and $\monKal=\Q_p$.
\end{itemize}
Let $K_6$ be an unramified extension of $K_3$ such that
$N_{K_3/\Q_p}\Delta(K_6/K_3)$ is a square in $\Z_p^\times$. In Case (a) above,
$K_6/K_3$ must split for $N_{K_3/\Q_p}\Delta(K_6/K_3)$ to be a
square. In Case (b), note that $\smash{N_{F/\Q_p}\Delta(F'/F)}$ is a square for
every unramified extension $F'$ of $F$; hence, for
$N_{K_3/\Q_p}\Delta(K_6/K_3)$ to be a square, the component $\Q_p$ of $K_3$
must split in $K_6$. Thus (\ref{firsteq}) follows from Theorem \ref{thmmassmain}.

Next assume that $(K_3,\alpha)$ is not sufficiently ramified. Then $K_3$
has a factor $F\neq \monKal$ that is not a ramified quadratic extension of $\Q_p$. 
Now \smash{$N_{F/\Q_p}\Delta(F'/F)$} is a square for an unramified quadratic extension $F'/F$ 
if and only if it is split; hence the sizes of $\RR^+(K_3,\alpha)$ and $\RR^-(K_3,\alpha)$ are
equal. Since the automorphism group of every $K_4\in \RR(K_3)$ is the
same, (\ref{secondeq}) follows.
\end{proof}

\vspace{-.1in}
\begin{defn}
\label{sec:4.21}
{\em
Let $p$ be a prime. A ternary quadratic form $A\in\mathcal{Q}_n(\Z_p)$ is said
to be {\bf good at~$p$} if the conic in $\P^2(\overline{\F}_p)$ given as the 
zero set of the reduction of $A$ modulo $p$ is either smooth or a union of two distinct 
lines.  In the latter case, if each line is defined over $\F_p$,
then we define $\kappa_p(A):=1$ and say that $A$ is {\bf residually
  hyperbolic}. Otherwise, the two lines are a pair of conjugate lines
each defined over $\F_{p^2}$; we then define $\kappa_p(A):=-1$ and
say that $A$ is~{\bf residually~nonhyperbolic}.}
\end{defn}

\begin{lemma}\label{lem:quadlemma}
Let $p$ be a fixed prime and $n$ a fixed integer. 
\begin{itemize}
\item[{\rm (a)}] If $p\nmid n$, then $\mathcal{Q}_n(\Z_p)$ is a single $\SL_3(\Z_p)$-orbit. 
\item[{\rm (b)}] If $p\mid n$, then the set of good elements in
  $\mathcal{Q}_n(\Z_p)$ breaks up into two $\SL_3(\Z_p)$-orbits consisting of
  elements having $\kappa_p=1$ and $\kappa_p=-1$, respectively.
\end{itemize}
Finally, if $A$ is a ternary quadratic form over $\Z_p$ that is not
good at $p$, then for every ternary quadratic form $B$ over $\Z_p$,
$\Res(A,B)$ corresponds to a ring that is nonmaximal at $p$.
\end{lemma}

\begin{proof}
Suppose $p\neq 2$. Then any $m$-ary quadratic form over $\Z_p$ can be diagonalized via $\SL_m(\Z_p)$-transformations (see, e.g., \cite[Chapter~8,~Theorem~3.1]{Cassels}). Furthermore, a diagonal binary quadratic form $[u_1,u_2]$ over $\Z_p$ with unit determinant (i.e., $u_1u_2\in\Z_p^\times$) is $\SL_2(\Z_p)$-equivalent to the diagonal form $[1,u_1u_2]$. Indeed, if at least one of $u_1$ or $u_2$ is a unit square, then this is immediate; if both $u_1$ and $u_2$ are 
unit nonsquares in $\Z_p$, then it suffices to check 
that a unit nonsquare in $\Z_p$ can be expressed as a sum of two unit squares, and this is true since the smallest nonsquare $m$ in $\Z/p\Z$ is the sum $1+(m-1)$ of two unit squares.

Part (a) now follows because any form $A\in \mathcal{Q}_n(\Z_p)$ with $n$ a unit in $\Z_p$ is 
seen to be $\SL_3(\Z_p)$-equivalent to the diagonal form $[1,1,n/4]$. To deduce Part (b), note that the reduction modulo~$p$ of a good form $A\in\mathcal{Q}_n(\Z_p)$ with $p\mid n$ must
have $\F_p$-rank~2.  Therefore, $A$ is $\SL_3(\Z_p)$-equivalent to either $[1,1,n/4]$ or $[1,m,n/(4m)]$, where $m$ is
a unit nonsquare; these two forms are $\SL_3(\Z_p)$-inequivalent, 
as one is residually hyperbolic while the other is not. 

Next suppose $p=2$. A binary quadratic form over $\Z_2$ with unit discriminant
is $\GL_2(\Z_2)$-equivalent to either $S(x,y):=xy$ or $U(x,y):=x^2+xy+y^2$ (see, e.g., \cite[Chapter~8,~Lemma~4.1]{Cassels}). 
Since a good ternary quadratic form $A$ over $\Z_2$
is not diagonalizable (for otherwise $A$ modulo~2 would be the square of a linear form), it must have a summand that is either $S$ or
$U$. Thus every good form
$A\in\mathcal{Q}_n(\Z_2)$ is equivalent to either $B:=S\oplus [-n]$ or~$B':=U\oplus [n/3]$. It remains to prove that $B$ is $\SL_3(\Z_2)$-equivalent
to $B'$ if and only if $n$ is odd. Since $x^2+xy+y^2$ represents 
every odd squareclass in $\Z_2$, it follows that when $n$
is odd, the form $B'$ represents $0$ and is thus equivalent to~$B$. When $n$ is even, $B$ and $B'$ are inequivalent since $B$
is residually hyperbolic while $B'$ is not. (For quadratic forms of arbitrary dimension, see \cite[Chapter~8]{Cassels} and \cite[Chapter~15,~\S7]{ConwaySloane} for the general theory over $\Z_p$, and \cite[Theorem~2.6]{Hanke_structure_of_massses} for a proof of~(a) over local~fields.)

Finally, if $A$ is a ternary quadratic form over $\Z_p$ that is not good, and $B$ is any ternary quadratic
form over $\Z_p$, then the $x^3$- and $x^2y$-coefficients of
$\Res(A,B)$ are divisible by $p^2$ and $p$, respectively.  The last
assertion of the lemma then follows from \cite[Lemma 2.10]{BBP} (stated below in
Lemma~\ref{propcondmax}), which asserts that a binary cubic form whose
$x^3$-coefficient is divisible by $p^2$ and whose $x^2y$-coefficient
is divisible by $p$ is nonmaximal at $p$.
\end{proof}

The above lemma motivates the definition of the following partial mass
formulas.
\begin{defn}
\label{sec:4.23}
{\em
Let $n$ be a positive integer, let $p$ a prime dividing $n$, and
let $f$ be an element in $U_n(\Z_p)$ corresponding to a maximal
ring. Then we define the {\bf $\kappa$-mass} of $f$ as
\begin{equation*}
  \Mass_p^\pm(f):=\sum_{\substack{(A,B)\in\frac{\Res^{-1}(f)}{\SL_3(\Z_p)}\\\kappa_p(A)=\pm1}}
  \smash{\frac1{\#\Stab_{\SL_3(\Z_p)}(A,B)}}, \vspace{-.1in}
\end{equation*}
where \smash{$\frac{\Res^{-1}(f)}{\SL_3(\Z_p)}$} is a
set of representatives for the action of $\SL_3(\Z_p)$ on
$\Res^{-1}(f)$.}
\end{defn}

\begin{corollary}\label{corpartialmass}
Let $n$ be a positive integer, let $p$ be a prime dividing $n$, and
let $f(x,y)$ be an element in $U_n(\Z_p)$ corresponding to a maximal
ring. Then
\begin{equation*}
\begin{array}{ll}
\displaystyle\Mass^+_p(f)=\,1\mbox{ and }\Mass^-_p(f)=\,0 & \mbox{ if $f$ is sufficiently ramified;}\\[.1in]
\displaystyle\Mass^+_p(f)=\frac12\mbox{ and }\Mass^-_p(f)=\frac12 & \mbox{ otherwise.}
\end{array}
\end{equation*}
\end{corollary}

\begin{proof}
Let $(A,B)$ be an element of $V(\Z_p)$ with
resolvent $f$. Then $(A,B)$ corresponds to an \'etale quartic algebra
$K_4$ with cubic resolvent $K_3$, and $K_4$ yields an unramified
quadratic extension~$K_6/K_3$. It follows that $\monKal$
splits (resp.\ stays inert) in $K_6$ if and only if $A$ is residually
hyperbolic (resp.\ residually nonhyperbolic). 
\end{proof}

\subsection{Volume computations and proof of the main $n$-monogenic theorem (Theorem~\ref{main_theorem})} \label{subsecnlocal}

Let $\Sigma$ denote a large collection of $n$-monogenized cubic fields. We may 
assume that every $(K,\alpha)$ arising from $\Sigma$ satisfies 
$\Tr(\alpha)=b\in[0,3n)$, 
as the general case follows by summing over
  all such~$b$. The large collection $\Sigma$ of $n$-monogenized
  cubic fields gives rise to a large collection $S=(S_p)_p$ of binary 
  cubic forms where $S_p\subset U_{n,b}(\Z_p)$ for every $p$ and every
  $f\in S_p$ is maximal. To compute $\nu(V_{A,b}(\Z)_S)$, we use the
  following Jacobian change of variables.

\begin{proposition}\label{propjac}
Let $K$ be $\R$ or $\Z_p$, let $|\cdot|$ denote the usual absolute
value on $K$, and let $s:U_{n,b}(K)\to V_{A,b}(K)$ be a continuous map
such that $\Res(s(f))=f$, for each $f\in U_{n,b}(K)$. Then there
exists a rational nonzero constant $\J$, independent of $K$ and $s$,
such that for any measurable function $\phi$ on $V_{A,b}(K)$, we have
\begin{equation*}\label{eqjac}
  \begin{array}{rcl}
 \!\! \! \displaystyle\int_{\SO_A(K)\cdot s(U_{n,b}(K))}\!\!\!\phi(v)dv&\!\!\!=\!\!\!&
|\J|\!\displaystyle\int_{f\in U_{n,b}(K)}\displaystyle\int_{g\in \SO_A(K)}
    \!\!\phi(g\cdot s(f))\omega(g) df,\\[0.2in]
\displaystyle\int_{V_{A,b}(K)}\!\!\phi(v)dv&\!\!\!=\!\!\!&
|\J|\!\displaystyle\int_{\substack{f\in U_{n,b}(K)\\ \Disc(f)\neq 0}}\!
\displaystyle\sum_{v\in\!\!\textstyle{\frac{V_{A,b}(K)\cap\Res^{-1}(f)}{\SO_A(K)}}}
\!\!\!\frac{1}{\#\Stab_{\SO_A(\Z_p)}(v)}\int_{g\in \SO_A(K)}\!\!\!\phi(g\cdot v)\omega(g)df,\\[-.05in]
  \end{array}
\end{equation*}
where \smash{$\frac{V_{A,b}(K)\cap\Res^{-1}(f)}{\SO_A(K)}$} is a
set of representatives for the action of $\SO_A(K)$ on
$V_{A,b}(K)\cap\Res^{-1}(f)$.
\end{proposition}

\begin{proof}
   Proposition~\ref{propjac} 
   follows immediately
  from the proofs of \cite[Propositions 3.11--12]{BhSh} (see also
  \cite[Remark~3.14]{BhSh}).
\end{proof}

\vspace{-.085in}
\begin{corollary}\label{corjac}
Let $S_p\subset U_{n,b}(\Z_p)$ be a closed subset whose boundary has
measure $0$. Consider the set $V_{A,b}(\Z)_{S,p}$ defined by
$V_{A,b}(\Z)_{S,p}:=V_{A,b}(\Z_p)\cap\Res^{-1}(S_p).$ Then
\begin{equation}
\Vol(V_{A,b}(\Z)_{S,p})=
\left\{
\begin{array}{ll}
\;\;\,\displaystyle |\J|_p\Vol(\SO_A(\Z_p))\Vol(S_p) &\mbox{when } p\nmid n;
  \\[.05in]
  \displaystyle \frac12|\J|_p\Vol(\SO_A(\Z_p))\Vol(S_p)
  (1+\kappa_p(A)\rho_p(\Sigma)) &\mbox{when } p\mid n.
\end{array}\right.
\end{equation}
\end{corollary}
\begin{proof}
By Proposition~\ref{propjac} (setting $\phi$ to be the characteristic
function of $V(\Z)_{S,p}$), we have
\begin{equation}\label{eqcorjact1}
\Vol(V_{A,b}(\Z)_{S,p})=|\J|_p\Vol(\SO_A(\Z_p))\int_{f\in S_p}
\sum_{v\in\!\!\textstyle{\frac{V_{A,b}(\Z_p)\cap\Res^{-1}(f)}{\SO_A(\Z_p)}}}
\!\!\!\smash{\frac{1}{\#\Stab_{\SL_3(\Z_p)}(v)}}
df, \vspace{-.065in}
\end{equation}
since the stabilizers of $v$ in $\SO_A(\Z_p)$ and $\SL_3(\Z_p)$ are
the same, being equal to $\SO_A(\Z_p)\cap \SO_B(\Z_p)$. 

If $p\nmid n$, then 
$\mathcal{Q}_n(\Z_p)$ 
is a single $\SL_3(\Z_p)$-orbit by Lemma~\ref{lem:quadlemma}(a). Hence there
is a one-to-one correspondence between the sets $\SO_A(\Z_p)\backslash
(V_{A,b}(\Z_p)\cap\Res^{-1}(f))$ and $\SL_3(\Z_p)\backslash
\Res^{-1}(f)$, and so the integrand on the right-hand side
of \eqref{eqcorjact1} is equal to~$\Mass_p(f)$. 

If $p\mid n$, then 
$\mathcal{Q}_n(\Z_p)$ 
consists of two $\SL_3(\Z_p)$-orbits by Lemma~\ref{lem:quadlemma}(b), having $\kappa_p=1$ and
$\kappa_p=-1$, respectively. Thus there is a one-to-one correspondence
between $\SO_A(\Z_p)\backslash (V_{A,b}(\Z_p)\cap\Res^{-1}(f))$ and the
set of elements $(A_1,B_1)\in\SL_3(\Z_p)\backslash \Res^{-1}(f)$ with
$\kappa_p(A_1)=\kappa_p(A)=\pm1$. Hence 
the integrand on the right-hand side
of \eqref{eqcorjact1} is $\Mass^{\pm}_p(f)$.
By Corollaries~\ref{cortotmass} and \ref{corpartialmass}, we then have
\begin{equation*}
\smash{\int_{f\in S_p}}\Mass_p(f)=1; \qquad \smash{\int_{f\in S_p}}\Mass_p^\pm(f)=\frac{1\pm\rho_p(\Sigma)}{2},
\end{equation*}
as desired. 
\end{proof}

\vspace{-.085in}
\begin{proposition}\label{propgenus}
Let $\GG$ be a genus of elements in $\mathcal{Q}_n(\Z)$ that are good at $p$ for all $p\mid n$. Then 
\begin{equation*}
  \sum_{A\in\SL_3(\Z)\backslash\GG}
  \frac{N^{(i)}(V_{A,b}(\Z)_S;A,X)}{N^\pm(U_{n,b}(\Z)_S;n,X)}
=\frac2{m_i}\prod_{p\mid n}\frac{1+\kappa_p(A)\rho_p(\Sigma)}{2}+o(1),
\end{equation*}
where: $i\in\{0, 2+, 2-, 2\#\}$ if $\pm=+$; $i=1$ if
$\pm=-$ and $A$ is isotropic over $\R$; and $i=2\#$ and $\pm=+$ if $A$ is anisotropic over~$\R$.
\end{proposition}

\begin{proof}
From Theorem \ref{countingpairsfinal}, we obtain
\begin{equation*}
\frac{\smash{N^{(i)}(V_{A,b}(\Z)_S;A,X)}}{N^\pm(U_{n,b}(\Z)_S;n,X)}=
\frac{\frac{1}{m_i}\nu(\smash{V_{A,b}(\Z)_S})\Vol(\FF_A\cdot \smash{\RF{i}}(X))}{
\nu(U_{n,b}(\Z)_S)\Vol(\smash{U_{n,b}}(\R)^\pm)} +o(1). 
\end{equation*}
We evaluate $\Vol(\FF_A\cdot \smash{\RF{i}}(X))$ and $\nu(V_{A,b}(\Z)_{S,p})$
using Proposition \ref{propjac} (with $\phi$ being the \pagebreak characteristic
function of $\FF_A\cdot \smash{\RF{i}}(X)$) and Corollary \ref{corjac},
respectively, yielding
\begin{equation*}
\begin{array}{rcl}
\displaystyle \frac{\frac{1}{m_i}\nu(V_{A,b}(\Z)_S)\Vol(\FF_A\cdot \smash{\RF{i}}(X))}{
\nu(U_{n,b}(\Z)_S)\Vol(U_{n,b}(\R)^\pm)}&=&
\displaystyle\frac{1}{m_i}\Vol(\FF_A)\prod_{p}\Vol(\SO_A(\Z_p))
\prod_{p\mid n}\frac{1+\kappa_p(A)\rho_p(\Sigma)}{2}.
\end{array}
\end{equation*}
Since 
\begin{equation*}
\sum_{A\in\SL_3(\Z)\backslash\GG}\Vol(\FF_A)\prod_{p}\Vol(\SO_A(\Z_p))=2,
\end{equation*}
the Tamagawa number of $\SO_3$, the proposition follows.
\end{proof}

\noindent
The following lemma
will be used repeatedly in the proof of Theorem~\ref{main_theorem}. 

\begin{lemma}[{\cite[{Lemma~4.17}]{Hanke_structure_of_massses}}]\label{hankelemma}
Suppose $\mathbb T$ is a nonempty finite set, $X_i$ and $Y_i$ are
indeterminates for each $i\in\mathbb T$, and $\mu_N$ is the set
of all $N^{th}$ roots of unity in $\C$. Then for any $c\in\mu_N$, we have
the polynomial~identity
\begin{equation}\label{eqcombeq}
\sum_{\substack{(\epsilon_i)_{i\in\mathbb T}\in\mu_N^{\mathbb T}\\\prod_{i\in\mathbb T}\epsilon_i=c}}
\prod_{i\in \mathbb T}(X_i+\epsilon_i Y_i)=N^{|\mathbb T|-1}\smash{\Bigl[\prod_{i\in\mathbb T}X_i+c\prod_{i\in \mathbb T}Y_i\Bigr]}.
\end{equation}
\end{lemma}

\noindent
We will also require the following notation. 
\begin{nota}
\label{sec:4.28}
{\em
For a genus $\GG$ of elements in $\mathcal{Q}_n(\Z)$ that are good at $p$ for all $p$, 
define $\kappa_p(\GG):=\kappa_p(A)$ for 
$A\in\GG$.  Let $T:=T_n$ denote the set of primes $p$ dividing $n$, and let
$T_{\even}:=T_{n,\even}$ (resp.\ $T_{\odd}:=T_{n,\odd}$) denote the set of $p\in T$ such that the
order of $p$ dividing $n$ is even (resp.\ odd).
}\end{nota}

\noindent{\bf Proof of Theorem \ref{main_theorem}:}  Let
\smash{$F^\pm_\Sigma(n,X)$} denote the set of fields $K$ in $F_\Sigma(n,X)$
with $\pm\Delta(K)>0$.  

\vspace{.1in}\noindent
{\bf 
(a):} 
In the case of 2-class groups of totally real cubic fields, 
by Theorem \ref{thclgp1} we have
\begin{equation*}
\begin{array}{rcl}
\Avg(\Cl_2,F_\Sigma^+(n,X))&=&1+\displaystyle
\displaystyle\frac{\#\bigl(\SL_3(\Z)\backslash
  \bigl(V(\Z)^{(0),\gen}\cap\Res^{-1}(U_{n,b}(\Z)_{S,X}^+)\bigr)
  \bigr)}{N^+(U_{n,b}(\Z)_S ;n,X)}+o(1)\\[.175in]
&=&1+\displaystyle
\frac{\displaystyle\sum_{A\in\SL_3(\Z)\backslash\mathcal{Q}_n(\Z)}
  N^{(0)}(V_{A,b}(\Z)_S;A,X)}{N^+(U_{n,b}(\Z)_S;n,X)}+o(1).
\end{array}
\end{equation*}

We claim that there is a (unique) genus $\GG$ of good ternary quadratic forms with
$\kappa_p(\GG)=\epsilon_p$ for all $p\in T$ if and only if the ordered tuple
\smash{$\epsilon\in\{\pm1\}^T$} satisfies
\smash{$\prod_{T_\odd}\epsilon_p=\kappa_\infty(\GG)$}; here we define 
$\kappa_\infty(\GG)$ to be $1$ if the forms in $\GG$ are isotropic over
$\R$ and $-1$ otherwise. 
Indeed, if $c_v(A)$ denotes the Hasse invariant of~$A$ (see~\cite[p.~55]{Cassels}) and $(a,b)_v$
denotes the Hilbert symbol of $a$ and $b$ over $\Q_v$, then one checks~that
\begin{equation}\label{eqkappaHasse}
  \smash{\displaystyle c_p(A)=(-1,-n)_p\cdot\kappa_p(A)^{{\rm ord}_p(n)}}
\end{equation}
for any prime $p$ and any $A\in\mathcal Q_n(\Z)$ that is good at all primes. Since $n$ is positive, we have $\kappa_\infty(A)=-c_\infty(A)$.  Thus the product formula for the Hilbert symbol translates
directly to
\begin{equation*}
\prod_{p\in T_\odd}\kappa_p(A)=\kappa_\infty(A).
\end{equation*}
The claim now follows from 
\cite[Theorem~1.3, p.~77]{Cassels} and \cite[Theorem 1.2, p.~129]{Cassels}.

By Proposition~\ref{propgenus}, we therefore 
have
\begin{equation}\label{eqavg+}
\begin{array}{rcl}
\Avg(\Cl_2,F_\Sigma^+(n,X))&\!=\!&1+\displaystyle\frac{1}{2}
\sum_{\GG:\kappa_\infty(\GG)=1}\displaystyle\prod_{p\mid n}
\frac{1+\kappa_p(\GG)\rho_p(\Sigma)}{2}+o(1)\\[.2in]
&\!=\!&1+\displaystyle\frac12
\sum_{\substack{(\epsilon_p)\in\{\pm1\}^T\\\prod_{T_\odd}\epsilon_p=1}}
\prod_{T}\frac{1+\epsilon_p\rho_p(\Sigma)}{2}+o(1)\\[.3in]
&\!=\!&1+\displaystyle\frac{1}{2}

\sum_{\substack{(\epsilon_p)\in\{\pm1\}^{T_\odd}\\\prod_{T_\odd}\epsilon_p=1}}
\prod_{T_\odd}\frac{1+\epsilon_p\rho_p(\Sigma)}{2}

\prod_{T_{{\rm even}}}\Bigl(
\frac{1+\rho_p(\Sigma)}{2}+\frac{1-\rho_p(\Sigma)}{2}\Bigr)+o(1)\\[.3in]
&\!=\!&1+\displaystyle\frac14\Bigl(1+\prod_{T_\odd}\rho_p(\Sigma)\Bigr)+o(1)\\[.275in]
&\!=\!&\displaystyle{\frac54+\frac{1}{4}\rho(\Sigma)+o(1)},
\end{array}
\end{equation}
where the penultimate equality follows from 
Lemma~\ref{hankelemma} with $N=2$, $c=1$, and $\mathbb T=T_\odd.$

\vspace{.1in}\noindent
{\bf 
(b):} In the case of complex cubic fields, we similarly have
\begin{equation}\label{eqavg-}
\begin{array}{rcl}
\Avg(\Cl_2,F_\Sigma^-(n,X))&=&1+
\displaystyle\frac{\#\Bigl(\SL_3(\Z)\backslash
  \bigl(V(\Z)^{(1),\gen}\cap\Res^{-1}(U_{n,b}(\Z)_{S,X}^-)\bigr)
  \Bigr)}{N^-(U_{n,b}(\Z)_S;n,X)}+o(1)\\[.15in]
&=&1+\displaystyle
\frac{\displaystyle\sum_{A\in\SL_3(\Z)\backslash\mathcal{Q}_n(\Z)}
  N^{(1)}(V_{A,b}(\Z)_S;A,X)}{N^-(U_{n,b}(\Z)_S;n,X)}+o(1)\\[.15in]
&=&1+\displaystyle\frac12\Bigl(1+\prod_{T_\odd}\rho_p(\Sigma)\Bigr)+o(1)\\[.25in]
&=&\displaystyle{\frac32+\frac{1}{2}\rho(\Sigma)+o(1)},
\end{array}
\end{equation}
where we again use 
Lemma~\ref{hankelemma} with $N=2$, $c=1$, and $\mathbb T=T_\odd$. 

\vspace{.1in}\noindent
{\bf 
(c):} 
Finally, in the case of narrow 2-class groups of totally real cubic fields, we have
\begin{equation*}
\Avg(\Cl^+_2,F_\Sigma^+(n,X))\;=\;1+\displaystyle
\frac{\displaystyle\sum_{A\in\SL_3(\Z)\backslash\mathcal{Q}_n(\Z)}
  N^{(0)}(V_{A,b}(\Z)_S;A,X)+N^{(2)}(V_{A,b}(\Z)_S;A,X)}{N^+(U_{n,b}(\Z)_S;n,X)}+o(1).
\end{equation*}
The contributions of $N^{(2+)}(V_{A,b}(\Z)_S;A,X)$ and
$N^{(2-)}(V_{A,b}(\Z)_S;A,X)$ are the same as that of
$N^{(0)}(V_{A,b}(\Z)_S;A,X)$, yielding a total contribution of
\begin{equation*}
\frac34\Bigl(1+\prod_{T_\odd}\rho_p(\Sigma)\Bigr)+o(1)
\end{equation*}
to the average size of $\Cl_2^+$ from the splitting types $0$, $2+$,
and $2-$. On the other hand, since $\kappa_\infty(A)=-1$ for
pairs $(A,B)\in \smash{V(\Z)^{(2\#)}}$, the contribution of
\smash{$N^{(2\#)}(V_{A,b}(\Z)_S;A,X)$} is
\begin{equation*}
\frac14\Bigl(1-\prod_{T_\odd}\rho_p(\Sigma)\Bigr)+o(1),
\end{equation*}
this time using 
Lemma~\ref{hankelemma} with $N=2$, $c=-1$, and $\mathbb T=T_\odd$.  Summing up, we obtain
\begin{equation}\label{eqavgn+}
\Avg(\Cl^+_2,F_\Sigma^+(n,X))=2+\frac{1}{2}\rho(\Sigma)+o(1).
\end{equation}
We have proven Theorem \ref{main_theorem}.
 $\Box$

\subsection{Deduction of Theorems \ref{thmoncubicfields}, \ref{thmain2}, and \ref{thmain1}}\label{subsecnlocal2}

{\bf Proof of Theorem~\ref{thmoncubicfields}:}
Theorem~\ref{thmoncubicfields} follows from Theorem~\ref{main_theorem}
by noting that there are no primes dividing $n=1$ to odd order, and so
$\rho(\Sigma)=1$ for every large collection $\Sigma$. $\Box$

\medskip
\noindent
{\bf Proof of Theorem~\ref{thmain1}:} 
Theorem~\ref{thmain1} follows from
Theorem~\ref{main_theorem} by noting that when $\Sigma_p\subset T_p$
contains no extensions ramified at primes dividing $k$, then
$\rho_p(\Sigma_p)=0$. $\Box$

\medskip
To deduce Theorem \ref{thmain2} from Theorem \ref{main_theorem}, we
determine the probability that an $n$-monogenized cubic field
is sufficiently ramified at a prime dividing $n$.  We use the
following lemmas.

\begin{lemma}\label{lemsuf}
Let $n$ be a positive integer, and let $p$ be a prime dividing
$n$. 
Suppose 
$f(x,y) 
\in U_n(\Z_p)$ 
corresponds
to a maximal $n$-monogenized cubic ring
$(\mathcal{C}_p,\alpha_p)$. Then the pair $(\mathcal{C}_p\otimes\Q_p,\alpha_p)$ is
sufficiently ramified if and only if either:
\begin{itemize}
\item[{\rm (a)}] $f(x,y)$ has a triple root in $\P^1(\F_p);$ or 
\item[{\rm (b)}] $f(x,1)$ has a double root in $\F_p$.
\end{itemize}
\end{lemma}
\begin{proof}
If $f(x,y)$ has distinct roots in $\P^1(\overline{\F}_p)$, then
$\mathcal K_p:=\mathcal{C}_p\otimes\Q_p$ is unramified and so $(\mathcal K_p,\alpha_p)$ is not
sufficiently ramified. If $f(x,y)$ has a triple root modulo $p$, then
$\mathcal K_p$ is a totally ramified cubic extension of $\Q_p$, and hence
$(\mathcal K_p,\alpha_p)$ is sufficiently ramified.

We now assume that the reduction of $f(x,y)$ modulo $p$ has a double
root but not a triple root. 
Then the pair $(\mathcal{C}_p,\alpha_p)$ is sufficiently ramified when
$\Z_p[\alpha]\otimes\F_p$ is isomorphic to $\F_p\times\F_p[t]/(t^2)$,
and not sufficiently ramified when $\Z_p[\alpha]\otimes\F_p$ is isomorphic
to $\F_p[t]/(t^3)$. 
Write $f(x,y)=nx^3+bx^2y+cxy^2+dy^3$. Then the
characteristic polynomial of $\alpha_p$ is equal~to
$x^3+bx^2y+nxy^2+n^2dy^3$. Hence 
$\Z_p[\alpha]\otimes\F_p\cong\F_p[x]/(x^3+bx^2)$, and so
$(\mathcal{C}_p,\alpha_p)$ is sufficiently ramified if and only if $p\nmid
b$. Thus, when the reduction of $f(x,y)$ modulo $p$ has a
double root but not a triple root, the pair $(\mathcal{C}_p,\alpha_p)$ is
sufficiently ramified if and only if the reduction of $f(x,1)$ modulo
$p$ has a double root in $\F_p$.
\end{proof}

The next result determines when
$f\in U(\Z)$ corresponds to a cubic ring nonmaximal at $p$.
\begin{lemma}[{\cite[Lemma 2.10]{BBP}}]\label{propcondmax}
  A cubic ring corresponding to a binary cubic form $f(x,y)$ fails
  to be locally maximal at $p$ if and only if either: {\rm (a)} $f$ is a
  multiple of $p$, or {\rm (b)} there is some $\GL_2(\Z)$-transformation of
  $f(x,y)=ax^3+bx^2y+cxy^2+dy^3$ such that $d$ is a multiple of $p^2$
  and $c$ is a multiple of~$p$.
\end{lemma}

\begin{proposition}\label{propdensmax}
The density $\mu_\max(U_n(\Z_p))$ of elements in $U_n(\Z_p)$ that are maximal is given by
\begin{equation*}
\mu_\max(U_n(\Z_p))=\left\{
\begin{array}{cl}
  (p^2-1)/p^2 &\mbox{if } p^2\nmid n;\\[.035in]
  (p-1)^2(p+1)/p^3&\mbox{if } p^2\mid n.
\end{array}\right.
\end{equation*}
\end{proposition}

\begin{proof}
We calculate the probability of nonmaximality in $U_n(\Z_p)$ using
conditions (a) and (b) of Lemma~\ref{propcondmax}. Regarding (a), the
probability that an element of $U_n(\Z_p)$ is imprimitive is 0
if $p\nmid n$ and is $1/p^3$ otherwise.
\pagebreak

Now (b) implies that a primitive binary cubic form
$f(x,y)=nx^3+bx^2y+cxy^2+dy^3\in U_n(\Z_p)$ is nonmaximal if and only
if: $p\mid b$ and $p^2\mid n$ {\bf (Condition $\mathrm C_\infty$)}; or for some
$r\in\{0,1,\ldots,p-1\}$, the binary cubic form 
$f(x+ry,y)\!=\!nx^3+b_rx^2y+c_rxy^2+d_ry^3$ satisfies $p\mid c_r$ and $p^2\mid
d_r$ {\bf (Condition~$\mathrm C_r$)}. We determine next the probabilities that
primitivity and nonmaximality occur due to Condition~$\mathrm C_r$ for
$r=0,1,\ldots,p-1,\infty$, and then~sum.
\begin{itemize}
\item
If $p\nmid n$, then for finite $r$, primitivity and Condition $\mathrm C_r$ hold
when $p\mid c_r$ and $p^2\mid d_r$, which has probability $1/p^3$;
Condition $\mathrm C_\infty$ never holds in this case.
\item
If $p\parallel n$, then for finite $r$, primitivity and Condition $\mathrm C_r$
hold when $p\nmid b_r$, $p\mid c_r$, and $p^2\mid d_r$, which has
probability $(p-1)/p\times1/p^3 = (p-1)/p^4$; Condition $\mathrm C_\infty$ never
holds in this case.
\item
If $p^2\mid n$, then for finite $r$, primitivity and Condition $\mathrm C_r$
hold (just as when $p\parallel n$) 
with probability $(p-1)/p^4$;
moreover, primitivity and Condition $\mathrm C_\infty$ hold when $p\mid b$ and either $p\nmid c$ or $p\nmid d$, which occurs with probability $1/p\times
(p^2-1)/p^2=(p^2-1)/p^3$.
\end{itemize}
Thus the probability of nonmaximality is given by: $p\times 1/p^3 = 1/p^2$ if
$p\nmid n$; \;$1/p^3+p\times (p-1)/p^4=1/p^2$ if $p\parallel n$; and
$1/p^2+(p^2-1)/p^3 = (p^2+p-1)/p^3$ if $p^2\mid n$.
\end{proof}

\vspace{-.085in}
\begin{proposition}\label{propdensuf1}
The density $\mu_{\max,\rm{suff}}(U_n(\Z_p))$ of elements in $U_n(\Z_p)$ that are maximal and
sufficiently ramified~is given by
\begin{equation*}
\mu_{\max,\rm{suff}}(U_n(\Z_p))=\left\{
\begin{array}{cl}
  (p-1)/p^2 &\mbox{if } p^2\nmid n;\\[.035in]
  (p-1)^2/p^3&\mbox{if } p^2\mid n.
\end{array}\right.
\end{equation*}
\end{proposition}

\begin{proof}
Here we use Lemma \ref{lemsuf} and Lemma~\ref{propcondmax}, and also the same notation and representatives~$r$ for
$\P^1(\F_p)$ as in the proof of Proposition~\ref{propdensmax}.
\begin{itemize}
\item If $p\nmid n$, then for finite $r$, we have that $f$ is maximal
  with a multiple root at $r$ precisely when $p\mid c_r$ and
  $p\parallel d_r$, which has probability $1/p\times
  (p-1)/p^2=(p-1)/p^3$; a multiple root at $\infty$ cannot occur in
  this case.

\item If $p\parallel n$, then for finite $r$, we have that $f$ is
  maximal with a multiple root at $r$ precisely when $p\nmid c_r$,
  $p\mid c_r$ and $p\parallel d_r$, which has probability
  $(p-1)/p\times 1/p\times (p-1)/p^2=(p-1)^2/p^4$; maximality and a
  triple root occur at $\infty$ when $p\mid b$, $p\mid c$, and $p\nmid
  d$, which has probability $1/p \times 1/p \times (p-1)/p=(p-1)/p^3$.

\item If $p^2\mid n$, then for finite $r$, we have that $f$ is maximal
  with a multiple root at $r$ (just as when~$p\parallel n$) with
  probability $(p-1)^2/p^4$; maximality with a multiple root at
  $\infty$ cannot~occur.
\end{itemize}
So the probability that $f$ is maximal and sufficiently ramified is:
$p\times (p-1)/p^3 = (p-1)/p^2$ if $p\nmid n$; \;$p\times
(p-1)^2/p^4+(p-1)/p^3=(p-1)/p^2$ if $p\parallel n$; and $p\times
(p-1)^2/p^4=(p-1)^2/p^3$ if $p^2\mid n$.
\end{proof}

Propositions~\ref{propdensmax} and \ref{propdensuf1} therefore imply the following.

\begin{corollary}\label{propdensuf}
For any positive integer $n$ and prime $p$,
the relative density of sufficiently-ramified elements in $U_n(\Z_p)$ among the
maximal elements in $U_n(\Z_p)$ is $$\frac{\mu_{\max,\rm{suff}}(U_n(\Z_p))}{\mu_\max(U_n(\Z_p))}=\frac1{p+1}.$$
\end{corollary}
\vspace{.05in}

\noindent
{\bf Proof of Theorem~\ref{thmain2}:} 
Theorem~\ref{thmain2} follows immediately from Theorem~\ref{main_theorem} 
and
 Corollary \ref{propdensuf}. $\Box$
\pagebreak

\vspace{.1in}
\addcontentsline{toc}{section}{Index of notation}
\begin{center}
{\bf \large Index of notation} 
\vspace{.05in}
\begin{longtable}{|@{\hspace{.06in}}p{1.13in}@{\hspace{.06in}}|@{\hspace{.06in}}p{4.6in}@{\hspace{.06in}}|@{\hspace{.06in}}p{.37in}@{\hspace{.06in}}|}
\hline
{\bf Notation} & {\bf Description} & {\bf In} \\
\hline
\hline
\endhead
\hline
$A_n$
&
The symmetric antidiagonal matrix with antidiagonal $(1/2,-n,1/2)$.
& 
\hyperref[sec:4.12]{C\ref{sec:4.12}}
\\
\hline
$\mathcal A_\infty(\alpha)$
&
The algebra at infinity of the $n$-monogenized cubic ring $(\O,\alpha)$.
& 
\hyperref[sec:4.18]{D\ref{sec:4.18}}
\\
\hline
$\Aut_{K_3}(K_4)$
&
The subgroup of $\Aut(K_4)$ inducing the trivial element of $\Aut(K_3)$.
&
\hyperref[sec:3.36]{N\ref{sec:3.36}}
\\
\hline$(C,\alpha)$ 
&
An $n$-monogenized cubic ring.
& 
\hyperref[sec:2.5]{D\ref{sec:2.5}}
\\
\hline
$\Cl_2(C)$, \!$\Cl_2^+(C)$,\phantom{\!\!\!|\!\!}  $\Cl_2(C)^*$, \!$\Cl_2^+(C)^*$
&
The 2-torsion subgroups of the class group, the narrow class group, and their respective dual groups.
& 
\hyperref[sec:2.14]{N\ref{sec:2.14}}
\\
\hline
$\Delta(f)$
&
$b^2c^2-4ac^3-4b^3d-27a^2d^2+18abcd$, 
 where $f(x,y)=ax^3+bx^2y+cxy^3+dy^3\in U(R)$.
&
\hyperref[sec:3.4]{D\ref{sec:3.4}}
\\
\hline
$d\gamma$
&
A Haar measure on $\SL_3(\R)$.
& 
\hyperref[sec:3.24]{C\ref{sec:3.24}}
\\
\hline\hline
$F(n,X)$&
The set of $n$-monogenized cubic fields $(K,\alpha)$ with $H(K)<X$. & 
\hyperlink{fnx}{\S1.2}
\\
\hline
$F(\cdeltaint,X)$&
The set of $n$-monogenized cubic fields $(K,\alpha)$ such that $n\leq
cH(K,\alpha)^\delta$, $H(K,\alpha)<X$, 
and $(K\otimes\Q_p,\alpha)\in\Sigma_p$ for all primes~$p$.
&
\hyperref[thsubmon]{T\ref{thsubmon}}
\\
\hline
{$F_\Sigma(n,X)$}&
The set of $n$-monogenized cubic fields $(K,\alpha)$ with local
conds.\ of $\Sigma$
and $H(K,\alpha)<X$. &
\hyperlink{fnx}{\S1.2}
\\
\hline
{$F_\Sigma(\cdeltaint,X)$} &
The set of $n$-monogenized cubic fields $(K,\alpha)$ with local conds.\ of $\Sigma$, $\,n\leq cH(K,\alpha)^\delta$, and $H(K,\alpha)<X$. &
\hyperref[sec:3.1]{D\ref{sec:3.1}}
\\
\hline\hline
$\FF_A$
&
A fundamental domain for the action of $\SO_A(\Z)$ on $\SO_A(\R)$.
& 
\hyperref[sec:4.12]{C\ref{sec:4.12}}
\\
\hline
$\FF_{\SL_3}$ 
&
A fundamental domain for the action of $\SL_3(\Z)$ on $\SL_3(\R)$.
& 
\hyperref[sec:3.22]{C\ref{sec:3.22}}
\\
\hline\hline
{$\FF_U^\pm$} & A fundamental domain for the action of $\MB(\Z)$ on $f\in U(\R)$ with $\pm\Delta(f)>0$ and $\ind(f)>0$. 
&
{\hyperref[sec:3.5]{C\ref{sec:3.5}}}
\\
\hline
$\FF_U^\pm(\cdeltaint,X)$ & The set $\{f\in\FF_U^\pm$ with $\ind(f)\leq cH(f)^\delta$ and $H(f)<X\}$. &
\hyperref[sec:3.5]{C\ref{sec:3.5}}
\\
\hline
$\FF_U^\pm(T;X)$ & The set $\{f\in\FF_U^\pm$ with $T\leq\ind(f)<2T$ and $X\leq H(f)<2X\}$.& 
\hyperref[sec:3.8]{N\ref{sec:3.8}}
\\
\hline
$\widetilde{\FF}_U^\pm(T;X)$  &
$\kappa(T;X)$-fold cover of $\FF_U^\pm(T;X)$. &
\hyperref[sec:3.20]{C\ref{sec:3.20}}
\\
\hline\hline
{$\FF_V^{(i)}(\cdeltaint,X)$}& A fundamental set for the action of $\SL_3(\R)$ on the set of elements
in $V(\R)^{(i)}$ with resolvent in $\FF_U^\pm(\cdeltaint,X)$. &
\hyperref[sec:3.22]{C\ref{sec:3.22}}
\\
\hline
{$\FF_V^{(i)}(T;X)$}& A fundamental set for the action of $\SL_3(\R)$ on the set of elements
in $V(\R)^{(i)}$ with resolvent in $\FF_U^\pm(T;X)$ &
\hyperref[sec:3.22]{C\ref{sec:3.22}}
\\
\hline
$\widetilde{\FF}_V^{(i)}(T;X)$ &
$\kappa(T;X)$-fold cover of $\FF_V^{(i)}(T;X)$. &
\hyperref[sec:3.22]{C\ref{sec:3.22}}
\\ 
\hline\hline
$\FF^{(i)}_{V_{A,b}}$& A fundamental set for the action of $\SO_A(\R)$ on $V_{A,b}(\R)^{(i)}$. & 
\hyperref[sec:4.8]{C\ref{sec:4.8}}
\\
\hline

$\RF{i}(X)$
&
The set of elements in $\RF{i}$ having height less than $X$. 
& 
\hyperref[sec:4.8]{C\ref{sec:4.8}}
\\
\hline\hline
$g_A$ & An element in $\SL_3(\Q)$ such that $g_AAg_A^t=A_n$. & \hyperref[sec:4.12]{C\ref{sec:4.12}} \\ 
\hline
$\GG$
&
A genus of good integer-coefficient ternary quadratic forms. 
& 
\hyperref[sec:4.28]{N\ref{sec:4.28}}
\\
\hline
$G(R)$
&
The subgroup $\{(g_2,g_3)\in\GL_2(R)\times\GL_3(R):\det(g_2)\det(g_3)=1\}$. &
\hyperref[sec:2.8]{N\ref{sec:2.8}}
\\
\hline
$(g_2,g_3)\cdot(A,B)$
&
The action of $(g_2,g_3) \in G(R)$ on $(A,B) \in V(R)$.
& 
\hyperref[sec:2.8]{N\ref{sec:2.8}}
\\
\hline
$\gamma \cdot f(x,y)$
&
The twisted action of $\gamma\in \GL_2$ on a binary cubic form $f(x,y)$.
& 
\hyperref[sec:2.3]{N\ref{sec:2.3}}
\\
\hline
{generic}
&
Corresponds to an order in an $S_3$-cubic field or an $S_4$-quartic field. 
& 
\hyperref[sec:3.10]{D\ref{sec:3.10}}
\\
\hline
{good at $p$}
&
The associated conic in $\P^2(\overline{\F}_p)$  is either smooth or a union of two lines.
&
\hyperref[sec:4.21]{D\ref{sec:4.21}}
\\
\hline
$H(\beta)$
&
$\max_v\bigl\{|\beta'|_v\bigr\}$, where $\beta'-\beta\in\Z$ and the absolute trace ${\rm
  Tr}(\beta')\in\{0,1,2\}$.
&
\hyperref[sec:3.30]{N\ref{sec:3.30}}
\\
\hline
$H(f)$
&
$a^{-2}\max\bigl\{|I(f)|^3,J(f)^2/4\}$,  where $f\in U(\R).$ 
&
\hyperref[sec:3.4]{D\ref{sec:3.4}}
\\
\hline
$I(f)$
&
$b^2-3ac$,  where $f(x,y)=ax^3+bx^2y+cxy^3+dy^3\in U(R)$.
& 
\hyperref[sec:3.4]{D\ref{sec:3.4}}
\\
\hline
$\ind(f)$
&
$a$,  where $f(x,y)=ax^3+bx^2y+cxy^3+dy^3\in U(R)$. 
& 
\hyperref[sec:3.4]{D\ref{sec:3.4}}
\\
\hline
$J(f)$
&
$-2b^3+9abc-27a^2d$,  where $f(x,y)=ax^3+bx^2y+cxy^3+dy^3\in U(R)$. 
& 
\hyperref[sec:3.4]{D\ref{sec:3.4}}
\\
\hline
$\kappa=\kappa(\ycox)$
&
$\bigl\lfloor X^{1/6}/\Y^{2/3}\bigr\rfloor$.
& 
\hyperref[sec:3.20]{C\ref{sec:3.20}}
\\
\hline
$\kappa_p := \kappa_p(A)$
&
(For $A$ having $\F_p$-rank two.) $1$ if residually hyperbolic, $-1$ otherwise. 
&
\hyperref[sec:4.21]{D\ref{sec:4.21}}
\\
\hline
$\kappa_p(\GG)$
&
The $\kappa$-invariant $\kappa_p(A)$ for any $A \in \GG$. 
& 
\hyperref[sec:4.28]{N\ref{sec:4.28}}
\\
\hline
$L^\gen$
&
The set of generic elements in $L$.
& 
\hyperref[sec:3.10]{D\ref{sec:3.10}}
\\
\hline
$L^\pm$
&
The set of elements $f\in L$ with $\pm\Delta(f)>0$. 
& 
\hyperref[sec:3.8]{N\ref{sec:3.8}}
\\
\hline
$L^{(i)}$
&
$L\cap V(\R)^{(i)}$.
& 
\hyperref[sec:3.19]{N\ref{sec:3.19}}
\\
\hline
{large}
&
$S := (S_p)_p$ is large if for all but finitely many~$p$,  the set $S_p$ contains all $f\in U(\Z_p)$ (resp. $(A,B) \in V(\Z_p)$) with $p^2\nmid\Delta(f)$ (resp. $p^2\nmid\Delta(A,B)$).
& 
\hyperref[sec:3.1]{D\ref{sec:3.1}}, \hyperref[sec:4.15]{D\ref{sec:4.15}}
\\
\hline
$\MB(R)$
&
The group of $2 \times 2$ lower triangular unipotent matrices over $R$.
& 
\hyperref[sec:2.6]{N\ref{sec:2.6}}
\\
\hline
$\Mass_{p}(f)$
&
The local mass of binary cubic form $f\in U(\Z_p)$ at the prime $p$.
& 
\hyperref[sec:3.38]{D\ref{sec:3.38}}
\\
\hline
{$\Mass_p^\pm(f)$}
&
The local $\kappa$-mass of the binary cubic form $f\in U(\Z_p)$ at the prime $p$.
& 
\hyperref[sec:4.23]{D\ref{sec:4.23}}
\\
\hline
$m_i$
&
$m_0=m_2=m_{2\pm}=m_{2\#}=4;\quad m_1=2.$
& 
\hyperref[sec:3.22]{C\ref{sec:3.22}}
\\
\hline
$N'$
&
A compact subset of unipotent lower triangular $3 \times 3$ matrices over $\R$.
& 
\hyperref[sec:4.12]{C\ref{sec:4.12}}
\\
\hline\hline
{$N_3^\pm(n,X)$} & Number of $n$-monogenized $S_3$-orders $(C,\alpha)$ with $\pm\Delta(C)>0$ and $H(C,\alpha)<X$. &
\hyperref[thcubringcount]{T\ref{thcubringcount}}
\\
\hline
$N_3^\pm(\cdeltaint,X)$ & Number of $n$-monogenized $S_3$-orders $(C,\alpha)$
with $\pm\Delta(C)>0$, \linebreak $n\leq cH(C,\alpha)^\delta$, and $H(C,\alpha)<X$. &
\hyperref[thsubcount]{T\ref{thsubcount}}
\\
\hline
{$N_4^{(i)}(n,X)$} & Number of pairs $(Q,(C,\alpha))$ where $Q$ is an $S_4$-order, $Q\otimes\R\cong \R^{4-2i}\times\C^i$, 
and $(C,\alpha)$ is an $n$-monogenized cubic resolvent of $Q$ with $H(C,\alpha)\!<\!X$.&
\hyperref[thqnmccountnr]{T\ref{thqnmccountnr}}
\\
\hline
{$N_4^{(i)}(\cdeltaint,X)$} & Number of pairs $(Q,(C,\alpha))$ where $Q$ is an $S_4$-order, $Q\otimes\R\cong \R^{4-2i}\times\C^i$, \linebreak 
and $(C,\alpha)$ is an $n$-monogenized cubic resolvent of $Q$ with \linebreak $n\leq cH(C,\alpha)^\delta$ and $H(C,\alpha)<X$. &
\hyperref[thqsubcount]{T\ref{thqsubcount}}
\\
\hline\hline
{$N^\pm(L;n,X)$}
&
The number of generic $\MB(\Z)$-orbits $f\in L^\pm$ with $H(f)<X$. &
\hyperref[sec:4.2]{N\ref{sec:4.2}}
\\
\hline
$N^{(i)}(L;A,X)$& The number of generic $\SO_A(\Z)$-orbits $v\in L\cap V_A(\Z)^{(i)}$ with $H(v)<X$. &
\hyperref[N:4.11]{N\ref{N:4.11}}
\\
\hline
{$N^{(i)}(L;T;X)$} & The number of generic $\MB(\Z)\times\SL_3(\Z)$-orbits $v\in L^{(i)}$ with \linebreak
$T\leq\ind(v)<2T$ and $X\leq H(v)<2X$. &
\hyperref[N:3.25]{N\ref{N:3.25}}
\\
\hline\hline
{nowhere \linebreak overramified} 
& Not overramified at any finite or infinite place.
& 
\hyperref[sec:2.12]{D\ref{sec:2.12}}
\\
\hline
$\nu(L)$
&
The volume of the closure of $L$ in $U(\widehat{\Z})$.
& 
\hyperref[sec:3.8]{N\ref{sec:3.8}}
\\
\hline
{overramified at $p$}
&
The prime ideal $p\mathbb{Z}$ factors as $P^4$, $P^2$, or $P_1^2 P_3^2$.
& 
\hyperref[sec:2.12]{D\ref{sec:2.12}}
\\
\hline
$\mathcal{Q}_n(R)$
&
The set of ternary quadratic forms with coefficients in $R$ and $4\det=n$.
& 
\hyperref[sec:4.5]{N\ref{sec:4.5}}
\\
\hline
{$\RR(K_3)$}
&
The set of \'etale non-overramified quartic extensions of $\Q_p$ with cubic resolvent $K_3$.
& 
\hyperref[sec:4.19]{N\ref{sec:4.19}}
\\
\hline
{$\RR^+(K_3)$, $\RR^-(K_3)$}
&
The elements of $\RR(K_3)$ that are respectively split or inert in the unramified quadratic extension $K_6/K_3$ corresponding to~$K_4$. 
& 
\hyperref[sec:4.19]{N\ref{sec:4.19}}
\\
\hline
$\Res$ & The resolvent map $V(R)\to U(R)$ defined by $(A,B)\to 4\det(Ax-By)$. &
\hyperref[sec:2.8]{N\ref{sec:2.8}}
\\
\hline
{residually hyperbolic at $p$}
&
The associated conic in $\P^2(\overline{\F}_p)$ is a union of two lines defined over $\F_p$.
&
\hyperref[sec:4.21]{D\ref{sec:4.21}}
\\
\hline
{residually nonhyperbolic at $p$}
&
The \!associated \!conic \!in \!$\P^2(\overline{\F}_p)$ \!is \!a \!union \!of \!two \!lines \!not \!defined over~$\F_p$. 
&
\hyperref[sec:4.21]{D\ref{sec:4.21}}
\\
\hline
$S:=(S_p)_p$
&
A collection of local cubic specifications $S_p$ indexed by primes  $p$.
& 
\hyperref[sec:3.1]{D\ref{sec:3.1}}, \hyperref[sec:4.15]{D\ref{sec:4.15}}
\\
\hline
$\sigma_A$
&
The map $V_{A,b}(F) \to V_{A_n,b}(F)$ given by $(A,B)\mapsto (A_n,g_ABg_A^t)$.
& 
\hyperref[sec:4.12]{C\ref{sec:4.12}}
\\
\hline
$\sigma_G$
&
The map $\SO_A(F) \to \SO_{A_n}(F)$ given by $g\mapsto g_Agg_A^{-1}$.
& 
\hyperref[sec:4.12]{C\ref{sec:4.12}}
\\
\hline
$\Sigma$ 
&
$\Sigma:=(\Sigma_p)_p$.
& 
\hyperref[sec:3.1]{D\ref{sec:3.1}}
\\
\hline
$\Sigma_p$
&
A set of pairs $(\mathcal K_p,\alpha_p)$ where $\mathcal K_p$ is an \'etale cubic extension of 
$\Q_p$ with ring of integers $\O_p$, $\alpha_p$ is an element of $\O_p$ that is primitive in $\O_p/\Z_p$, and the pair 
$(\O_p,\alpha_p)$ corresponds to some $f(x,y)\in S_p$. &
\hyperref[sec:3.1]{D\ref{sec:3.1}}
\\
\hline
${\rm sk}(C)$
&
The skewness of the cubic ring $C$.
& 
\hyperref[sec:3.31]{D\ref{sec:3.31}}
\\
\hline
$T'$
&
A subset of the diagonal $3 \times 3$ matrices over $\R$ with $\det=1$.
& 
\hyperref[sec:4.12]{C\ref{sec:4.12}}
\\
\hline
$T_\even$, $T_\odd$
&
The set of primes $p$ dividing $n$ to even or odd order, respectively. 
& 
\hyperref[sec:4.28]{N\ref{sec:4.28}}
\\
\hline\hline
$U(R)$ & The set of binary cubic forms $f(x,y)$ over $R$. &
\hyperref[sec:2.3]{N\ref{sec:2.3}}
\\
\hline
$U_n(R)$ & The set of binary cubic forms $f(x,y)\in U(R)$ with $x^3$-coefficient $n$.&
\hyperref[sec:2.6]{N\ref{sec:2.6}}
\\
\hline
$U_{n,b}(R)$ & The set of binary cubic forms $f(x,y)\in U_n(R)$ with $x^2y$-coefficient~$b$.&
\hyperref[sec:4.2]{N\ref{sec:4.2}}
\\
\hline
$U(\Z)_S$ &
The set of elements in $U(\Z)$ satisfying the local prescriptions of $S$.&
\hyperref[sec:3.26]{N\ref{sec:3.26}}
\\
\hline
$U_{n,b}(\Z)_S$ &
The set of elements in $U_{n,b}(\Z)$ satisfying the local prescriptions of $S$. &
\hyperref[sec:4.15]{D\ref{sec:4.15}}
\\
\hline\hline
$V(R)$ & The set of pairs of ternary quadratic forms $(A,B)$ over $R$. &
\hyperref[sec:2.8]{N\ref{sec:2.8}}
\\
\hline
$V_A(R)$ & The subset of $(A,B)\in V(R)$ with fixed $A$. &
\hyperref[sec:4.5]{N\ref{sec:4.5}}
\\
\hline
$V_{A,b}(R)$ & The subset of $(A,B)\in V_A(R)$ with resolvent in $U_{n,b}(R)$ for some $n$. &
\hyperref[sec:4.5]{N\ref{sec:4.5}}
\\
\hline
$V(\Z)_S$ & The set of elements $v\in V(\Z)$
such that $\Res(v)\in U(\Z)_S$. &
\hyperref[sec:3.26]{N\ref{sec:3.26}}
\\
\hline
$V_{A,b}(\Z)_S$ & The set of elements $v\in V_{A,b}(\Z)$
such that $\Res(v)\in U_{n,b}(\Z)_S$. &
\hyperref[sec:4.15]{D\ref{sec:4.15}}
\\
\hline\hline
$V(\R)_+$
&
The set of elements $(A, B) \in V(\R)$ with $\det(A) > 0$.
& 
\hyperref[sec:3.14]{D\ref{sec:3.14}}
\\
\hline
$V(\Z)_+$
&
$V(\Z)\cap V(\R)_+$.
& 
\hyperref[sec:3.13]{N\ref{sec:3.13}}
\\
\hline
$V(\R)^{(i)}$
&
A $G(\R)$-orbit in $V(\R)_+$ specified by the index $i\in\{0,1,2,2\#,2+,2-\}$.
&
\hyperref[sec:3.14]{D\ref{sec:3.14}}, \hyperref[sec:3.19]{N\ref{sec:3.19}}
\\
\hline
$V(\Z)^{(i)}$
&
$V(\Z)\cap V(\R)^{(i)}$.
&
\hyperref[sec:3.19]{N\ref{sec:3.19}}
\\
\hline\hline
$\W_p(U)$, \,$\W_p(V)$
&
Subsets of $U(\Z)$, \!$V(\Z)$ where $p^2\mid\Delta$.
&
\hyperref[sec:3.28]{N\ref{sec:3.28}}
\\
\hline
$\W_p^{(i)}(U)$, \!\!$\W_p^{(i)}(V)$
&
Subsets of $U(\Z)$, \!$V(\Z)$ where $p^2\mid \Delta$ for 
``mod $p^i$ reasons". 
&
\hyperref[sec:3.28]{N\ref{sec:3.28}}
\\
\hline
\end{longtable}
\end{center}

\end{document}